\documentclass[12pt]{article}
\usepackage[utf8]{inputenc}

\usepackage{amsmath,amsthm,amssymb,amsfonts,amscd,graphicx}
\usepackage{mathtools}
\usepackage{bbm}
\usepackage{comment}
\usepackage{natbib}
\usepackage{xcolor}
\usepackage[colorlinks, citecolor=blue]{hyperref}
\usepackage{framed,tcolorbox}
\usepackage{authblk}

\renewcommand{\exp}[1]{\operatorname{exp}\left(#1\right)} 
\providecommand{\argmax}{\mathop\mathrm{arg max}} 
\providecommand{\argmin}{\mathop\mathrm{arg min}}

\newcommand\var{\mathrm{var}}
\newcommand\cov{\mathrm{cov}}
\newcommand{\eps}{\varepsilon}

\newcommand{\wh}{\widehat}
\newcommand{\wt}{\widetilde}

\newcommand{\rom}[1]{\uppercase\expandafter{\romannumeral #1\relax}}

\newcommand{\bmln}{\begin{multline*}}
\newcommand{\emln}{\end{multline*}}

\newcommand{\be}[1]{\begin{equation*}#1\end{equation*}}
\newcommand{\ben}[1]{\begin{equation}#1\end{equation}}

\newcommand{\ml}[1]{\begin{multline*}{#1}\end{multline*}}

\def\l{\left}
\def\r{\right}
\newcommand{\m}{\mathcal}
\newcommand{\mb}{\mathbb}

\newcommand{\pr}[1]{\mathbb{P}{\left(#1\right)}}
\newcommand{\dotp}[2]{\left\langle#1,#2\right\rangle}

\def\ml#1{\begin{multline*}{#1}\end{multline*}}

\newtheorem{theorem}{Theorem}
\newtheorem{lemma}{Lemma}

\theoremstyle{thmstyletwo}%
\newtheorem{remark}{Remark}%

\newtheorem{corollary}{Corollary}

\theoremstyle{thmstylethree}%

\begin{document}

\title{Deviation inequalities for U-processes of growing order, with applications to half-space depth}

\author[1]{Stanislav Minsker\footnote{S. Minsker acknowledges support by the National Science Foundation grant DMS CAREER-2045068.}}
\author[2]{Shunan Yao}
\affil[1]{Department of Mathematics, University of Southern California, \protect\\
\texttt{minsker@usc.edu}}
\affil[2]{Department of Mathematics, Hong Kong Baptist University, \protect\\
\texttt{yaoshunan@hkbu.edu.hk}}

\date{}
\maketitle
\begin{abstract}
We prove new deviation inequalities for infinite-order U-processes, that is, U-processes whose order may increase with the sample size. Our goal is to obtain the sharpest possible dependence of the tail bounds on the order of the process, as well as to characterize the largest range in which the process exhibits sub-Gaussian deviations.

We derive closed-form inequalities for processes indexed by Euclidean classes, including VC-subgraph classes of functions. We then illustrate the statistical applications of our results to multivariate depth. In particular, we introduce a multivariate analogue of the Hodges–Lehmann estimator based on Tukey’s median, establish sub-Gaussian concentration inequalities with explicit constants for its deviations from the mean, and identify an interesting phenomenon concerning its asymptotic distribution.
\end{abstract}

\mathtoolsset{showonlyrefs = true}

\section{Introduction}
\label{section:intro}

Inequalities for the sums of independent random variables form the backbone of the modern non-asymptotic theory of high-dimensional statistics. Notable examples include exponential concentration inequalities due to Bernstein and Hoeffding as well as moment inequalities due to Rosenthal, Marcinkiewicz, and Zygmund. Of particular importance are the sharp functional extensions of these results for suprema of empirical processes, including McDiarmid's and Talagrand's concentration inequalities, see \citep[][]{boucheron2013concentration} for the modern exposition. However, many problems involve dependence where the aforementioned inequalities are often insufficient. 
U-statistics, introduced by P. Halmos \citep{halmos1946theory} and W. Hoeffding \citep{hoeffding1948class} in the 1940s, constitute one such example. U-statistics have been an active topic of research ever since due to their far-reaching applications \citep{serfling2009approximation,lee2019u}. 

Let $(S,\m S)$ be a measurable space, and let $X\in S$ be a random variable with distribution $P$. Moreover, suppose that $X_1,\ldots,X_N$ are i.i.d. copies of $X$. Given $m<N$, assume that $h(x_1,\ldots,x_m)$ is a permutation-symmetric function of $m$ $S$-valued variables\footnote{We say that $h$ is permutation-symmetric if $h(x_{\pi(1)},\ldots,x_{\pi(m)}) = h(x_1,\ldots,x_m)$ for any $x_1,\ldots,x_m\in \mb R$ and any permutation $\pi:\{1,\ldots,m\}\mapsto \{1,\ldots,m\}$.}. Let $\mathcal{I}^N_m = \l\{ J\subset \{1,\ldots,N\}: \ |J|=m\r\}$. A U-statistic with kernel $h$ is defined as 
\[
U_{N,m}(h) = \frac{1}{{N \choose m}} \sum_{J \in \mathcal{I}^N_m} h\l(X_j,\  j \in J\r)\,,
\]
where for $J=\{i_1,i_2,\ldots,i_m\}$, $h\l(X_j,\  j \in J\r)$ is a shorthand for $h\left(X_{i_1},\ldots,X_{i_m}\right)$. The integer $m\geq 1$ is known as the \emph{order} or \emph{degree} of the U-statistic. The key property of $U_{N,m}(h)$ is that under rather general assumptions, it has minimal variance among all unbiased estimators of $\mb Eh:=\mb Eh(X_1,\ldots,X_m)$ \citep{lee2019u}. 
In this work, we study U-statistics with non-degenerate kernels, meaning that the conditional expectation $h^{(1)}(X):=\mb E \l[ h(X,X_2,\ldots,X_m)|X\r]$ is such that its variance is positive. 
Next, we define U-processes, a functional extension of U-statistics. 
Let $\m H_m$ be a collection of measurable, bounded and permutation-symmetric kernels of $m$ variables. A U-process indexed by $\m H_m$ is a stochastic process defined as a collection of random variables
\[
\l\{U_{N,m}(h), \ h\in \m H_m\r\}.
\]
U-processes generalize U-statistics in the same way as empirical processes generalize sample means. 
Observe that when $J_1\cap J_2\ne \emptyset$, the terms $h\l(X_j,\  j \in J_1\r)$ and $h\l(X_j,\  j \in J_2\r)$ in the definition of a U-statistic are dependent, whence methods designed for dealing with independent summands are not readily applicable. 

Under rather general assumptions \cite[e.g. see][chapter 5]{de2012decoupling}, non-degenerate U-process $\sqrt{N/m}(U_{N,m}(h) - \mb Eh)$ converges weakly to a Gaussian process. In this paper, we are interested in the following closely related but non-asymptotic question: \emph{for which values of $t$ the tail probabilities $\pr{\sup_{h\in \m H_m}\l| U_{N,m}(h) - \mb Eh\r|\geq t}$ are sub-Gaussian?} In other words, we are interested in the regime when 
$\sup_{h\in \m H_m}\l| U_{N,m}(h) - \mb Eh\r|$ behaves like the supremum of a Gaussian process. 
The main assumptions that we make are that 
\begin{enumerate}
\item[(a)] $m$ is an increasing function of the sample size $N$ and
\item[(b)] $\sup_{h\in \m H_m} \|h\|_\infty<\infty$, where $\|\cdot\|_\infty$ is the usual sup-norm.
\end{enumerate}
While condition (b) is standard, the majority of existing literature focuses on the situation when $m$ is fixed, and the case of varying $m$ received less attention. \citet{frees1989infinite} was the first to consider ``infinite-order'' U-statistics and their applications to renewal theory. Recently, such U-statistics have been investigated in relation to performance of Breiman's random forests algorithm (for example, see the papers by \citet{song2019approximating,peng2022rates,mentch2016quantifying}) as well as the median of means estimator \citep{minsker2023u}. 
U-processes of fixed degree have been studied in relation to statistical procedures that generalize classical M-estimators \citep{arcones1993limit,brunel2025asymptotics}. \citet{arcones1995bernstein} was the first to prove Bernstein-type concentration inequalities for U-statistics and U-processes that capture the correct variance parameter $m\,\sup_{h\in \m H_m}\var\l(h^{(1)}(X)\r)$. The proofs were based on the powerful decoupling technique \citep{de2012decoupling} that unfortunately yields highly suboptimal dependence of the constants on the order $m$.  \cite{maurer2019bernstein} used a different technique that yields Bernstein-type bound for U-statistics that is useful for $m$ up to $o(N^{1/3})$.
\cite{sherman1994maximal} introduced a new symmetrization technique that allowed him to bypass decoupling, thus yielding sharper dependence on $m$. These ideas were further developed and extended by \citet{heilig2001limit} allowing $m$ to grow as $m=o(N^{1/3})$ and by \citet{song2019approximating} with $m$ of order $o(N^{1/2})$ in some applications. However, these advances still did not yield sharp concentration inequalities. Recently, \citet{minsker2023u} showed that for a large class of underlying distributions, sub-Gaussian bounds for U-statistics are possible for $m$ as large as $m=o(N)$ and $t = o(N/m)$. 
In this work we will harness advances in the theory of concentration inequalities to prove bounds for U-process of order $m=o\l(N^{1/2}\r)$ indexed by Euclidean classes, formally defined below. 

\subsection{Preview of the main results}
\label{section:prelim}

Let us recall the definition and main properties of Hoeffding's decomposition which is crucial for the analysis of U-statistics. For $j=1,\ldots,m,$ define 
\ben{
\label{eq:pi}
h^{(j)}(x_1,\ldots,x_j):=(\delta_{x_1} - P_Y)\times\ldots\times(\delta_{x_j}-P_Y)\times P_Y^{m-j} h.
}
It is well known that $h^{(j)}$ is an orthogonal projection of $h$ onto a particular subspace of $L_2(P)$ \citep[][section 1.7]{lee2019u}.
The kernels $h^{(j)}$ have the property of \emph{complete degeneracy}, meaning that $\mb E h^{(j)}(x_1,\ldots,x_{j-1},X_j) = 0$ for $P$-almost all $x_1,\ldots,x_{j-1}$ while $h^{(j)}(X_1,\ldots,X_j)$ is non-zero with positive probability. It is immediate that 
\[
h(x_1,\ldots,x_m) = \sum_{j=1}^m \sum_{J\subseteq [m]: |J|=j} h^{(j)}(x_i,\, i\in J),
\]
where $[m]$ is a shorthand for the set $\{1,\ldots,m\}$. 
In particular, the partial sum $\sum_{j=1}^k \sum_{J\subseteq [m]: |J|=j} h^{(j)}(x_i,\, i\in J)$ can be viewed as the best approximation of $h$, in the mean-squared sense, in terms of sums of functions of at most $k$ variables. 
Hoeffding's decomposition \citep{hoeffding1948class} states that 
\begin{align}
\label{eq:h-decomposition}
  U_{N,m}(h) - \mb E U_{N,m}(h) = \frac{m}{N} \sum_{i=1}^N h^{(1)}\l(X_i\r) + \sum_{j=2}^m \frac{{m \choose j}}{{N \choose j}} \sum_{J \in \mathcal{I}^N_j} h^{(j)}\l(X_i,\  i \in J\r).
\end{align}
The first term 
\[
\frac{m}{N} \sum_{j=1}^N h^{(1)}\l(X_j\r)
\]
is known as the \textit{H\'{a}jek's projection}, while the terms corresponding to $j\geq 2$ are U-statistics with completely degenerate kernels $h^{(j)}$, that is, 
\[
U_{N,m}^{(j)}(h):=\frac{1}{{N\choose j}} \sum\limits_{J \in \mathcal{I}^N_j} h^{(j)}(X_i, \ i\in J).
\]
Moreover, all terms in the representation \eqref{eq:h-decomposition} are uncorrelated. It is known \citep{minsker2023u} that 
\[
\var\l(  \frac{{m\choose j} }{{N\choose j}} \sum_{J \in \mathcal{I}^N_j} h^{(j)}(X_i,\, i\in J) \r)\leq \var(h(X_1,\ldots,X_m) \l( \frac{m}{N}\r)^j,
\]
hence in the case of non-degenerate U-statistics, H\'{a}jek's projection is a leading term that controls the behavior of $U_{N,m}(h)$. 
We further expand this claim in two directions: first, by studying U-processes instead of U-statistics and second, by investigating the rate of decay of tail probabilities of the higher-order terms $U_{N,m}^{(j)}(h), \ j\geq 2$ and establishing the range where these terms are negligible \emph{with exponentially high probability}. 
Detailed tail behavior of the higher-order terms $U_{N,m}^{(j)}(h), \ j\geq 2$ in the ``fixed-$m$ regime'' is well understood, although rather complex \citep{adamczak2006moment,gine2000exponential}. Here, we focus on the regime where $m=m(N)$ is an increasing sequence, which presents new challenges in the analysis. 
To this end, we prove new moment inequalities for U-processes with completely degenerate kernels. Specifically, we derive upper bounds on 
\[
\mb E \sup_{h \in \mathcal{H}_m}\left\lvert \frac{{m \choose j}}{{N \choose j}} \sum_{J \in \mathcal{I}^N_j} h^{(j)}\l(X_i,\  i \in J\r)\right\rvert^p
\]
for \(p \ge 2\) and \(j \ge 2\) that are summarized in Theorem \ref{thm: thm2}. Our approach is based on a new version of Bonami's inequality for Rademacher chaos processes inspired by the work of \cite{dirksen2015tail} (see section \ref{sec:hypercontractive}). While sharp inequalities that capture mixed tail behavior exist in the literature (see \citep[Theorem 14]{boucheron2005moment}, \cite{adamczak2006moment}), our result provides simple \emph{computable} upper bounds with better dependence on the order $m$ than, say, \cite[Corollary 5]{boucheron2005moment}. 

If the class $\m H_m$ admits well-behaved complexity estimates (for example, this is the case for VC-subgraph classes \citep{wellner2013weak}), then the bounds of Theorem \ref{thm: thm2} can be combined with the standard tools for the analysis of empirical processes to state an explicit deviation inequality whose informal version looks as follows: for a non-degenerate $U$-process $U_{N,m}(h)$, 
\begin{multline}
\label{eq:U-b}
\sup_{h\in \m H_m} \l| U_{N,m}(h) - \mb E U_{N,m}(h) \r| 
    \leq (1+\eps)\l[\mb E\sup_{h\in \m H_m}\l|\frac{m}{N} \sum_{i=1}^N h^{(1)}\l(X_i\r)\r| \r.
    \\
    + \l. \sqrt{2t \,m\var\l(h^{(1)}(X)\r)}\sqrt{\frac{m}{N}} \r]
\end{multline}
with probability at least \(1-2e^{-t}\) for $t$ smaller than $c\frac{N}{m^2}$. 
Here, $\eps=\eps(N,m)>0$ can be arbitrary small when $N$ is large. 
A complete version of this statement is given in Theorem \ref{thm:Bernstein-bound}. Let us remark that inequality \eqref{eq:U-b} can be viewed as a version of Bousquet's concentration inequality \citep{bousquet2002bennett} for U-processes. 

As an application of our results, we define and analyze an extension of halfspace depth and the associated halfspace median. Our construction is inspired by the ideas behind Tukey's depth \citep{tukey1975mathematics} and the univariate Hodges-Lehmann estimator \citep{hodges1963estimates}. The latter is often used instead of the sample median when efficiency is a concern. It is known \citep{nolan1999min,bai1999asymptotic, masse2002asymptotics} that the asymptotic distribution of Tukey's median is non-normal and is given by a functional of a Brownian bridge process. On the other hand, a version of Tukey's median that we define in section \ref{sec:THL} turns out to be asymptotically normal (see Theorem~\ref{thm:normality}). Moreover, it satisfies non-asymptotic deviation guarantees with explicit constants, stated in Theorem \ref{th:tukey-nonasymp}, that mimics guarantees for the Gaussian sample mean even if the underlying distribution is heavy-tailed.

The remainder of the paper is organized as follows: section \ref{section:notation} introduces notation frequently used throughout the paper. 
Section \ref{sec:bernstein} contains the main results about U-processes. Section \ref{sec:THL} describes examples of statistical applications. 
Proofs of the technical results are given in appendices.

\subsection{Notation}
\label{section:notation}

Throughout this paper, we use $c, c', C, C_1$, etc., to denote universal constants that can take different values in different parts of the paper. We also write $C_a$ to denote a constant that depends only on $a$. Given two positive sequences $a_n$ and $b_n$, we will write $a_n \gg b_n$ if $b_n/a_n \rightarrow 0$ as $n \rightarrow \infty$. Similarly, we write $a_n \ll b_n$ if $b_n \gg a_n$. For any $a, b \in \mb R$, $a \vee b := \max\{a,b\}$ and $a \wedge b := \min\{a,b\}$. The ``positive logarithm'' $\log_+(x)$ stands for $\log(x) \vee 0$ for all $x > 0$. 
Given $x\in \mb R^d$, $\|x\|$ will denote its Euclidean norm. Similarly, for a matrix $A\in \mb R^{d\times p}$, $\|A\|$ will stand for its operator norm.  
Given a random variable $Z$, we will write $\mb E^{1/p} |Z|$ in place of $\l(\mb E|Z|\r)^{1/p}$. 
If $h$ is a function of $m$ variables, we will write $\mb Eh$ and $\var(h)$ instead of $\mb Eh(X_1,\ldots,X_m)$ and $\var(h(X_1,\ldots,X_m)$. The $L_p$ norm of $Z$ is denoted $\|Z\|_p:= \mb E^{1/p}\l[\l|Z\r|^p\r]$. 
Finally, if $f(X,Y)$ is a function of random vectors $X$ and $Y$, then $\mb E_X f(X,Y)$ will stand for the conditional expectation $\mb E[f(X,Y)|Y]$.

Throughout the paper, we will frequently use the fact that sum and maximum are interchangeable: that is, for nonnegative numbers $a,b,c$, $a\leq b+c$ implies that $a\leq 2(b\vee c)$. We will often use this fact without further elaboration. Another elementary consequence of Markov's inequality, which will be applied routinely, concerns the transition between moment and deviation estimates. Namely, if $Z$ is a nonnegative random variable such that for all $p\geq p_0$, $\mb E^{1/p} Z^p \leq f(p)$, then for all $t\geq p_0$, $\pr{Z\geq ef(t)}\leq e^{-t}$. 

\section{Moment and deviation inequalities for $U$-processes}
\label{sec:bernstein}

Let $\l\{X_t, t \in T\r\}$ be a stochastic process. Following a standard convention, we define $\mb E \sup_{t \in T} X_t$ via
\begin{equation}
\label{def:sup}
    \mb E \sup_{t \in T} X_t = \sup\l\{\mb E \sup_{t \in F} X_t, F \subseteq T,  F \text{ is finite}\r\}\,.
\end{equation}
It is well known that this definition of $\mb E \sup_{t \in T} X_t$ coincides with ``usual'' definition of $\mb E \sup_{t \in T}X_t$ if $T$ is separable (see section 2.2 in the book by \citet{talagrand2014upper} for the details). 
The goal of this section is to prove a deviation bound for the supremum of the U-process
\begin{align*}
\m H_m \ni  h \mapsto  U_{N,m}(h) :=  \frac{1}{{N \choose m}}  \sum_{J \in \mathcal{I}_m^N} h\l(X_j, j \in J\r),
\end{align*}
where $\mathcal{H}_m$ is a class of symmetric kernels of order $m$. Throughout this section, we assume that $\sup_{h \in \mathcal{H}_m}\|h\|_\infty = 1$ for all $m$. Let 
\[
V_{N,j}(h) = \frac{{m \choose j}^{1/2}}{{N \choose j}^{1/2}} \sum_{J \in \mathcal{I}^N_j} h^{(j)}\l(X_i, i \in J\r).
\] 
Our first goal is to find upper bounds for the moments $\mb E\sup_{h \in \mathcal{H}_m} \l\lvert V_{N,j}(h) \r\rvert^p$. To this end, let us recall the definition of the generic chaining complexity. Let $(T,d)$ be a pseudometric space and let $\alpha > 0$. The generic chaining complexity 
$\gamma_\alpha(T,d)$ is defined via
\begin{align}
\label{eq:gamma_alpha}
    \gamma_\alpha(T, d) = \inf_\mathcal{T} \sup_{t \in T} \sum_{j=0}^\infty 2^{j/\alpha} d(t, P_j)\,,
\end{align}
where the infimum is taken over all sequences of subsets $\mathcal{T} := \{T_0,T_1, T_2, T_3, \ldots\}$ such that $T_j\subset T_{j+1}$ and $\mathrm{card}(T_0)=1$, $\mathrm{card}(T_j)\leq 2^{2^j}$ for all $j\geq 1$ \citep{talagrand2014upper}. Moreover, for $l \geq 1$, let
\begin{align}
\label{eq:gamma_alpha_l}
    \gamma_{\alpha, l}(T, d) = \inf_\mathcal{T} \sup_{t \in T} \sum_{j \geq l}^\infty 2^{j/\alpha} d(t, T_j)\,.
\end{align}
That is, $\gamma_{\alpha, l}(T, d)$ is the complexity evaluated on the ``tail'' of the chain \citep{dirksen2015tail}. It is easy to see that $\gamma_{\alpha, l}(T, d) \leq \gamma_{\alpha}(T, d)$ for all $l \geq 1$ and that $\gamma_{\alpha_1}(T, d)\leq \gamma_{\alpha_2}(T, d)$ for $\alpha_1\geq \alpha_2$. 
Computable upper bounds for the functional $\gamma_{\alpha}(T, d)$ are often expressed in terms of  the metric entropy defined as follows. Let the \emph{covering number} $N(T,d,\eps)$ be the smallest $M\in \mb N$ such that there exists a subset 
$F\subseteq T$ of cardinality $M$ with the property that for all $z\in T$, $d(z,F)\leq \eps$. The corresponding metric entropy $H(T,d,\eps)$ is given by the natural logarithm of $N(T,d,\eps)$. 
We will often be interested in the functionals $\gamma_\alpha(T,d)$ evaluated with respect to (random) distances $d^{(1)}_{N,j}(\cdot,\cdot)$ and $d^{(2)}_{N,j}(\cdot,\cdot)$ defined on $\m H_m$ via
    \begin{align}\label{eq:d_N}
        d^{(1)}_{N,j}(h_1, h_2) = e \l[\frac{1}{{N \choose j}}\sum_{J \in \mathcal{I}^N_j} \l(\l(h^{(j)}_1 - h^{(j)}_2 \r)\l(X_i, i \in J\r)\r)^2\r]^{1/2},
    \end{align}
    and
    \begin{align}\label{eq:d_Nj}
        d^{(2)}_{N,j}(h_1, h_2) = \l[\frac{1}{\lfloor N/j \rfloor} \sum_{i=1}^{\lfloor N/j \rfloor} \l(\l(h_1^{(j)} - h_2^{(j)}\r)\l(X_{(i-1)j + 1} , \ldots , X_{ij}\r)\r)^2 \r]^{1/2}.
    \end{align}
Finally, everywhere below, we will denote 
\[
\Sigma^2_{\m H_m} = \sup_{h \in \mathcal{H}_m}\var\l(h((X_1,\ldots,X_m)\r) \text{ and } \sigma^2_{\m H_m}=\sup_{h \in \mathcal{H}_m} \var(h^{(1)}(X)).
\]
Let us recall here that $m\sigma^2_{\m H_m}\leq \Sigma^2_{\m H_m}$ \citep[see][]{lee2019u}.
\begin{theorem}
\label{thm: thm2}
   There exists an absolute constant $C>0$ such that for every integer $1\leq j \leq m$ and real number $p \geq 2$,
    \begin{align*}
    \mb E^{1/p} \sup_{h \in \mathcal{H}_m} \lvert V_{N,j}(h) \rvert^p  \leq C^j\max\l\{A_1(p,j), A_2(p,j), A_3(p,j)\r\},  
\end{align*}
where
\begin{align*}
    A_1(p,j) &=  p^{j/2} \Sigma_{\m H_m} \,, 
    \\
    \\
    A_2(p,j) &= p^{j/2}\l(\frac{m}{j}\r)^{j/2}\l(\frac{j}{N}\r)^{1/4} \l( p\sqrt{\frac{j}{N}}\bigvee \mb E \gamma_2 (\mathcal{H}_m, d^{(2)}_{N,j})\r)^{1/2}
    \\
    A_3(p,j) & =  \mb E^{1/p} \l(\gamma_{2/j, \lfloor \log_2 (pj)\rfloor}\l(\mathcal{H}_m, d^{(1)}_{N,j}\r) \r)^p  \,.\\
\end{align*}
\end{theorem}
Let us draw some parallels with moment bounds for suprema of empirical processes \citep[Theorem 12]{boucheron2005moment}: there, we expect to see a ``complexity term'' that controls the size of the first moment as well as the terms that describe the growth of moments with respect to $p$. In our case, the role of the complexity term is played by $A_3(p,j)$ while the moment growth is controlled by $A_1(p,j)$ and $A_2(p,j)$. 
The proof of the theorem relies on new moment inequalities for homogeneous Rademacher chaos processes of arbitrary order that we present in appendix \ref{sec:hypercontractive}.

To make the bounds of Theorem \ref{thm: thm2} useful in applications, we need to find explicit bounds for the complexity terms. We accomplish this task for the so-called Euclidean classes \citep{nolan1987u} that are frequently encountered in statistical applications, although our general results can be applied beyond this case. Assume that $F_{\m H_m}$ is such that $|h(x_1,\ldots,x_m)|\leq F_{\m H_m}(x_1,\ldots,x_m)$ for all $x_1,\ldots,x_m\in S$ (since we make the boundedness assumption, we can set $F_{\m H_m}\equiv 1$). Then $\m H$ is Euclidean with exponent $V$ if there exists $A,C>0$, $V\geq 1$ such that for every measure $Q$, 
\ben{
\label{eq:euclidean}
N\l(\m H_m,\|\cdot\|_{L_2(Q)},\eps\r)\leq C\l( \frac{A\l\| F_{\m H_m} \r\|_{L_2(Q)}}{\eps}\r)^V.
}
The most prominent example of Euclidean classes are the ``VC-subgraph'' classes \citep[see Theorem 2.6.7 in][]{wellner2013weak}. For Euclidean classes, the moments $\mb E \gamma_{2} \l(\mathcal{H}_m, d^{(2)}_{N,j}\r)$ and $ \l(\mb E \l(\gamma_{2/j} (\mathcal{H}_m, d^{(1)}_{N,j}) \r)^p\r)^{1/p}$ admit explicit upper bounds presented in Lemma \ref{lemma:1} below. Define 
\begin{multline}
\label{eq:gamma}
 \Gamma(j,m,N,V) =  \Sigma_{\m H_m} V^{1/2}\l(\frac{j}{m}\r)^{j/2} \log^{1/2} \l(\frac{C_1^j}{\Sigma_{\m H_m}^2}\l(\frac{m}{j}\r)^j\r)
 \\
 \bigvee  V\l(\frac{j}{N}\r)^{1/2} \log\l(\frac{C_1'}{\Sigma_{\m H_m}^2}\l(\frac{m}{j}\r)^j\r)
\end{multline}
and 
\ben{
\label{eq:D(p,j)}
D(p,j):=D(p,j,m,N) = \l[\frac{pj}{N}\bigvee \frac{Vj}{N}\log\l(\frac{C_1'}{\Sigma_{\m H_m}^2}{m \choose j}\r)\bigvee \Sigma^2_{\m H_m}\l(\frac{j}{m}\r)^{j} \r]^{1/2}.
}
\begin{remark}
    The quantity $\Sigma_{\m H_m}$ can be replaced by its upper bound in the expressions for $\Gamma(j,m,N,V)$ and $D(p,j)$. 
\end{remark}
\noindent The following result follows from the combination of Lemma \ref{lemma:1} and Theorem \ref{thm: thm2}. 
\begin{corollary}
\label{corollary:1}
Let $\m H_m$ be a Euclidean class with exponent $V$. Then for any $p \geq 2$ and $1 \leq j \leq m$,   
    \begin{align*}
    \l(\mb E \sup_{h \in \mathcal{H}_m} \lvert V_{N,j}(h) \rvert^p \r)^{1/p}  \leq C^j\max\l\{B_1(p,j), B_2(p,j)\r\} 
\end{align*}
where 
    \begin{align*}
    B_1(p,j) &= p^{j/2}\l(\frac{m}{j}\r)^{j/2} D(p,j)\,,
    \\
    B_2(p,j) &= 
    \l(C_1V \l[j\log\l(\frac{C_2m}{j}\r) \bigvee \log\l( \frac{e}{\Sigma_{\m H_m}}\r)\r]\r)^{j/2} D(p,j).
    \end{align*}
    Moreover, if $p\geq 4V$, then $B_2(p,j)$ can be chosen as
    \[
    B_2(p,j) = ( pj)^{j/2} e^{-pj/V}.
    \]
\end{corollary}
\begin{remark}
\label{remark:small-m}
The terms $A_1(p,j)$ and $A_2(p,j)$ in Theorem \ref{thm: thm2} are incorporated into $B_1(p,j)$ while $B_2(p,j)$ is an upper bound for $A_3(p,j)$. 
In the ``classical'' regime when $m$ is a constant that does not grow with $N$ and $V\log\l(\frac{1}{\Sigma_{\m H_m}}\r)\leq CN$, the expressions simplify significantly. Specifically, the resulting inequalities are
\begin{align*}    
\Gamma(j,m,N,V)& \leq C \Sigma_{\m H_m }V^{1/2}\log^{1/2}\l(\frac{e}{\Sigma_{\m H_m}}\r),
\\
D(p,j)& \leq C\l(\Sigma_{\m H_m } \bigvee \sqrt{\frac{p}{N}}\bigvee \sqrt{\frac{V}{N}\log\l(\frac{e}{\Sigma_{\m H_m}} \r)} \r)
\end{align*}
and
\begin{multline*}
\l(\mb E \sup_{h \in \mathcal{H}_m} \lvert V_{N,j}(h) \rvert^p \r)^{1/p}  
\leq 
C p^{j/2}
\l(\Sigma_{\m H_m}\bigvee \sqrt{\frac{p}{N}}\r)
\\
+C\l(V \log\l(\frac{e}{\Sigma_{\m H_m}}\r)\r)^{j/2}
\l(\Sigma_{\m H_m}\bigvee \sqrt{\frac{V}{N}\log\l(\frac{e}{\Sigma_{\m H_m}} \r)}\r).
\end{multline*}
Note that in the last inequality, the ``complexity term'' that is of order $V^{j/2}$ is decoupled from the $p$-dependent term that controls the growth of moments. 
One can equivalently rewrite the bound in terms of a deviation inequality as follows: for all $t\geq 2$, 
\begin{multline*}
\sup_{h \in \mathcal{H}_m} \lvert V_{N,j}(h) \rvert  
\leq 
C t^{j/2}
\l(\Sigma_{\m H_m}\bigvee \sqrt{\frac{t}{N}}\r)
\\
+C\l(V \log\l(\frac{e}{\Sigma_{\m H_m}}\r)\r)^{j/2}
\l(\Sigma_{\m H_m}\bigvee \sqrt{\frac{V}{N}\log\l(\frac{e}{\Sigma_{\m H_m}} \r)}\r)
\end{multline*}
with probability at least $1-e^{-t}$.
\end{remark}
We will rely on Corollary \ref{corollary:1} in order to establish the sufficient conditions guaranteeing that the ``remainder'' term in Hoeffding's decomposition $\sup_{h\in \m H_m} \l[\sum_{j=2}^m \frac{{m \choose j}}{{N \choose j}} \l| \sum_{J \in \mathcal{I}^N_j} h^{(j)}\l(X_i,\  i \in J\r)\r|\r]$ is negligible compared to the H\'{a}jek's projection given by $\frac{m}{N} \sum_{j=1}^N h^{(1)}\l(X_j\r)$. Specifically, in Theorem \ref{thm:Bernstein-bound} given below, we only require that $\frac{m^2}{N}$ is not too large. 
This is an improvement over the earlier work by \citet{heilig2001limit}, where each term of order $j\geq 2$ is shown to be negligible when $m^2$ is small compared to $N$ but the sum of these terms is only guaranteed to be negligible if $m^3/N$ is small.

\begin{theorem}
\label{thm:Bernstein-bound}
    Let $\mathcal{H}_m$ be a Euclidean class with exponent $V$ consisting of measurable permutation-symmetric functions. There exist absolute constants $c,C>0$ with the following properties. If $\max(V,t)\frac{m^2}{N}\leq c$, then for all $0<\eps\leq 1$
    \begin{multline*}
    \sup_{h\in \m H_m} \l| U_{N,m}(h) - \mb E U_{N,m}(h) \r| 
    \leq (1+\eps)\mb E\sup_{h\in \m H_m}\l|\frac{m}{N} \sum_{i=1}^N h^{(1)}\l(X_i\r)\r| 
    \\
    + \sqrt{2t m\sigma^2_{\m H_m}}\sqrt{\frac{m}{N}} +C\l(\frac{1}{\eps}+\sqrt{\frac{Vm^2\log(m)}{N}}\r)\frac{tm}{N}
    +\frac{CV}{N}\log^{3/2}(m)
\end{multline*}
with probability at least $1-2e^{-t}$.
\end{theorem}

\begin{remark}
\label{remark:explicit}
Let us emphasize the sharp constants in the first two terms on the right-hand side of the display above that are inherited from the concentration inequality for empirical processes due to \citet{bousquet2002bennett}. While Corollary \ref{corollary:1} provides an upper bound for $\mb E\sup_{h\in \m H_m}\l|\frac{m}{N} \sum_{i=1}^N h^{(1)}\l(X_i\r)\r|$, 
it  can be crude in situations where sharp absolute constants are of essence. The variance term $m\sigma^2_{\m H_m}$ is also sharp (as opposed to $\Sigma^2_{\m H_m}$) and is exactly the quantity driving the asymptotic behavior of $U_{N,m}(h)$. The term $\frac{CV}{N}\log^{3/2}(m)$ is an upper bound for the expected supremum of higher-order terms in Hoeffding's decomposition and is typically of smaller order compared to $\mb E\sup_{h\in \m H_m}\l|\frac{m}{N} \sum_{i=1}^N h^{(1)}\l(X_i\r)\r|$. 
Finally, when $m:=m(N)$ is increasing and $\inf_{m\geq 1} m\sigma^2_{\m H_m}>0$, it is easy to see that the term $C\l(\frac{1}{\eps}+\sqrt{\frac{Vm^2\log(m)}{N}}\r)\frac{tm}{N}$ is negligible compared to $ \sqrt{2t m\sigma_{\m H_m^2}}\sqrt{\frac{m}{N}}$ for $t\leq c \frac{N}{m^2}$.
\end{remark}

\section{Statistical applications: Hodges-Lehmann estimator in $\mb R^d$}
\label{sec:THL}

In this section, we apply previously established inequalities for U-processes to the problem of robust location estimation. Our estimator is based on the notion of half-space depth and the associated half-space median, also known as ``Tukey's median'' (see \citep{hodges1955bivariate,tukey1975mathematics,donoho1992breakdown}) as well as the construction by \cite{hodges1963estimates} devised to improve efficiency of the standard sample median and rooted in U-statistics. 

Half-space depth of a point $z\in \mb R^d$ with respect to a probability measure $P$ is defined as  
\begin{equation*}
D_P(z)=\inf_{u\in S^{d-1}} P H(z,u),
\end{equation*}
where $H(z,u)=\{x\in \mathbb{R}^d:\ (x-z)^T u \geq 0\}$  is the half-space passing through $z$ in direction $u$, and $S^{d-1}$ is the unit sphere in $\mb R^d$ with respect to the Euclidean norm. 
The set of points having maximal depth 
\[
R_\ast:=\arg\max_{z\in \mb R^d} D_P(z)
\]
is referred to as the ``median set,'' while the barycenter $\mu_\ast$ of this (convex) set is called Tukey's median. When $d=1$, it coincides with the usual median.  Given an i.i.d. sample $X_1,\ldots,X_N$ from distribution $P$, define the empirical distribution $P_N:=\frac{1}{N}\sum_{j=1}^N \delta_{X_j}$. The associated depth $D_N(z):=D_{P_N}(z)$ is called the empirical halfspace depth, and, similarly, the associated median $\wt\mu_{N,m}$ will be referred to as the \emph{empirical Tukey's median}. 
We refer the reader the excellent overview by \cite{nagy2022halfspace} of the background on half-space depth and its connections to other areas in geometry. 

Robustness properties of Tukey's median are closely related to existence of points with high depth. Generally, maximal depth of a distribution in $\mb R^d$ can be as low as $\frac{1}{d+1}$, but it is much higher for distributions satisfying certain geometric constraints, such as log-concavity or half-space symmetry: in the latter case, the maximal depth equals $\frac{1}{2}$. 
In this section, we assume that the underlying distribution is elliptically symmetric: this class includes families such as Gaussian and multivariate t-distributions. In particular, Tukey's median corresponding to such distributions has maximal possible depth of $\frac{1}{2}$. We provide the definitions and the necessary background on elliptical symmetry in section \ref{sec:elliptical-symmetric-distribution}. 
When the underlying distribution is Gaussian, Tukey's median is an inefficient estimator of the mean already when $d=1$. To address this shortcoming, \cite{hodges1963estimates} proposed an approach that reduces the asymptotic variance at the cost of decreased breakdown point.  
Specifically, assume that $X_1,\ldots,X_N$ are i.i.d. $\mb R$-valued random variables, $2\leq m < N$ and, given $J\subseteq [N]$ of cardinality $|J|=m$, set $\bar X_J:= \frac{1}{m}\sum_{j\in J}X_j$. Define 
\be{
\wh\mu_{N,m} : =\mathrm{median}\l(\bar X_J, \ J\in \m I_{m}^{(N)}\r).
}
where $\l\{\bar X_J, \ J\in \m I_{m}^{(N)}\r\}$ denotes the set of sample averages computed over all possible subsets of $[N]$ of cardinality $m$, and $\mathrm{median}(\cdot)$ is the usual univariate median. 
Classical Hodges-Lehmann estimator corresponds to the case $m=2$. When $m$ is a fixed integer greater than $2$, $\wh\mu_{N,m}$ is known as the \emph{generalized Hodges-Lehmann estimator}. Its asymptotic properties are well known and can be deduced from results by \citet{serfling1984generalized}, among other works. For example, its breakdown point is $1-(1/2)^{1/m}$ and, in case of normally distributed data with variance $\sigma^2$, the asymptotic distribution of $\sqrt{N}(\wh\mu_{N,m}-\mu)$ is centered normal with variance $\Delta_m^2 = m\sigma^2\arctan\l( \frac{1}{\sqrt{m^2-1}}\r)$. It leads to a natural conjecture that when $m$ is a growing function of $N$, the asymptotic variance of $\sqrt{N}(\wh\mu_{N,m}-\mu)$ equals $\sigma^2$. \cite{diciccio2022clt} rigorously confirmed this intuition, and \cite{minsker2023u} extended it beyond the case of normal distributions and established non-asymptotic guarantees for $\wh \mu_{N,m}$. 

Using the fact that the sample median has the highest (univariate) depth, one can equivalently define $\wh\mu_{N,m}$ as follows: let $P_{N,m}$ be the empirical measure corresponding to the set of points $\l\{ \bar X_J, \ J\in \m I_{m}^{(N)}\r\}$. Then 
\[
\wh\mu_{N,m} = \arg\max_{z\in \mb R^d} D_{P_{N,m}}(z),
\]
where we implicitly assume that $\wh\mu_{N,m}$ is a barycenter of the $\arg\max$ set when the latter contains more than one element. This definition extends to the multivariate case verbatim: stated explicitly, 
\begin{align}
\label{eq:tukey}
    \wh\mu_{N,m} \in \arg\max_{z\in \mb R^d}\inf_{u\in S^{d-1}} \frac{1}{{N \choose m}} \sum_{J \in \mathcal{I}^N_m} I\left\{ \dotp{\bar X_J-z}{u}\geq 0\right\}.
\end{align}
Let us mention that other multivariate extensions of Hodges-Lehmann estimator have been studied before: for instance, \cite{chaudhuri1992multivariate} investigated the approach based on spatial quantiles and established asymptotic normality of resulting estimators. 

It is known that the asymptotic distribution of the (usual) Tukey's median $\wt\mu_{N}$ is normal only when $d=1$. When $d\geq2$, it was shown in a sequence of papers by \cite{nolan1999min, bai1999asymptotic, masse2002asymptotics} that  
\begin{equation}
\label{eq:asymp-tukey}
 \sqrt{N}\l( \wt\mu_{N} - \mu \r) \xrightarrow[]{d}  \arg\max_{z\in \mb R^d} \inf_{u\in S^{d-1}} \l[W(u) - \dotp{u}{z} p_u(0)\r]
\end{equation}
where $W(u)$ is a centered Gaussian process with covariance function  
\[
\cov(W(u), W(u')) = \mb E\l(I\{X\in H(0,u)\}-\frac12\r)\l(I\{X\in H(0,u')\}-\frac12\r),
\]
where $X\sim P$ and $p_u(0)$ is the density of the one-dimensional projection $\dotp{X-\mu}{u}$ evaluated at $0$. Somewhat surprisingly, we show in Theorem \ref{thm:normality} that when $m(N)\to\infty$ as $N\to\infty$ at a specified rate, then the Hodges-Lehmann version of Tukey's median satisfies 
\[
\sqrt{N}\l( \wh\mu_{N,m} - \mu \r) \xrightarrow[]{d} \m N(0,\Sigma)
\]
where $\Sigma$ is the covariance matrix of the underlying distribution, so that $\wh\mu_{N,m}$ is asymptotically equivalent to the sample mean. Moreover, as we demonstrate in Theorem \ref{th:tukey-nonasymp}, it also satisfies strong non-asymptotic deviation guarantees. 

Let us mention another useful property of Tukey's median, namely, affine equivariance. For any affine map $z\mapsto T(z) = Mz+b$ such that $M\in \mb R^{d\times d}$ is non-singular and $b\in \mb R^d$, 
\[
\mu_\ast(P\circ T^{-1}) = M \mu_\ast(P) + b,
\]
where $P\circ T^{-1}(A) = P\l( T^{-1}(A)\r)$. For example, if $X$ has mean $\mu$ and the covariance matrix $\Sigma = \mb E(X-\mu)(X-\mu)^T$ that is non-degenerate, then 
\[
\wh\mu_{N,m}(X_1,\ldots,X_N) = \Sigma^{1/2}\wh\mu_{N,m}(Z_1,\ldots,Z_N) + \mu, 
\]
where $Z_j = \Sigma^{-1/2}(X_j - \mu)$ has centered isotropic distribution. 
Due to this fact, everywhere below we will assume that distribution $P$ of $X_1,\ldots,X_n$ belongs to the class $\mathcal{E}(\mu, I_d, F)$ of isotropic elliptically symmetric distributions such that $\mb E\|X_1\|^2<\infty$ and the characteristic generator $\psi$ satisfies $\psi(t^2) = o\l(\frac{1}{t^\alpha}\r)$ as $t \rightarrow \infty$ for some $\alpha > 0$; notable examples include Gaussian distribution as well as multivariate t-distribution with $\nu>2$ degrees of freedom. Variants of the deviation bounds in the case of non-isotropic distributions can be obtained from the combination of results below and Theorem 2.1 in \citep{minsker2024improved}. Since we are interested in the sharp asymptotic values of constants, we also assume in the results of this section that $m=m(N)\to\infty$ and that $d=d(N)$ can also increase at a specified rate. When $m$ is fixed, the bounds remain valid modulo the values of absolute constants. 
\begin{theorem}
\label{th:tukey-nonasymp}
    \textbf{(a)} Assume that $\max(d,t)\frac{m^2}{N}\leq c$ for a sufficiently small absolute constant $c>0$ and that 
    $\mb E\l|\dotp{X}{u}\r|^p<\infty$ for some $p>4$. Moreover, suppose that 
    \[
    \sqrt N\gg d(\log(m)+\log(d)), \quad m\gg \log^2(d) \ \text{ as } N\to\infty,
    \]
    and let $\eps>0$ be arbitrary. Then
    \begin{equation}
        \label{eq:deviation-1}
    \l\|\wh\mu_{N,m} - \mu \r\| \leq (2 + o_\eps(1))\l((1+\eps)\sqrt{\frac{d}{N}} +
    \sqrt{\frac{2t}{N}}\r)
    \end{equation}
    with probability at least $1-4e^{-t}$, where the sequence $o_\eps(1)\to 0$ as $N,m\to\infty$. 
    \\
    \noindent \textbf{(b)} Assuming only that $\mb E\|X\|^2<\infty$ and the condition $\max(d,t)\frac{m^2}{N}\leq c$, we have that
    \begin{equation}
        \label{eq:deviation-2}
    \l\|\wh\mu_{N,m} - \mu \r\| \leq (2 + o_\eps(1)) \l(4(1+\eps)\sqrt{\frac{d}{N}} +
    \sqrt{\frac{2t}{N}} \r),
    \end{equation}
    with probability at least $1-4e^{-t}$. 
\end{theorem}
\noindent The proof of this result is given in section \ref{proof:tukey-nonasymp}. Several comments are in order:
\begin{enumerate}
    \item[(i)] Condition $\sqrt N\gg d(\log(m)+\log(d))$ can be replaced by the weaker requirement that $\mb E\max_{j=1,\ldots,N} \|X_j\|^2\,\frac{\log(md)}{N}\to 0$ as $N\to\infty$.
    \item[(ii)] The $o_\eps(1)$ terms in the statement are generally distribution-dependent, however, dependence is expressed only in terms of the distribution of one-dimensional projection $\dotp{X}{u}$.
    \item[(iii)] 
    When the underlying distribution is isotropic Gaussian, the sample mean $\bar X_N$ is the estimator that satisfies the best possible non-asymptotic deviation guarantees as $t\to\infty$, given by Borell-TIS inequality \citep{borell1975brunn,tsirel2006norms}. Namely, 
    \[
    \|\bar X_N - \mu\|\leq \sqrt{\frac{d}{N}} + \sqrt{\frac{2t}{N}}
    \]
    with probability at least $1-e^{-t}$. 
    Inequality \eqref{eq:deviation-1} differs from this optimal bound by the factor of 2 (asymptotically), despite the fact that the distribution of $X$ is only assumed to have finitely many moments. Whether it is possible, for every $\eps>0$, to reduce the multiplicative factor to $1+o_\eps(1)$ in the range $1\leq t\leq N^{\alpha}$ for some $\alpha>0$, is currently an open question (note that the multivariate Berry-Esseen theorem \citep{gotze1991rate} implies that this holds for the sample mean for $t\ll \sqrt{\log(N)}$ when $d$ is fixed). We refer the reader to the papers by \citet{gupta2024beyond,lee2022optimal} for further results in this direction. 
    We conjecture that when $d$ grows very slowly, the estimator $\wh\mu_{N,m}$ satisfies inequality \eqref{eq:deviation-1} with asymptotic constant $1$ instead of $2$ for $t\leq cN^{1/d}$; additional technical comments related to this question are given in Remark \ref{remark:optimality}.
\end{enumerate}
We conclude this section with the characterization of the asymptotic distribution of $\wh\mu_{N,m}$ when the ambient dimension $d$ is fixed. As we mentioned in the beginning of this section, limiting distribution of usual Tukey's median is non-normal unless $d=1$. It turns out that the Hodges-Lehmann modification $\wh\mu_{N,m}$ regains asymptotic normality when $m\to\infty$. 
\begin{theorem}
\label{thm:normality}
    Assume that $d$ is fixed, $\mb E\|X\|^2<\infty$, and that $\limsup\limits_{m,N \rightarrow \infty}\frac{m^2}{N} < c$ for a sufficiently small constant $c>0$.
Then $\sqrt{N}\l(\wh\mu_{N,m} - \mu\r) \rightarrow \mathcal{N}(0, I_d)$ in distribution as $m,N\to\infty$.
\end{theorem}
Affine equivariance of $\wh\mu_{N,m}$ immediately implies that in the non-isotropic case, the limiting distribution is $\mathcal{N}(0, \Sigma)$ where $\Sigma$ is the covariance of $X$. 
Let us mention that the core argument behind Theorem \ref{thm:normality} consists in establishing the relation
\begin{equation}
    \label{eq:asymp-tukey-2}
 \sqrt{N}\l( \wh\mu_{N,m} - \mu \r) \xrightarrow[]{d}  \arg\max_{z\in \mb R^d} \inf_{u\in S^{d-1}} \l[\dotp{\frac{Z}{\sqrt{2\pi}}}{u} - \frac{1}{\sqrt{2\pi}}\dotp{u}{z} \r]
\end{equation}
where $Z\sim N(0,I_d)$. Comparing this to display \eqref{eq:asymp-tukey} that describes asymptotic distribution of the usual Tukey's median, we see that the Brownian bridge $W(u)$ has been replaced by the scaled isonormal Gaussian process $\frac{1}{\sqrt{2\pi}}\dotp{Z}{u}$. Now, calculation in \eqref{eq:asymp-tukey-2} can be performed explicitly, immediately yielding that the required $\arg\max$ equals $Z$.

\section{Conclusion and discussion}

In this work, we proved new deviation inequalities for $U$-processes of growing order. Along the way, we established a new version of Bonami's inequality for Rademacher chaos processes of arbitrary order and, as a corollary, moment bounds for degenerate U-processes.   
As a case study, we applied these tools to study a version of half-space median defined \`{a} la the univariate Hodges-Lehmann estimator. We proved  sub-Gaussian concentration inequalities for this estimator, with a particular emphasis on the values of absolute constants. 
At the same time, results for the half-space median are currently restricted to the situation where the underlying distribution is elliptically symmetric. It would be interesting to extend the scope of the bounds to a larger family of distributions with well-defined covariance matrices, as well as to consider other estimators in place of Tukey's median in our construction, such as versions of the multivariate median of means estimator \citep{lugosi2019mean}. Detailed results in this direction will be presented elsewhere.

\appendix

\section{Technical background}
\label{appendix:backgrond}

In this section we present the background material and discuss the technical tools used throughout the paper. 

\subsection{Moment and concentration bounds}
\label{section:tech}

\begin{lemma}[Theorem 11 in \citep{boucheron2005moment}]
\label{lemma:sup-rosenthal}
    Let $\mathcal{F}$ be a countable class of non-negative functions defined on some measurable set $\mathcal{X}$. Let $X_1, \ldots, X_N$ denote  a collection of $\mathcal{X}$-valued independent random variables. Let $Z = \sup_{f \in \mathcal{F}}\sum_i f(X_i)$ and
    \begin{align*}
        M = \max_{1 \leq i \leq N} \sup_{f \in \mathcal{F}} f(X_i)\,.
    \end{align*}
    Then, for all $p \geq 2$ and $\theta \in (0, 2)$,
    \begin{align*}
        \|Z\|_p \leq (1 + \theta)\mb E Z + cp\l(1 + \frac{1}{\theta}\r)\|M\|_p\,,
    \end{align*}
    where $c$ is an absolute constant.
\end{lemma}

\begin{lemma}[Theorem 2.6.7 in \citep{wellner2013weak}]
\label{lemma:covering_bound}
Let $\mathcal{F}$ be a VC-subgraph class of functions of VC dimension $V$ and with measurable envelope $F$. Then for any probability measure $Q$ with $\|F\|_{Q,2}>0$,
\begin{equation*}
N\l(\mathcal{F}, \|\cdot\|_{L_2(Q)},\varepsilon\|F\|_{Q,2}\r)\leq \l(\frac{K}{\varepsilon}\r)^{2V}
\end{equation*}
for a universal constant $K$ and $0<\varepsilon<1$.
\end{lemma}

\begin{lemma} 
\label{lemma:talagrand-bousquet} 
\cite[Theorem 12.5]{boucheron2013concentration}
Let $\m F$ be a class of measurable real-valued functions and let $X_1,\ldots,X_N$ be i.i.d. Assume that $\mb{E}f(X_1)=0$ for all $f\in \m F$ and that $\sup_{f\in \m F} |f(X_1)|\leq U$ with probability $1$. Denote $Z=\sup _{f \in \mathcal{F}} \sum_{k=1}^N f\left(X_k\right)$.
Assume that $\sigma_\ast^2\geq N\sup_{f\in \m F} \mb{E}f^2(X_1)$ and set $v = \sigma_\ast^2 + 2U\mb{E}[Z]$. Then for all $t\geq 0$, 
\begin{equation}
\label{eq:bousquet}
    \mb{P}\left(Z \geq \mb{E} Z + \sqrt{2 t v}+\frac{tU}{3}\right) \leq e^{-t}.
\end{equation}
\end{lemma}

\subsection{Moment bounds for the homogeneous Rademacher chaos}
\label{sec:hypercontractive}

Let $(T,d)$ be a pseudometric space and $\alpha$ -- a positive real number. A centered stochastic process $\{X_t, \ t\in T\}$ is called \emph{a process with $\psi_\alpha$ increments} if for all $t, s \in \mathcal{X}$ and $u>0$,
\begin{align}\label{eq:tail_prob}
    \mb P\l(\l|X_{t} - X_{s}\r| \geq u\,d(t, s) \r) \leq 2e^{-u^{\alpha}}\,.
\end{align}
When  $\alpha = 2$, the processes is typically referred to as sub-Gaussian, and when $\alpha = 1$, it is called sub-exponential. For a Gaussian process, a celebrated majorizing measures theorem due to M. Talagrand states that $\mb E \sup_{t \in T} \lvert X_t \rvert$ is equivalent to the generic chaining complexity $\gamma_2(T,d)$. For general $\alpha>0$, \citet{dirksen2015tail} proved upper bounds for the moments $\mb E \sup_{t\in T} \lvert X_t \rvert^p$ for arbitrary $p \geq 1$  using the ``truncated'' chaining complexity $\gamma_{\alpha, l}(T, d)$ defined in display \eqref{eq:gamma_alpha_l}.  
Next, we state a version of Dirksen's bound that we will apply to the homogeneous Rademacher chaos process. The proof generally follows the argument given in \citep{dirksen2015tail} but we state it below for completeness.
\begin{theorem}
\label{thm:generic_bound}
    Let $j\in N$ and assume that $\{X_t, t \in T\}$ is a $\psi_{2/j}$ process. Then for any $t_0 \in T$ and integer $p \geq 2$, \begin{multline}
    \label{eq:general_bound_after_bonami}
        \l(\mb E \sup_{t \in T} \lvert X_t - X_{t_0} \rvert^p \r)^{1/p} \\
        \leq  2^{j}(p-1)^{j/2} \sup_{t \in T}\|X_t - X_{t_0}\|_2+C2^{j}\gamma_{2/j, \l\lfloor \log_2 (pj) \r\rfloor}(T, d)
         \,. 
    \end{multline}
\end{theorem}
A key property of this inequality is the fact that the term involving the chaining complexity does not depend on $p$. 
The main use case for Theorem \ref{thm:generic_bound} in this paper is the  homogeneous Rademacher chaos process defined as follows. Let $j, N$ be positive integers such that $j < N$, and let $\varepsilon_1,\ldots,\varepsilon_N$ be a sequence of i.i.d. Rademacher random variables. Let $T\subset \mb R^{N \choose j}$ and for every $J \in \mathcal{I}^N_j$, denote 
\[
\varepsilon_J := \Pi_{i \in J}\varepsilon_i.
\]
Then 
\[
T\ni t\mapsto X_t = \sum_{J \in \mathcal{I}^N_j}\varepsilon_J t_J
\]
is called the \emph{homogeneous Rademacher chaos} of order $j$. 
In view of Lemma \ref{lemma:tail_prob} (stated in section \ref{section:tech-results}), the process $\{X_t, t\in T\}$ is a $\psi_{2/j}$ process with distance $d$ defined by $d(x, y) = e\|X_t - X_s\|_2$. Moreover, it is easy to check that
\begin{align*}
    \|X_t - X_s\|_2^2 = \sum_{J \in \mathcal{I}^N_j}(t_J - s_J)^2,
\end{align*}
hence Theorem \ref{thm:generic_bound} holds with $d^2(t, s) = e^2\sum_{J \in \mathcal{I}^N_j}(t_J - s_J)^2$.


\begin{proof}[Proof of Theorem \ref{thm:generic_bound}]

    It suffices to consider the case where $\lvert T \rvert$ is finite (see the defintion \eqref{def:sup}). 
    Let $l \geq 1$ be an arbitrary integer and $\{T_{n}\}_{n \geq 1}$ be the optimal admissible sequence, i.e.
    \begin{align}
    \label{eq:admissible_seq}
        \gamma_{2/j, l}(T, d) = \sup_{t \in T} \sum_{n \geq l} 2^{nj/2} d(t, T_{n}).
    \end{align}
    Furthermore, let $\{\pi_n\}_{n \geq 1}$ be a sequence of functions $\pi_n: T \mapsto T_{n}$ defined by $\pi_n(t) = \argmin_{s \in T_{n}} d(t, s)$. 
    Let $t_0\in T$ and recall that $\mb EX_{t_0} =0$. In view of Minkowski's inequalities,
    \begin{multline}
    \label{eq:separate_estimation}
        \l(\mb E \sup_{t\in T} \lvert X_t - X_{t_0}\rvert^p \r)^{1/p} 
        \\
        \leq \l(\mb E \sup_{t\in T} \lvert X_t - X_{\pi_l(t)} \rvert^p \r)^{1/p} + \l(\mb E \sup_{t\in T} \lvert X_{\pi_l(t)} - X_{t_0} \rvert^p \r)^{1/p}\,.
    \end{multline}
    Recalling that $\lvert T_{l} \rvert \leq 2^{2^l}$, we deduce that
    \begin{align*}
        \mb E \sup_{t\in T} \lvert X_{\pi_l(t)} - X_{t_0} \rvert^p \leq \mb E \l[\sum_{x \in T_{l}} \lvert X_t - X_{t_0} \rvert^p\r] = \sum_{x \in T_{l}} \mb E\lvert X_t - X_{t_0} \rvert^p \\
        \leq \lvert T_{l} \rvert \sup_{t\in T} \mb E \lvert X_t - X_{t_0} \rvert^p = 2^{2^l} \sup_{t\in T} \mb E \lvert X_t - X_{t_0} \rvert^p\,.
    \end{align*}
    Moreover, Bonami's inequality \citep[][Theorem 3.2.2]{de2012decoupling} implies that
    \[
    \mb E^{1/p} \lvert X_t - X_{t_0} \rvert^p \leq (p-1)^{j/2} \mb E^{1/2} \lvert X_t - X_{t_0} \rvert^2.
    \]
    Therefore, 
    \begin{align*}
        \l(\mb E \sup_{t\in T} \lvert X_{\pi_l(t)} - X_{t_0} \rvert^p \r)^{1/p} \leq 2^{2^l/p} p^{j/2}\sup_{t\in T}\|X_t - X_{t_0}\|_2.
    \end{align*}
    To estimate the first term in display \eqref{eq:separate_estimation}, consider the telescoping sum
    \begin{align*}
        X_t - X_{\pi_l(t)} = \sum_{n \geq l} \l(X_{\pi_{n+1}(t)} - X_{\pi_n(t)}\r).
    \end{align*}
    In view of inequality \eqref{eq:tail_prob},
    \begin{align*}
        \mb P\l(\lvert X_{\pi_{n+1}(t)} - X_{\pi_n(t)}\rvert > u2^{nj/2}d(\pi_{n+1}(x), \pi_n(t)) \r) \leq 2 e^{-2^n u^{2/j}}\,.
    \end{align*}
    Let $\Omega_{u, p}$ denote the event
    \begin{align*}
        \l\{\text{For any }n > l, t\in T, \l\lvert X_{\pi_n(t)} - X_{\pi_{n-1}(t)} \r\rvert \leq u2^{nj/2}d(\pi_{n}(t), \pi_{n-1}(t))  \r\}\,.
    \end{align*}
    Mimicking the proof of Lemma A.4 in \cite{dirksen2015tail} without setting a specific value for $l$, we easily find that $\mb P \l(\Omega_{u,p}^c\r) \leq C\exp{-2^{l-1}u^{2/j}}$ for $u \geq 2^{j/2}$, where $C \leq 16$ is an absolute constant. On event $\Omega_{u,p}$,
    \begin{multline*}
        \l\lvert \sum_{n > l} X_{\pi_n(t)} - X_{\pi_{n-1}(t)} \r\rvert \leq \sum_{n > l} \l\lvert X_{\pi_n(t)} - X_{\pi_{n-1}(t)} \r\rvert \\
        \leq u \sum_{n > l}2^{nj/2} d(\pi_n(t), \pi_{n-1}(t))\,.
    \end{multline*}
    Therefore, for every $t\in T$
    \begin{multline*}
        \sum_{n > l}2^{nj/2} d(\pi_n(t), \pi_{n-1}(t)) \leq \sum_{n > l} 2^{nj/2}d(t, \pi_n(t)) + \sum_{n > l} 2^{nj/2}d(t, \pi_{n-1}(t)) \\
        \leq (1 + 2^{j/2}) \sum_{n \geq l} 2^{nj/2} d(t, \pi_n(t)) \leq \l(1 + 2^{j/2}\r) \gamma_{2/j, l}(T, d)\,.
    \end{multline*}
    We conclude that for all $u \geq 2^{j/2}$,
    \begin{align*}
        \mb P \l(\sup_{t\in T} \lvert X_t - X_{\pi_l(t)} \rvert > u(1 + 2^{j/2})\gamma_{2/j, l}(T, d) \r) \leq C e^{-2^{l-1}u^{2/j}}
    \end{align*}
    Invoking Lemma \ref{lemma:compute_expectation}, we deduce from the display above that
    \begin{align*}
        \l(\mb E \sup_{t\in T} \l\lvert X_t - X_{\pi_l(t)}\r\rvert^p\r)^{1/p} \leq \l(1 + 2^{j/2}\r) \gamma_{2/j, l}(T, d)\l(\frac{C^{1/p} (jp)^{j/2 + 1/(2p)}}{2^{jl/2 + 1/p}} + 2^{j/2}\r)\,,
    \end{align*}
    which provides the upper bound for the first term on the right-hand side of display \eqref{eq:separate_estimation}. The result follows by taking $l = \l\lfloor \log_2 p + \log_2 j \r\rfloor$.

\end{proof}
\begin{remark}
     While this result doesn't capture mixed tail behavior of $\sup_{t\in T}X_t$ (e.g. see \citep[Theorem 14]{boucheron2005moment}, \cite{adamczak2006moment}), it does provide simple \emph{computable} upper bounds with better dependence on the order $m$, cf. \cite[Corollary 5]{boucheron2005moment}. 
\end{remark}

\subsection{Elliptically symmetric distributions}
\label{sec:elliptical-symmetric-distribution}

A random vector $X \in \mb R^d$ is said to have an \textit{elliptically symmetric distribution} with parameters $\mu$, $\Sigma$, and $F$, denoted $X \sim \mathcal{E}\l(\mu, \Sigma, F\r)$, if
\begin{align*}
    X \overset{d}{=} \mu + RAU\,,
\end{align*}
where $\overset{d}{=}$ denotes equality in distribution, $\mu \in \mb R^d$ is the mean of $X$, $U \in \mb R^d$ has uniform distribution on the unit sphere $\mathcal{S}^{d-1}$, $A \in \mb R^{d \times d}$ is a fixed matrix such that $\Sigma = AA^\top$, and $R \in \mb R$ is a random variable independent of $U$ with distribution function $F$. Note that the choice of $A$ may not be unique. Thus, we allow $A$ to be any matrix satisfying such property. Moreover, under the assumption that $\mb E \|X\|^2 < \infty$, we will always assume that $\Sigma = \cov(X):=\mb E(X-\mu)(X-\mu)^T$. Distribution of the centered vector $X-\mu$ is uniquely defined via its characteristic function
\begin{align*}
    t \mapsto \mb E e^{i\langle t,X-\mu\rangle}=\psi(t^\top \Sigma t), \quad t \in \mb R^d\,,
\end{align*}
where $\psi : \mb R_+ \mapsto \mb R$ is called the \emph{characteristic generator} of $X$.
 For example, multivariate Gaussian distribution $\mathcal{N}(0, \Sigma)$ is elliptically symmetric with $R = \sqrt{Z}$ where $Z$ has $\chi^2$ distribution with $d$ degrees of freedom. In this case, the characteristic generator is $\psi(x) = e^{-x/2}$.

Next, we state several useful properties of random vectors with elliptically symmetric distributions. Let $X \sim \mathcal{E}(0, I_d, F)$ and assume that $u \in \mathcal{S}^{d-1}$. Then the distribution of $\dotp{X}{u}$ is independent of $u$. Indeed, for any two unit vectors $u_1, u_2 \in \mb R^d$, there exists a unitary matrix $Q \in \mb R^{d\times d}$ such that $u_1 = Qu_2$, hence 
\[
\dotp{X}{u_1} \overset{d}{=} R\dotp{U}{u_1} = R\dotp{U}{Qu_2} = R\dotp{Q^\top U}{u_2} \overset{d}{=} \dotp{X}{u_2}.
\]
Moreover, it is easy to conclude that the characteristic function of $\dotp{X}{u}$ satisfies 
\[
\mb E e^{it\dotp{\Sigma^{-1/2}(X-\mu)}{u}} = \mb E e^{i\dotp{\Sigma^{-1/2}(X-\mu)}{tu}} = \psi\l((tu)^\top(tu)\r) = \psi(t^2)
\]
for any $t \in \mb R$. In section \ref{sec:THL}, we are interested in distributions satisfying the following condition: there exists $\alpha > 0$ such that $\psi(t^2) = o\l(\frac{1}{|t|^\alpha}\r)$ as $|t| \rightarrow \infty$. 
Let us present some examples satisfying this condition. The most obvious one is the Gaussian distribution, where $\psi(t^2) = e^{-t^2/2}$. Another example is the multivariate t-distribution with mean $\mu$, covariance matrix $\Sigma$, and number of degrees of freedom $\nu$. For $\nu > 2$, this distribution is often defined as the distribution of $\frac{Y}{\sqrt{Z_\nu / \nu}} + \mu$ where $Y \in \mb R^d$ has Gaussian distribution $\mathcal{N}(0, \frac{\nu}{\nu-2}\Sigma)$ and $Z_\nu \in \mb R$ has $\chi^2$ distribution with $\nu$ degrees of freedom and is independent of $Y$. It is easy to see that multivariate t-distribution is also elliptically symmetric. In this case, simple analysis shows that
\begin{multline*}
    \mb E e^{it\dotp{\Sigma^{-1/2}(X-\mu)}{u}} = \mb E_{Z_\nu} \mb E_Y e^{\frac{it}{\sqrt{Z_\nu / \nu}}\dotp{\Sigma^{-1/2}Y}{u}} = \mb E_{Z_\nu} e^{-C_{\nu}t^2/Z_\nu} \\
    \leq C'_{\nu}\mb E_{Z_\nu} \l(\frac{t^2}{Z_{\nu}}\r)^{-1} \leq  \frac{C_{\nu}''}{t^2}\,,
\end{multline*}
where we used the inequality $e^{-x} \leq x^{-1}$ for all $x > 0$. We conclude that the multivariate t-distribution satisfies the desired condition when $\nu > 2$.

Additional technical results for densities of elliptically symmetric distributions are given in section \ref{section:tech-results}.

\section{Proofs of the main results}
\label{sec:proof}

In this section, we present the key steps of the proofs of the main theorems. To streamline the arguments, we place details behind several technical lemmas in separate subsections.

\subsection{Proof of Theorem \ref{thm: thm2}}

The main idea is to reduce the bounds for degenerate U-statistics to the bounds on Rademacher chaos processes via the symmetrization argument. The key tool here is the symmetrization inequality for degenerate U-processes due to  \citet{sherman1994maximal}; also see \citep[Theorem 5.2]{song2019approximating} for the modern version used below. 
Specifically, this inequality states the following: let $\varepsilon_1,\ldots,\varepsilon_N$ be a sequence of i.i.d. Rademacher random variables. For every $J \in \mathcal{I}^N_j$, set 
\(\varepsilon_J := \Pi_{i \in J}\varepsilon_i.\) Let $h_0 \in \mathcal{H}_m$ be an arbitrary kernel function. Then
    \begin{multline*}
        \mb E \sup_{h \in \mathcal{H}_m} \lvert V_{N,j}(h) \rvert ^ p \leq 2^p\mb E \lvert V_{N,j}(h_0) \rvert ^ p \\
        + 2^{(j+1)p} \mb E_X \mb E_\varepsilon \sup_{h \in \mathcal{H}_m} \l\lvert \frac{{m \choose j}^{1/2}}{{N \choose j}^{1/2}} \sum_{J \in \mathcal{I}^N_j} \varepsilon_J \l(h^{(j)} -h_0^{(j)}\r)\l(X_i, i \in J\r)\r\rvert ^ p\,.
    \end{multline*}
    Notably, this inequality allows one to avoid decoupling techniques that result in suboptimal dependence on the index $j$. It was shown in the proof of Theorem 3.1 in \cite{minsker2023u} that
    \begin{align*}
        \mb E \lvert V_{N,j}(h_0) \rvert ^ p \leq C^{pj}p^{pj/2}\l(\Sigma_{\mathcal{H}_m} + p^{1/2}\l(\frac{m}{j}\r)^{j/2}\l(\frac{j}{N}\r)^{1/2}\r)^p\,.
    \end{align*}
    Next, inequality \eqref{eq:general_bound_after_bonami} of Theorem \ref{thm:generic_bound} combined with Jensen's inequality implies that
    \begin{multline}
    \label{eq: thm2-for-rademacher}
        \mb E_\varepsilon \sup_{h \in \mathcal{H}_m} \l\lvert \frac{{m \choose j}^{1/2}}{{N \choose j}^{1/2}} \sum_{J \in \mathcal{I}^N_j} \varepsilon_J \l(h^{(j)} -h_0^{(j)}\r)\l(X_i, i \in J\r)\r\rvert ^ p 
        \\
        \leq C_1^{pj}(p-1)^{pj/2} \sup_{h \in \mathcal{H}_m} \l(\mb E_\varepsilon \l\lvert \frac{{m \choose j}^{1/2}}{{N \choose j}^{1/2}} \sum_{J \in \mathcal{I}^N_j} \varepsilon_J \l(h^{(j)} -h_0^{(j)}\r)\l(X_i, i \in J\r)\r\rvert ^ 2\r) ^ {p/2} 
        \\
        + C_2^{pj}\l(\gamma_{2/j, l}\l(\mathcal{H}_m, d^{(1)}_{N,j}\r)\r)^p,
    \end{multline}
    where the pseudo-distance $d_{N,j}^{(1)}$ is defined in \eqref{eq:d_N}. In view of the fact that
    \begin{multline*}
        \mb E_\varepsilon \l\lvert \frac{{m \choose j}^{1/2}}{{N \choose j}^{1/2}} \sum_{J \in \mathcal{I}^N_j} \varepsilon_J \l(h^{(j)} -h_0^{(j)}\r)\l(X_i, i \in J\r)\r\rvert ^ 2 \\
        = \frac{{m \choose j}}{{N \choose j}} \sum_{J \in \mathcal{I}^N_j} \l(\l(h^{(j)} -h_0^{(j)}\r)\l(X_i, i \in J\r)\r)^2\,,
    \end{multline*}
    we can further bring the upper bound to the form
    \begin{multline*}
        \mb E \sup_{h \in \mathcal{H}_m} \lvert V_{N,j}(h) \rvert ^ p 
        \\
        \leq C_1^{pj}(p-1)^{pj/2} \mb E \sup_{h \in \mathcal{H}_m} \l(\frac{{m \choose j}}{{N \choose j}} \sum_{J \in \mathcal{I}^N_j} \l(\l(h^{(j)} -h_0^{(j)}\r)\l(X_i, i \in J\r)\r)^2\r)^{p/2} 
        \\
        + C_2^{pj} \mb E \l(\gamma_{2/j, l}\l(\mathcal{H}_m, d_{N,j}^{(1)}\r)\r)^p
    \end{multline*}
    Next, we are going to estimate
    \begin{align*}
        \mb E \sup_{h \in \mathcal{H}_m} \l(\frac{{m \choose j}}{{N \choose j}} \sum_{J \in \mathcal{I}^N_j} \l(\l(h^{(j)} -h_0^{(j)}\r)\l(X_i, i \in J\r)\r)^2\r)^{p/2}\,.
    \end{align*}
    In view of Hoeffding's representation \citep{hoeffding1994probability},
    \begin{align*}
        \frac{1}{{N \choose j}}\sum_{J \in \mathcal{I}^N_j} \l(\l(h^{(j)} -h_0^{(j)}\r)\l(X_i, i \in J\r)\r)^2 = \frac{1}{N!} \sum_\pi w_{\pi}(h)\,,
    \end{align*}
    where the sum is taken over all permutations $\pi: [N] \mapsto [N]$, and
    \begin{align*}
        w_{\pi}(h) = \frac{1}{k}\sum_{i=0}^{k-1} \l(\l(h^{(j)} -h_0^{(j)}\r)\l(X_{\pi(ij+1)}, X_{\pi(ij+2)}, \ldots, X_{\pi(ij+j)}\r)\r)^2 
     \end{align*}
    for $k = \l\lfloor N/j \r\rfloor$. Jensen's inequality implies that
    \begin{multline*}
        \mb E \sup_{h \in \mathcal{H}_m} \l(\frac{{m \choose j}}{{N \choose j}} \sum_{J \in \mathcal{I}^N_j} \l(\l(h^{(j)} -h_0^{(j)}\r)\l(X_i, i \in J\r)\r)^2\r)^{p/2} 
        \\
        \leq \mb E \l[\frac{1}{N!}\sum_{\pi}\sup_{h \in \mathcal{H}_m} \l({m \choose j} w_{\pi}(h)\r)^{p/2}\r] 
        = \mb E \sup_{h \in \mathcal{H}_m} \l(\frac{{m \choose j}}{{\l\lfloor\frac{N}{j}\r\rfloor}}\sum_{i = 1}^{\lfloor N/j\rfloor} W_{i}(h)\r)^{p/2}\,,
    \end{multline*}
    where $W_{i}(h) = \l(\l(h^{(j)} -h_0^{(j)}\r)(X_{(i-1)j + 1}, X_{(i-1)j + 2}, \ldots, X_{ij})\r)^2$. 
    Therefore, we reduced the problem to the task of estimating the moments of suprema of sums of nonnegative random variables. This task is readily handled via Lemma \ref{lemma:sup-rosenthal}, which implies that
    \begin{multline*}
        \mb E \sup_{h \in \mathcal{H}_m} \l(\frac{{m \choose j}}{{\l\lfloor\frac{N}{j}\r\rfloor}}\sum_{i = 1}^{\lfloor N/j\rfloor} W_{i}(h)\r)^{p/2} 
        \\
        \leq {m \choose j}^{p/2}\l[\frac{C_1}{\lfloor N/j\rfloor}\mb E \sup_{h \in \mathcal{H}_m} \sum_{i=1}^{{\lfloor N/j\rfloor}} W_{i}(h) + p\frac{C_2}{\lfloor N/j\rfloor} \l(\mb E \max_i \sup_{h \in \mathcal{H}_m} W^{p}_{h,i}\r)^{1/p}\r]^{p/2}.
    \end{multline*}
    Recall that by the definition \eqref{eq:pi}, the function $h^{(j)}$ is a linear combination of $2^j$ terms of the form $\Pi_{i \in I} \delta_{x_i} P_X^{m - \lvert I \rvert}$ for all choices of $I \subseteq [j]$. Hence, $W_{i}(h) \leq 2^{2j+2}\|h\|^2_\infty$. Since $\|h\|_\infty\leq 1$ by assumption, we conclude that 
    \[
    \max_i \sup_{h \in \mathcal{H}_m} W^{p}_{h,i} \leq  \l(2^{2j+2}\sup_{h \in \mathcal{H}_m} \|h\|^2_\infty\r)^{p} = 4^{p(j+1)}.
    \]
    Estimation of $\mb E \sup_{h \in \mathcal{H}_m} \frac{1}{\lfloor 
 N/j\rfloor}\sum_{i=1}^{\lfloor N/j \rfloor} W_{i}(h)$ relies on the following lemma; its proof is given in section \ref{proof:lemma:E_sup_sum_W}.
    \begin{lemma}
    \label{lemma:E_sup_sum_W}
    Let the metric $d^{(2)}_{N,j}$ be defined via \eqref{eq:d_Nj}. Then
    \begin{align*}
    \mb E \sup_{h \in \mathcal{H}_m} \frac{1}{\lfloor N/j\rfloor}\sum_{i=1}^{\lfloor N/j \rfloor} W_{i}(h) \leq \frac{C_1 2^{2j}}{\lfloor N/j \rfloor^{1/2}} \mb E \gamma_2 (\mathcal{H}_m, d^{(2)}_{N,j}) + \frac{C_2\Sigma^2_{\m H_m}}{{m \choose j}}.
    \end{align*} 
    \end{lemma}
    Combined with the previously established estimates, this lemma yields that 
    \begin{multline}
        \label{eq:h_j_square}
        \mb E \sup_{h \in \mathcal{H}_m} \l(\frac{{m \choose j}}{{N \choose j}} \sum_{J \in \mathcal{I}^N_j} \l(\l(h^{(j)} -h_0^{(j)}\r)\l(X_i, i \in J\r)\r)^2\r)^{p/2} \leq 
        C^{p}_1 \Sigma_{\m H_m}^{p}
        \\
        + C^{pj}_2 \l(\frac{m}{j}\r)^{pj/2}\l(\frac{j}{N}\r)^{p/4} \l( \l(p\sqrt{\frac{j}{N}}\r)^{p/2} + \l(\mb E \gamma_2 (\mathcal{H}_m, d^{(2)}_{N,j})\r)^{p/2}\r).  
    \end{multline}
    Therefore,
    \begin{multline*}
        \mb E \sup_{h \in \mathcal{H}_m} \lvert V_{N,j}(h) \rvert ^ p \leq  
        C_1^{pj} \mb E \l(\gamma_{2/j, l}\l(\mathcal{H}_m, d^{(1)}_N\r)\r)^p 
        \\
        +(C_2 p)^{pj/2} \l[\Sigma^p_{\m H_m} 
        + \l(\frac{m}{j}\r)^{pj/2}\l(\frac{j}{N}\r)^{p/4} \l( \l(p\sqrt{\frac{j}{N}}\r)^{p/2} + \l(\mb E \gamma_2 (\mathcal{H}_m, d^{(2)}_{N,j})\r)^{p/2}\r)\r].
    \end{multline*}
    The equivalent expression in terms of $A_1(p,j),A_2(p,j)$ and $A_3(p,j)$ follows from the display above.

\subsection{Proof of Theorem \ref{thm:Bernstein-bound}}
\label{section:proof:bernstein-bound}

The first ingredient of the proof is the following lemma that gives explicit upper bounds on the moments of random chaining complexities for Euclidean classes. Recall that 
\begin{multline*}
 \Gamma(j,m,N,V) =  \Sigma_{\m H_m} V^{1/2}\l(\frac{j}{m}\r)^{j/2} \log^{1/2} \l(\frac{C_1'}{\Sigma_{\m H_m}^2}{m \choose j}\r)
 \\
 \bigvee  V\l(\frac{j}{N}\r)^{1/2} \log\l(\frac{C_1'}{\Sigma_{\m H_m}^2}{m \choose j}\r)
\end{multline*}
and that  
\[
D(p,j):=D(p,j,m,N) = \l[\frac{pj}{N}\bigvee \frac{Vj}{N}\log\l(\frac{C_1'}{\Sigma_{\m H_m}^2}{m \choose j}\r)\bigvee \Sigma^2_{\m H_m}\l(\frac{j}{m}\r)^{j} \r]^{1/2}.
\]
\begin{lemma}
\label{lemma:1}
    Let $\mathcal{H}_m$ be a Euclidean class with exponent $V$. Then for any $j \in [m]$, 
    \begin{align}\label{eq:E_gamma_2}
    \mb E \gamma_{2} \l(\mathcal{H}_m, d^{(2)}_{N,j}\r) \leq C^j\Gamma(j,m,N,V).
    \end{align}
    Moreover, for any $p \geq 2$ and $2 \leq j \leq m$,  
    \begin{equation}
    \label{eq:truncated_gamma_bound_1}
    \l(\mb E \l(\gamma_{2/j} (\mathcal{H}_m, d^{(1)}_{N,j}) \r)^p\r)^{1/p} 
    \leq \l(C_1V \l[j\log\l(\frac{C_2m}{j}\r) \bigvee \log\l( \frac{e}{\Sigma_{\m H_m}}\r)\r]\r)^{j/2}D(p,j).
    \end{equation}
    Finally, if $p \geq 4V$, then  
    \begin{align}
    \label{eq:truncated_gamma_bound_2}
        \gamma_{2/j, \lfloor \log_2 (pj) \rfloor}(\mathcal{H}_m, d^{(1)}_{N,j}) \leq 
        (C_5 pj)^{j/2} e^{-pj/V}
    \end{align}
    with probability $1$.
\end{lemma}
\noindent The proof is given in appendix \ref{proof:lemma:1}. Note that the estimates of Lemma \ref{lemma:1} remain valid if we replace $\Sigma_{\m H_m}$ by its upper bound $1$. 
Combined with Theorem \ref{thm: thm2}, it implies the bounds stated in Corollary \ref{corollary:1}. Namely, for $j=2,\ldots,m$, 
\[
\l(\mb E \sup_{h \in \mathcal{H}_m} \lvert V_{N,j}(h) \rvert^p \r)^{1/p}\leq C^j\max\l(B_1(p,j),B_2(p,j) \r)
\]
where $B_1(p,j),B_2(p,j)$ are defined as
\begin{align*}
B_1(p,j) &= p^{j/2}\l(\frac{m}{j}\r)^{j/2} D(p,j), 
\\
B_2(p,j) &= (pj)^{j/2} e^{-pj/V}, \ p\geq 4V
\end{align*}
and $B_2(p,j) = \l(C_1V j\log\l(\frac{C_2m}{j}\r) \r)^{j/2} D(p,j)$ if $p<4V$. 
Our first goal is to upper bound 
\[
S_{m,p}:=\mb E^{1/p}\sup_{h\in \m H_m}\l|\sum_{j=2}^m \frac{{m \choose j}}{{N \choose j}} \sum_{J \in \mathcal{I}^N_j} h^{(j)}\l(X_i,\  i \in J\r)\r|^p.
\]
\begin{lemma}
\label{lemma:higher-order}
    Assume that $\frac{(V\vee p)m^2}{N}\leq c$ for a sufficiently small absolute constant $c>0$. Then
    \[
    S_{m,p} \leq C\l( \frac{V}{N}\log^{3/2}(m) + \frac{pm}{N}\sqrt{1+\frac{Vm^2\log(m)}{N}}\r).
    \]
\end{lemma}
\begin{proof}
    In view of Minkowski's inequality, it suffices to find an upper bound for
    \[
    \sum_{j=2}^m \frac{{m \choose j}^{1/2}}{{N\choose j}^{1/2}}\l(\mb E \sup_{h \in \mathcal{H}_m} \lvert V_{N,j}(h) \rvert^p \r)^{1/p}.
    \]
Employing the inequality $\frac{{m \choose j}^{1/2}}{{N\choose j}^{1/2}}\leq \l(\frac{m}{N}\r)^{j/2}$ and Corollary \ref{corollary:1}, we see that 
\[
S_{m,p}\leq \sum_{j=2}^m \l(\frac{Cm}{N}\r)^{j/2}\max\l( B_1(p,j),B_2(p,j)\r)
\]
Note that whenever $\frac{pm^2}{N}\leq c$ for $c>0$ small enough,
\begin{multline}
\label{eq:B_1}
\sum_{j=2}^m \l(\frac{Cm}{N}\r)^{j/2} B_1(p,j)
\leq 
\sum_{j=2}^m \l(\frac{Cmp}{N}\r)^{j/2}
\\
+ \sum_{j=2}^m \l(\frac{Cpm^2}{Nj}\r)^{j/2} \l[\frac{pj}{N}\bigvee \frac{Vj^2}{N}\log\l(\frac{Cm}{j}\r)\r]^{1/2}
\\
\leq C_1\l( \frac{mp}{N} + \frac{pm^2}{N} \l[\frac{p}{N}\bigvee \frac{V\log(m)}{N}\r]^{1/2}\r).
\end{multline}
Next, consider the inequality
\begin{multline*}
    \sum_{j=2}^m \l(\frac{Cm}{N}\r)^{j/2} B_2(p,j) \leq 
    \sum_{j=2}^m \l(\frac{CVj^2 \log\l( \frac{C_1m}{j} \r)}{N}\r)^{j/2}
    \\
    +\sum_{j=2}^m \l(\frac{CVjm\log\l( \frac{C_1m}{j} \r)}{N}\r)^{j/2}\l[\frac{pj}{N}\bigvee \frac{Vj^2}{N}\log\l(\frac{C_1m}{j}\r)\r]^{1/2}.
\end{multline*}
Let $\delta=\l(\frac{CV}{N}\r)^{1/2}$. Then, whenever $\frac{Vm^2}{N}\leq c_1$ for a sufficiently small positive constant $c_1$,
\begin{multline*}    
\sum_{j=2}^m \l(\frac{CVj^2 \log\l( \frac{C_1m}{j} \r)}{N}\r)^{j/2} = \sum_{j=2}^m \delta^j \l(j\log^{1/2}\l(C_1m/j\r)\r)^j 
\\
\leq 4\delta^2 \log\l(C_2 m\r) \l(1 + \sum_{j\geq 3} \l[\delta j \log^{1/2}\l(C_1m/j\r)\r]^{j-2}\frac{j^2}{4}\r)
\\ 
\leq 4\delta^2 \log\l(C_2 m\r)\l( 1+ \sum_{j\geq 3} 2^{-(j-2)}\frac{j^2}{4}\r) = \frac{CV}{N} \log\l(C_2 m\r)
\end{multline*}
where in the last transition we used the elementary inequality $j\log\l(\frac{Cm}{j}\r)\leq Cm$ and the fact that $\delta j \log^{1/2}\l(C_1m/j\r) \leq \l(\frac{CVm^2}{N}\r)^{1/2}\leq \frac{1}{2}$ if $\frac{Vm^2}{N}$ is sufficiently small. 
Similarly, letting $\delta_1 = \l(\frac{CVm}{N} \r)^{1/2}$ and using the fact that $\delta_1 j^{1/2} \log^{1/2}\l(C_1m/j\r)\leq \l( \frac{CVm^2}{N}\r)^{1/2}\leq \frac 12$, we deduce that
\begin{multline*}
    \sum_{j=2}^m \l(\frac{CVjm\log\l( \frac{C_1m}{j} \r)}{N}\r)^{j/2}\sqrt{\frac{pj}{N}}
    =\sqrt{\frac{p}{N}}\sum_{j=2}^m \delta_1^j \l(j^{1/2}\log^{1/2}\l(C_1m/j\r)\r)^j \sqrt{j} 
    \\
    \leq 2\sqrt{\frac pN}\delta_1^2 \log(C_2m) \l(1 + \sum_{j\geq 3} \l[\delta_1 j^{1/2} \log^{1/2}\l(C_1m/j\r)\r]^{j-2}\frac{j^{3/2}}{2}\r)
    \\
    \leq 2\sqrt{\frac pN}\delta_1^2 \log(C_2m)\l( 1 + \sum_{j\geq 3} 2^{-(j-2)}\frac{j^{3/2}}{2}\r) = \frac{C'Vm}{N} \sqrt{\frac{p}{N}}\log(C_2m)
\end{multline*}
and, repeating the argument, 
\begin{multline*}
    \sum_{j=2}^m \l(\frac{CVjm\log\l( \frac{C_1m}{j} \r)}{N}\r)^{j/2}\l[\frac{Vj^2}{N}\log\l(\frac{C_1m}{j}\r)\r]^{1/2} 
    \\
    \leq \frac{C'Vm}{N}\sqrt{\frac{V}{N}}\log^{3/2}(C_2m).
\end{multline*}
Therefore, 
\begin{multline*}
    \sum_{j=2}^m \l(\frac{Cm}{N}\r)^{j/2} B_2(p,j) 
    \leq C\frac{V}{N}\log(C_2m)\l( 1+ m\sqrt{\frac{p}{N}} + m\sqrt{\frac{V}{N}\log(C_2m)}\r).
\end{multline*}
Whenever $p<\frac{3}{2}V\log(C_2m)$, we deduce using the inequality $\frac{Vm^2}{N}<c_1$ that 
\begin{equation}
\label{eq:B_2a}
\sum_{j=2}^m \l(\frac{Cm}{N}\r)^{j/2} B_2(p,j)\leq 
C\frac{V}{N}\log^{3/2}(C_2m).
\end{equation}
When $p\geq \frac{3}{2}V\log(C_2m)$, it is easy to see that $pe^{-2p/V}\leq c\frac{V\log(C_2m)}{m^3}$, hence in this case
\[
\sum_{j=2}^m \l(\frac{Cm}{N}\r)^{j/2} B_2(p,j)
\leq 
\sum_{j=2}^m \l(\frac{CV j \log(C_2 m)m}{Nm^3}\r)^{j/2} \leq C_1 \frac{V}{N}.
\]
Therefore, we conclude that whenever $\max\l( \frac{pm^2}{N},\frac{Vm^2}{N}\r)\leq c_1$, 
\begin{multline}
    \label{eq:S(m,p)}
S_{m,p}\leq C\l( \frac{V}{N}\log^{3/2}(m) +  \frac{pm}{N} + \frac{pm^2}{N} \l[\frac{p}{N}\bigvee \frac{V\log(m)}{N}\r]^{1/2}\r)
\\
\leq C\l( \frac{V}{N}\log^{3/2}(m) + \frac{pm}{N}\sqrt{1+\frac{Vm^2\log(m)}{N}}\r).
\end{multline}
\end{proof}
\noindent An immediate consequence of Lemma \ref{lemma:higher-order} is that whenever $\frac{Vm^2}{N}\leq c_1$, 
\begin{multline}
    \label{eq:concentration}
\sup_{h\in \m H_m}\l|\sum_{j=2}^m \frac{{m \choose j}}{{N \choose j}} \sum_{J \in \mathcal{I}^N_j} h^{(j)}\l(X_i,\  i \in J\r)\r|  
\\
\leq C\l( \frac{V}{N}\log^{3/2}(m) + \frac{tm}{N}\sqrt{1+\frac{Vm^2\log(m)}{N}} \r)
\end{multline}
with probability at least $1-e^{-t}$ for $t\leq c'\frac{N}{m^2}$.   
Next, Hoeffding's decomposition yields that 
\begin{multline*}
\sup_{h\in \m H_m} \l| U_{N,m}(h) - \mb E U_{N,m}(h) \r|\leq \sup_{h\in \m H_m}\l|\frac{m}{N} \sum_{i=1}^N h^{(1)}\l(X_i\r)\r|
\\
+\sup_{h\in \m H_m}\l|\sum_{j=2}^m \frac{{m \choose j}}{{N \choose j}} \sum_{J \in \mathcal{I}^N_j} h^{(j)}\l(X_i,\  i \in J\r)\r|.
\end{multline*}
It remains to estimate supremum of the H\'{a}jek's projection 
\[
Z_{m,N}:=\sup_{h\in \m H_m}\l|\frac{m}{N} \sum_{i=1}^N h^{(1)}\l(X_i\r)\r|.
\]
This task is readily solved by the application of Lemma \ref{lemma:talagrand-bousquet} which yields that with probability at least $1-e^{-t}$ and for any $\eps>0$,
\begin{align*}
    Z_{m,N}
    &\leq \mb E Z_{m,N} + \sqrt{2t\l(\frac{m}{N}m\sigma^2_{\m H_m}+2\frac{m}{N}\mb EZ_{m,N}\r)} + \frac{tm}{3N} 
    \\
    &\leq (1+\eps)\mb EZ_{m,N} + \sqrt{2t m\sigma^2_{\m H_m}}\sqrt{\frac{m}{N}} +\l(\frac{1}{3}+\frac{1}{\eps}\r)\frac{tm}{N}. 
\end{align*}
We conclude that for $t\leq c\frac{N}{m^2}$, with probability at least $1-2e^{-t}$ 
\begin{multline*}
    \sup_{h\in \m H_m} \l| U_{N,m}(h) - \mb E U_{N,m}(h) \r| 
    \\
    \leq (1+\eps)\mb E\sup_{h\in \m H_m}\l|\frac{m}{N} \sum_{i=1}^N h^{(1)}\l(X_i\r)\r| + \sqrt{2t m\sigma^2_{\m H_m}}\sqrt{\frac{m}{N}} +\l(\frac{1}{3}+\frac{1}{\eps}\r)\frac{tm}{N}
    \\
    +C\l( \frac{V}{N}\log^{3/2}(m) + \frac{tm}{N}\sqrt{1+\frac{Vm^2\log(m)}{N}} \r).
\end{multline*}
The claim follows. 

\subsection{Proof of Theorem \ref{th:tukey-nonasymp}}
\label{proof:tukey-nonasymp}

Without loss of generality, we will assume that $\mu = 0$. 
Let us outline the roadmap of the proof. Denote via $\Phi_m$ the distribution of $\frac{1}{\sqrt{m}}\sum_{j=1}^m X_j$. Since $\wh\mu_{N,m}$ maximizes the empirical depth corresponding to the dataset $\l\{ \bar X_J, \ J\in \m I_{m}^{(N)}\r\}$, its ``population'' depth corresponding to the measure $\Phi_m$ must also be large. We quantify this using the tools for U-processes developed in the first half of the paper. Then it only remains to show that points with high depth must be close to $\mu$, the point with largest depth $\frac12$. To this end, denote $D_m(z):=D_{\Phi_m}(z)$ and
\[
 D_{N,m}(z):= \inf_{u\in S^{d-1}} \frac{1}{{N \choose m}} \sum_{J \in \mathcal{I}^N_m} I\left\{ \sqrt{m}\dotp{\bar X_J}{u}\geq \sqrt{m}\dotp{z}{u}\right\}.
\]
Note that $D_m(\cdot)$ is the ``population version'' of $D_{N,m}(\cdot)$. The following decomposition is immediate: 
\be{
D_m(\wh\mu_{N,m}) =  D_{N,m}(\wh\mu_{N,m}) + D_m(\wh\mu_{N,m}) -  D_{N,m}(\wh\mu_{N,m}).
}
Our next goal is to understand the magnitude of $D_{N,m}(\wh\mu_{N,m})$. To this end, observe that since $D_m(0)=\frac12$,
\[
 D_{N,m}(\wh\mu_{N,m})\geq  D_{N,m}(0) \geq \frac12 - \l| D_{N,m}(0)-D_m(0)\r|.
\]
Moreover, 
\begin{multline*}
\l| D_{N,m}(0)-D_m(0)\r| \leq R_{N,m}
\\
:= \sup_{\|u\|=1} \l| \frac{1}{{N\choose m}} \sum_{J\in \m I_{m}^{(N)} }I\left\{ \dotp{\bar X_J}{u}\geq 0\right\} - \Phi_m\l( H(0,u) \r)\r|
\end{multline*}
where, as before, $H(z,u)= \l\{ x\in\mb R^d: \ \dotp{x-z}{u}\geq 0\r\}$; in particular, $\Phi_m\l( H(0,u) \r) = \frac12$. In Lemma \ref{lemma:halfspace} below, we prove an upper bound for $R_{N,m}$ that is of order $\sqrt{\frac{\max(d,t)m}{N}}$, implying that 
$D_{N,m}(\wh\mu_{N,m}) \geq \frac12 - C\sqrt{\frac{\max(d,t)m}{N}}$ with exponentially high probability.
\begin{remark}
\label{remark:optimality}
Under stronger assumptions, we conjecture that when $m\to\infty$, 
\[
\l| \frac12 - D_{N,m}(\wh\mu_{N,m})\r|\ll \sqrt{\frac{\max(d,t)m}{N}}
\]
with high probability, ultimately leading to asymptotically sharp constant factors in \eqref{eq:deviation-1}.  The proof of this result requires new tools unrelated to U-processes and will be explored elsewhere.
\end{remark}
\noindent Next, consider the term $D_m(\wh\mu_{N,m}) -  D_{N,m}(\wh\mu_{N,m})$ and note that
\begin{multline*}
D_m(\wh\mu_{N,m}) -  D_{N,m}(\wh\mu_{N,m}) \leq -\tilde R_{N,m}
\\
:=-\sup_{\|u\|=1,\|z\|\leq \wh\delta\in \mb R^d} \l| \frac{1}{{N\choose m}} \sum_{J\in \m I_{m}^{(N)} }I\left\{ \dotp{\bar X_J-z}{u}\geq 0\right\} - \Phi_m\l( H(z,u)\r)\r|
\end{multline*}
where $\wh\delta = \|\wh \mu_{N,m}\|$. Combining two estimates, we deduce that 
\begin{equation}
\label{eq:depth}
D_m(\wh\mu_{N,m}) \geq \frac12 - 2R_{N,m} - \l|\tilde R_{N,m}-R_{N,m}\r|.
\end{equation}
Next, we want to show that $\l|\tilde R_{N,m}-R_{N,m}\r|$ is negligible compared to $R_{N,m}$, so that it suffices to estimate just the latter. 
To achieve this goal, we need a preliminary upper bound on $\wh \delta$.
\begin{lemma}
    \label{lemma:prelim}
There exist absolute constants $c,C>0$ with the following properties. Assume that $\frac{\max(t,d)m}{N}\leq c$. Then 
\[
\hat \delta \leq C\sqrt{\frac{\max(d,t)}{N}}
\]
with probability at least $1-e^{-t}$.
\end{lemma}
\noindent The proof of this lemma is standard and is given in section \ref{proof:lemma:prelim}.
Next, we state the bound for $\l|\tilde R_{N,m}-R_{N,m}\r|$. 
\begin{lemma}
\label{lemma:cont-modulus}
Assume that $\frac{\max(d,t)m^2}{N}\leq c$. Then 
\[
\l|\tilde R_{N,m}-R_{N,m}\r| 
        \leq C\sqrt{\frac{\max(d,t)m}{N}}\l( \frac{1}{\sqrt m} + \frac{\log(m)}{\sqrt m} \r).
\]
with probability at least $1-2e^{-t}$.
\end{lemma}
\noindent The proof of this result is given in section \ref{section:proof:cont-modulus}. Since $R_{N,m}$ is of order $\sqrt{\frac{\max(d,t)m}{N}}$, as shown below, the term $\l|\tilde R_{N,m}-R_{N,m}\r|$ is negligible compared to it. 

We will need to introduce additional definitions before proceeding. 
In what follows, $F_k(\cdot)$ will denote the distribution function of a random variable $\frac{1}{\sqrt{k}}\sum_{j=1}^k \dotp{X_j}{u}$ (recall that $F_k$ does not depend on $u$ due to elliptical symmetry). In particular, $\mb E F_{k}\l( \dotp{X}{u}\r)=\frac12$ due to symmetry around $0$. 
Finally, let $\phi_k(x)$ be the probability density function corresponding to $F_k$; its existence is guaranteed by Lemma \ref{lemma:uniform-lipschitz}, moreover, $\phi_m(0)=\frac{1}{\sqrt{2\pi}}+o(1)$ as $m\to\infty$ in view of Lemma \ref{lemma:implication-of-uniform-lipschitz}. The following result constitutes the core of the proof. 
\begin{lemma}
\label{lemma:halfspace}
Assume that $\frac{\max(t,d)m^2}{N}\leq c$ for sufficiently small $c>0$ and that $\mb E\l|\dotp{X}{u}\r|^p<\infty$ for some $p>4$. Then for any $\eps>0$, 
\begin{multline*}  
R_{N,m}\leq \frac{ 1 + \alpha_m}{\sqrt{2\pi}}\l[(1+\eps)\sqrt{\frac{dm}{N}} +
\sqrt{\frac{2tm}{N}} \r]
\\
+C\l(\frac{1}{\eps}+\sqrt{\frac{dm^2\log(m)}{N}}\r)\frac{tm}{N}
    +\frac{Cd}{N}\log^{3/2}(m).
\end{multline*}
with probability at least $1-2e^{-t}$.
Here, 
\[
\alpha_m = \l| \phi_m(0) - \frac{1}{\sqrt{2\pi}}\r| + C(p)\max\l(\frac{1}{m^{1/4}},\frac{\sqrt{d}}{N^{1/4}}\r)\log^{1/2}(dm)
\]
where $C(p)$ also depends on the distribution $F_1$ of the random variable $\dotp{X}{u}$. 
\end{lemma}
\begin{remark}
    \textbf{(a)} It is easy to see that $\alpha_m = o(1)$ as $N\to\infty$ whenever $\frac{\sqrt N}{d}\gg \log(m)+\log(d)$ and $m\gg \log^2(d)$. It is also clear that whenever $\frac{\max(t,d) m^2}{N}\leq c$ for $c$ small enough, the last two terms in the display above are negligible compared to the first two. Therefore, in this case we see that for any $\eps>0$, 
    \[
    R_{N,m}\leq \frac{ 1 + o_\eps(1)}{\sqrt{2\pi}}\l[(1+\eps)\sqrt{\frac{dm}{N}} +
    \sqrt{\frac{2tm}{N}} \r]
    \]
    with probability at least $1-2e^{-t}$, where $o_\eps(1)\to 0$ as $N,m(N)\to\infty$, for every $\eps>0$. \\
    \noindent \textbf{(b)} It follows from Remarks \ref{remark:expectation-simple} and \ref{remark:variance-simple} below that without assuming existence of moments of order greater than $4$ and requiring only that $\frac{\max(t,d)m^2}{N}$ is sufficiently small, the previous display is replaced by the inequality
    \[ 
    R_{N,m}\leq \frac{1 + o_\eps(1)}{\sqrt{2\pi}}\l[4(1+\eps)\sqrt{\frac{dm}{N}} +
    \sqrt{\frac{2tm}{N}} \r]
    \]
that holds with probability at least $1-2e^{-t}$ (here, $o_\eps(1)$ is a different, possibly slower converging sequence, compared to the previous statement). In this case, we pay the price of an additional factor $4$ in the first term proportional to $\sqrt{\frac{dm}{N}}$, while the term controlling the deviations remains optimal. 
\end{remark}
\begin{proof}
    Will apply Theorem \ref{thm:Bernstein-bound} and estimate the quantities involved in the bound. The first object of interest is the expected supremum
    \[
    \mb E\sup_{h\in \m H_m}\l|\frac{m}{N} \sum_{i=1}^N h^{(1)}\l(X_i\r)\r|. 
    \]
    Here, $\m H_m = \l\{ h_u(\cdot), \ \|u\|=1\r\}$ and $h_u(x_1,\ldots,x_m) = I\l\{\frac{1}{m}\sum_{j=1}^m x_j\in H(0,u) \r\}$. Then 
    \begin{multline*}
    h^{(1)}_u(X) = \mb E\l[ h_u(X,X_2,\ldots,X_m) | X\r]
    \\
    = F_{m-1}\l( \dotp{\frac{X}{\sqrt{m-1}}}{u}\r) - \mb E F_{m-1}\l( \dotp{\frac{X}{\sqrt{m-1}}}{u}\r).
    \end{multline*}
    Next, we want to show that for large $m$, $h^{(1)}_u(X) \approx \phi_{m-1}(0)\dotp{\frac{X}{\sqrt m}}{u}$. Modulo the necessary technicalities, it would immediately imply that 
    \[
    \mb E\sup_{\|u\|=1} \l|\frac{m}{N} \sum_{j=1}^N h_u^{(1)}\l(X_j\r)\r| \approx \frac{1}{\sqrt{2\pi}}\sqrt{\frac{dm}{N}}.
    \]
    Let $Z_j(u) = \dotp{\frac{X_j}{\sqrt{m-1}}}{u}$ and observe that 
    \[
    h_u^{(1)}(X_j) = \phi_m(0) Z_j(u) + \int_0^{Z_j(u)} (\phi_m(t) - \phi_m(0))dt,
    \]
    whence 
    \begin{multline*}
    \mb E\sup_{\|u\|=1}\l|\frac{m}{N} \sum_{i=1}^N h_u^{(1)}\l(X_i\r)\r| \leq 
    \phi_m(0)\mb E\sup_{\|u\|=1} \frac{m}{N}\l| \sum_{j=1}^N Z_j(u)\r|
    \\
    + \underbrace{\mb E\sup_{\|u\|=1} \frac{m}{N}\l| \sum_{j=1}^N \int_0^{Z_j(u)} (\phi_m(t) - \phi_m(0))dt\r|}_{=:G(m,N)}.
    \end{multline*}
    We immediately get that 
    \begin{multline*}
    \mb E\sup_{\|u\|=1} \frac{m}{N}\l| \sum_{j=1}^N Z_j(u)\r| = 
    \frac{m}{N}\mb E \l\| \sum_{j=1}^N Z_j(u)\r\| 
    \\
    \leq \frac{m}{N}\mb E^{1/2} \l\| \sum_{j=1}^N Z_j(u)\r\|^2 
    = \sqrt{\frac{md}{N}}.
    \end{multline*}
On the other hand, 
\begin{equation}
    \label{eq:z-decomposition}
Z_j(u) = \underbrace{\dotp{\frac{X_j}{\sqrt{m}}}{u}\psi\l(\l|\frac{\dotp{X_j}{u}}{m^{1/4}}\r| \r)}_{Z_j^L(u)} + \underbrace{\dotp{\frac{X_j}{\sqrt{m}}}{u}\l( 1-\psi\l(\l|\frac{\dotp{X_j}{u}}{m^{1/4}}\r| \r)\r)}_{Z_j^U(u)},
\end{equation}
where 
\[
\psi(z) = \begin{cases}
    1, & z\leq 1, \\
    0, & z\geq 3/2, \\
    1-2(z-1), & z\in \l(1,\frac32\r).
\end{cases}
\] 
is the continuous approximation of the indicator function $I\{z\leq 1\}$; note that the Lipschitz constant of $\psi$ equals $2$ and that $I\{z\leq 1\}\leq \psi(z) \leq I\l\{z\leq \frac32\r\}$. Therefore, 
\begin{multline*}
G(m,N)\leq \mb E\sup_{\|u\|=1} \frac{m}{N}\l| \sum_{j=1}^N \int_0^{Z^L_j(u)} (\phi_m(t) - \phi_m(0))dt\r| 
\\
+ \mb E\sup_{\|u\|=1} \frac{m}{N}\l| \sum_{j=1}^N \int_0^{Z^U_j(u)} (\phi_m(t) - \phi_m(0))dt\r|.
\end{multline*}
It remains to estimate both terms on the right-hand side above. Let $L_m$ be the Lipschitz constant of $\phi_m$; recall that there exists $L>0$ such that $L_m\leq L$ for all $m$ large enough in view of Lemma \ref{lemma:implication-of-uniform-lipschitz}. Note that 
\begin{equation}
\label{eq:lipschitz}
    \l|\int_{z_1}^{z_2} (\phi_m(t)-\phi(0))dt\r|\leq L\l|\int_{z_1}^{z_2} t dt \r|\leq \frac{L}{2}(|z_1|+|z_2|)\l| z_1-z_2\r|,
\end{equation}
implying that the function
$z\mapsto \int_0^z (\phi_m(t)-\phi(0))dt$ is Lipschitz continuous with Lipschitz constant $L m^{-1/4}$ whenever $|z_1|,|z_2|$ do not exceed $m^{-1/4}$. Next, we apply the symmetrization and Talagrand's contraction inequalities to deduce that 
\begin{equation*}
    \mb E\sup_{\|u\|=1} \frac{m}{N}\l| \sum_{j=1}^N \int_0^{Z^L_j(u)} (\phi_m(t) - \phi_m(0))dt\r|
    \leq \frac{4L}{m^{1/4}}\mb E\sup_{\|u\|=1} \frac{m}{N}\l| \sum_{j=1}^N \eps_j Z_j^L(u) \r|.
\end{equation*}
Applying Talagrand's contraction inequality again, this time to the function $z\mapsto z \psi(z)$ that has Lipschitz constant bounded by 3, we see that 
\begin{multline}
\label{eq:Omega_m}
    \mb E\sup_{\|u\|=1} \frac{m}{N}\l| \sum_{j=1}^N \int_0^{Z^L_j(u)} (\phi_m(t) - \phi_m(0))dt\r| 
    \\
    \leq \frac{24L}{m^{1/4}}\mb E\sup_{\|u\|=1} \frac{m}{N}\l| \sum_{j=1}^N Z_j(u)\r| \leq \frac{24L}{m^{1/4}}\sqrt{\frac{md}{N}},
\end{multline}
which is $o\l(\sqrt{\frac{md}{N}}\r)$ as $m,N\to\infty$. 
Finally, we will show that under conditions of the lemma,
\begin{equation}
    \label{eq:B}
B:=\mb E\sup_{\|u\|=1} \frac{m}{N}\l| \sum_{j=1}^N \int_0^{Z^U_j(u)} (\phi_m(t) - \phi_m(0))dt\r|
\end{equation}
is also much smaller than $\sqrt{\frac{md}{N}}$ for large $m$. To this end, observe that for all $z_1<z_2$,  $\int_{z_1}^{z_2}\l| \phi_m(t) - \phi_m(0)\r| dt\leq 2\|\phi_m\|_\infty \l|z_2 - z_1\r|$, hence applying the symmetrization and contraction inequalities as before, we deduce that 
\[
B\leq 8\|\phi_m\|_\infty \mb E\sup_{\|u\|=1}\frac{m}{N}\l| \sum_{j=1}^N Z_j^U(u)\r|.
\]
We are going to estimate the right-hand side of the display above using Dudley's entropy integral bound. The challenge here is that the process inside the expectation is unbounded, and estimation of the empirical diameter of the class 
\[
\m H'=\{u\mapsto f_u(\cdot), \|u\|=1\}, \quad 
f_u(z) =  \dotp{\frac{z}{\sqrt{m-1}}}{u}\l(1-\psi\l(\l| \frac{\dotp{z}{u}}{m^{1/4}}\r|\r)\r)
\]
is quite delicate. First, observe that
\begin{multline*}
\sup_{\|u\|,\|v\|=1}\mb E (f_u(X) - f_v(X))^2
\leq 
2\sup_{\|u\|=1}\mb E|Z^U(u)|^2 
\\
\leq \frac{2}{m-1}\mb E \l[\l|\dotp{X}{u}\r|^2 I\l\{ |\dotp{X}{u}|\geq m^{1/4}\r\}\r]
\\
\leq \frac{2}{m-1}\underbrace{\frac{\mb E^{1/2}\dotp{X}{u}^4}{\sqrt m}}_{:=\Delta_m^2}
:=\Omega^2_m
\end{multline*}
where the term $\mb E\dotp{X}{u}^4$ does not depend on $u$ due to properties of elliptically symmetric distributions. 
Similarly, using the fact that $z\mapsto z(1-\psi(z))$ is Lipschitz continuous as well as the Cauchy-Schwarz inequality, we deduce that 
\begin{align*}
   \l(f_u(X) - f_v(X) \r)^2 &=\l| Z^U(u) - Z^U(v)\r|^2
   \\
    & \leq 9\l| Z(v) - Z(u)\r|^2 
    \leq \frac{9}{m-1}\|X\|^2\|u-v\|^2.
\end{align*}
It implies that the covering number for the class $\m H'$ with respect to the empirical $L_2$ distance $d_N^2(f_u,f_v):=\frac{1}{N}\sum_{j=1}^N \l(f_u(X_j) - f_v(X_j) \r)^2$ satisfies 
\[
N\l(\m H',\eps \r)\leq \l( \frac{C\|X\|_N}{\sqrt{m}\eps}\r)^d,
\]
where $\|X\|_N^2:=\frac{1}{N}\sum_{j=1}^N \|X_j\|^2$ and $C>0$ is an absolute constant. 
Finally, let 
\[
\Omega_{N,m}:=\sup_{u,v} d_N(f_u,f_v)
\]
be the empirical diameter of $\m H'$. We proceed to estimate $B$ via 
Dudley's entropy integral, yielding that
\begin{multline*}
    B\leq Cm \|\phi_m\|_\infty\sqrt{\frac{d}{N}}\mb E \l[\int_0^{\Omega_{N,m}} \log^{1/2}\l( \frac{C\|X\|_N}{\sqrt{m}\eps}\r)d\eps\r]
    \\
    = C\|\phi_m\|_\infty \sqrt{\frac{md}{N}}\mb E\l[ \|X\|_N \int_0^{\tilde\Omega_{N,m}}\log^{1/2}\l(\frac{1}{\eps}\r)d\eps\r]
\end{multline*}
where we set $\tilde\Omega_{N,m}:=C'\frac{\Omega_{N,m}\sqrt{m}}{\|X\|_N}$. 
In view of Cauchy-Schwarz inequality, for any $0<z<1$, $\int_0^z\log^{1/2}\l(1/\eps\r)d\eps \leq z \log^{1/2}(e/z)$, therefore 
\[
B\leq C\|\phi_m\|_\infty\sqrt{\frac{md}{N}} \mb E\l[  \|X\|_N \tilde\Omega_{N,m}\log^{1/2}\l(e/\tilde \Omega_{N,m}\r)\r].
\]
Consider the event $\m E=\l\{\tilde\Omega_{N,m}\leq C' \frac{\Omega_m \sqrt{m}}{ \sqrt{d}}\r\}$. Since the function $z\mapsto z\log^{1/2}(1/z)$ is increasing for $0<z\leq e^{-1/2}$ and $\mb E\|X\|_N\leq \sqrt{d}$, 
\begin{equation}
    \label{eq:int-1}
\mb E\l[ \|X\|_N \tilde\Omega_{N,m}\log^{1/2}\l(e/\tilde \Omega_{N,m}\r) I\{\m E\}\r]
\leq C_1 \Omega_m \sqrt{m}
\log^{1/2}\l( \frac{C_2\sqrt{d}}{\Omega_m \sqrt{m}} \r).
\end{equation}
On the other hand, 
\begin{multline*}
\mb E\l[ \|X\|_N \tilde\Omega_{N,m}\log^{1/2}\l(e/\tilde \Omega_m\r) I\{\m E^c\}\r]
\leq C\sqrt{m}\log^{1/2}\l( \frac{C_1 \sqrt{d}}{\Omega_m \sqrt{m}} \r)\mb E^{1/2} \Omega^2_{N,m}.  
\end{multline*}
It remains to estimate $\mb E  \Omega^2_{N,m}$. To this end, we will rely on results by \citet[][Theorem 3]{abdalla2022covariance}, see also \cite[Theorem 3.1]{jirak2025concentration}. Specifically, these bounds imply that whenever $\frac{d}{n}<c$ for $c>0$ small enough and $\mb E\l| \dotp{X}{u}\r|^p<\infty$ for some $p>4$, with probability at least $1-\frac{C(p)}{N}$ 
\begin{equation}
\label{eq:f_k}
   \sup_{\|u\|=1}\frac{1}{N}\sum_{j\in J} \dotp{X_j}{u}^2 \leq 
    C(p)\left(\frac{\max_{j=1,\ldots,N} \|X_j\|_2^2}{N} + \left(\frac{l}{N}\right)^{p/(4+p)} \log^4 \frac{N}{l}\right)
\end{equation}
uniformly for all $J\subseteq \{1,\ldots,N\}$ such that $|J|\leq l$, where $d\leq l\leq cN$ for a sufficiently small absolute constant $c>0$. 
To apply this result in our situation, note that 
\begin{equation}
    \label{eq:omega}
\Omega^2_{N,m} \leq 4\sup_{\|u\|=1}\frac{1}{(m-1)N}\sum_{j=1}^N \dotp{X_j}{u}^2 I\l\{ |\dotp{X_j}{u}|\geq m^{1/4}\r\},
\end{equation}
hence the upper bound for $\Omega^2_{N,m}$ follows from display \eqref{eq:f_k}: indeed, let $\hat u$ be a unit vector for which the supremum on the right-hand side of display \eqref{eq:omega} is attained. Then take 
\[
\hat J:=\{ j\in\{1,\ldots,N\}: \ |\dotp{X_j}{\hat u}|\geq m^{1/4} \}
\]
and $l=\max(d,|\hat J|)$. It only remains to check that $|\hat J|\leq cN$ with high probability. 
\begin{lemma}
\label{lemma:hat-J}
    Define the random set $J(u):=\{ j\in\{1,\ldots,N\}: \ |\dotp{X_j}{u}|\geq m^{1/4} \}$ and assume that $N\geq cd$ and that $m=m(N)$ is increasing. Then there exist absolute constants $c_1,c_2$ such that for a sufficiently large $N$, 
    \[
    \sup_{\|u\|=1}|J(u)|\leq c_1N
    \]
    with probability at least $1-e^{-c_2N}$.
\end{lemma}
\noindent The proof is similar to the argument in \citep[Theorem 3.1]{jirak2025concentration}. We provide the details in section \ref{proof:hat-J}. We deduce from Lemma \ref{lemma:hat-J} and display \eqref{eq:f_k} that 
\begin{multline*}
    \sqrt{m} \frac{|\hat J|}{N} \leq \sup_{\|u\|=1}\frac{1}{N}\sum_{j=1}^N \dotp{X_j}{u}^2 I\l\{ |\dotp{X_j}{u}|\geq m^{1/4}\r\} 
    \\
    \leq 
C(p)\left(\frac{\max_{j=1,\ldots,N} \|X_j\|_2^2}{N} + \sqrt{\frac{|\hat J|}{N}}\right), 
\end{multline*}
where we also used the fact that $p>4$, whence $\left(\frac{l}{N}\right)^{p/(4+p)} \log^4 \frac{N}{l}\leq C(p) \sqrt{\frac{l}{N}}$. 
Solutions to the inequality $x \leq a\sqrt{x} + b$ satisfy $x\leq 2\max(a^2, b)$, implying that on event of probability at least $1-\frac{C(p)}{N}$,
\[
|\hat J| \leq C_1(p)\max\l(\frac{\max_{j=1,\ldots,N} \|X_j\|_2^2}{\sqrt m}, \frac{N}{m} \r), 
\]
hence on the same event 
\[
\Omega^2_{N,m} \leq \frac{C_2(p)}{m}\max\l(\frac{\max_{j=1,\ldots,N} \|X_j\|_2^2}{N}, \frac{1}{\sqrt m} \r).
\]
It remains to note that if $|\hat J|<d$, then taking $l=d$ in \eqref{eq:f_k} implies that 
\[
\Omega^2_{N,m}\leq \frac{C_2(p)}{m}\max\l(\frac{\max_{j=1,\ldots,N} \|X_j\|_2^2}{N}, \sqrt{\frac{d}{N}}\r).
\]
To summarize, we have shown that on event $\m E$ of probability at least $1-\frac{C(p)}{N}$, 
\begin{equation}
    \label{eq:omega-f}
\Omega^2_{N,m}\leq \frac{C_2(p)}{m}\max\l(\frac{\max_{j=1,\ldots,N} \|X_j\|_2^2}{N}, \frac{1}{\sqrt m} ,\sqrt{\frac{d}{N}}\r).
\end{equation}
We also have the following trivial bound that will be used on the complement $\m E^c$: 
\[
\Omega^2_{N,m}\leq \frac{4}{(m-1)N}\sum_{j=1}^N \|X_j\|^2 I\{\|X_j\|\geq m^{1/4}\}.
\]
Combined with display \eqref{eq:omega-f}, it yields that 
\begin{multline*}
   m \mb E\Omega^2_{N,m} \leq C(p)\l[\max\l( \frac{d}{N},\frac{1}{\sqrt m},\frac{\mb E \max_{j=1,\ldots,N} \|X_j\|_2^2}{N}\r) + 
    \mb E\|X\|^2 I\{\m E^c\}\r].
\end{multline*}
Since 
\[
\l(\mb E \max_{j=1,\ldots,N} \|X_j\|_2^2\r)^2 \leq \mb E\sum_{j=1}^N \|X_j\|^4 \leq C N d^2 
\]
whenever $\mb E \l| \dotp{X}{u}\r|^4<\infty$, and
\[
\mb E\|X\|^2 I\{\m E^c\}\leq \sqrt{\mb E\|X\|^4\pr{\m E^c}} \leq C_1(p)\frac{d}{\sqrt N}
\]
we conclude that 
\begin{equation}
    \label{eq:omega-final}
\mb E \Omega^2_{N,m}\leq \frac{C(p)}{m}\max\l( \frac{1}{\sqrt m},\frac{d}{\sqrt N}\r).
\end{equation}
The previous display implies that the quantity $B$ introduced in \eqref{eq:B} admits an upper bound of the form
\begin{equation*}
    B \leq C(p)\|\phi_m\|_\infty\sqrt{\frac{md}{N}} \l( \max\l(\Delta_m,\frac{\sqrt{d}}{N^{1/4}}\r)
\log^{1/2}\l( \frac{C_2\sqrt{d}}{\Delta_m}\r) 
 \r) 
\end{equation*}
where $\Delta_m = \Omega_m\sqrt{m} = \frac{\sqrt{2}\mb E^{1/4}\dotp{X}{u}^4}{m^{1/4}}$. Recalling that $\sup_{m\geq m_0}\|\phi_m\|_\infty<\infty$ for sufficiently large $m_0$ (see Lemma \ref{lemma:uniform-lipschitz}), we conclude that $B\ll \sqrt{\frac{md}{N}}$ whenever $\frac{\sqrt N}{d}\gg \log(m)+\log(d)$ and $m\gg \log^2(d)$. 

\noindent In summary, we have shown that 
\begin{equation}
    \label{eq:expection}
    \mb E\sup_{h\in \m H_m}\l|\frac{m}{N} \sum_{i=1}^N h^{(1)}\l(X_i\r)\r| \leq 
    \sqrt{\frac{dm}{N}}\l( \frac{1}{\sqrt{2\pi}} + \alpha_m\r), 
\end{equation}
where $\alpha_m = \l| \phi_m(0) - \frac{1}{\sqrt{2\pi}}\r| + C(p)\max\l(\frac{1}{m^{1/4}},\frac{\sqrt{d}}{N^{1/4}}\r)\log^{1/2}(dm)$ and $C(p)$ also depends on the distribution $F_1$ of the random variable $\dotp{X}{u}$. 
\begin{remark}
        \label{remark:expectation-simple}
        Since $\|\phi_m\|_\infty \to \frac{1}{\sqrt{2\pi}}$ as $m\to\infty$  in view of Lemma \ref{lemma:implication-of-uniform-lipschitz}, the function $\psi(z)=\int_0^z \phi_m(t)dt$ is a contraction. Therefore, consecutive application of symmetrization and Talagrand's contraction inequalities yields that 
        \begin{multline*}
        \mb E\sup_{\|u\|=1}\l|\frac{m}{N} \sum_{i=1}^N h_u^{(1)}\l(X_i\r)\r| \leq 4 \|\phi_m\|_\infty \mb E\sup_{\|u\|=1}\l| \sum_{j=1}^N Z_j(u)\r| 
        \\
        \leq 4 \|\phi_m\|_\infty \sqrt{\frac{md}{N}} = 4 \l(\frac{1}{\sqrt{2\pi}} +o(1)\r) \sqrt{\frac{md}{N}}.
        \end{multline*}
    As a result, we get an order-correct estimate with suboptimal constants without assuming existence of 4th order moments or requiring that $d\ll \sqrt{N}$. 
    \end{remark}

Next, we will estimate $\sup_{h \in \mathcal{H}_m} \var(h^{(1)}(X))$. Recall that  $Z_1(u) = \dotp{\frac{X_1}{\sqrt{m-1}}}{u}$ and that 
    \[
    h_u^{(1)}(X_1) = \phi_m(0) Z_1(u) + \int_0^{Z_1(u)} (\phi_m(t) - \phi_m(0))dt,
    \]
Moreover, it follows from \eqref{eq:lipschitz} that $\int_0^{Z_1(u)} (\phi_m(t) - \phi_m(0))dt \leq \frac{L}{2} Z^2_1(u)$, so that 
\begin{multline*}
    \mb E \l(h_u^{(1)}(X_1)\r)^2 
    \\
    \leq \frac{\phi_m^2(0)\mb E\dotp{X_1}{u}^2}{m-1} + \frac{L^2}{4(m-1)^2}\mb E\dotp{X_1}{u}^4 + \frac{L\phi_m(0)}{(m-1)^{3/2}}\mb E\l|\dotp{X_1}{u} \r|^3 
    \\
    \leq\frac{1}{2\pi (m-1)} \l( 1 + C\l(\l|\phi_m(0)-\frac{1}{\sqrt{2\pi}}\r| + \frac{1}{\sqrt m}\r)\r) \leq \frac{1}{2\pi m} \l( 1 + \alpha_m\r)
\end{multline*} 
\begin{remark}
\label{remark:variance-simple}
    A slightly longer argument based on the decomposition \eqref{eq:z-decomposition} implies that $\mb E \l(h_u^{(1)}(X_1)\r)^2 \leq \frac{1}{2\pi m} \l( 1 + o(1)\r)$ where $o(1)\to 0$ as $m\to\infty$ assuming only existence of second moments.
\end{remark}
Finally, Theorem \ref{thm:Bernstein-bound} applied with $V=d$ (indeed, the set of indicator functions of half-spaces passing through the origin is Euclidean for $V=d$) yields that 
\begin{multline*}  
R_{N,m}\leq (1+\eps)\sqrt{\frac{dm}{N}}\l( \frac{1}{\sqrt{2\pi}} + \alpha_m\r) +
\sqrt{\frac{2tm}{N\sqrt{2\pi}}}\l( 1 + \alpha_m\r)
\\
+C\l(\frac{1}{\eps}+\sqrt{\frac{dm^2\log(m)}{N}}\r)\frac{tm}{N}
    +\frac{Cd}{N}\log^{3/2}(m).
\end{multline*}
\end{proof}
Finally, combining inequality \eqref{eq:depth} with the bounds of Lemma \ref{lemma:cont-modulus} and Lemma \ref{lemma:halfspace}, we deduce that for any $\eps>0$,
\begin{equation}
    \label{eq:depth-final}
D_m(\wh\mu_{N,m})\geq \frac12 - \frac{2(1+o_\eps(1))}{\sqrt{2\pi}}\l[(1+\eps)\sqrt{\frac{dm}{N}} +
    \sqrt{\frac{2tm}{N}} \r]
\end{equation}
with probability at least $1-4e^{-t}$. Let us remark that when $m^2=o(N)$ and $t\ll \frac{N}{m^2}$, this probability can be improved to $1-(1+o(1))e^{-t}$ where $o(1)\to 0$ as $N,m\to\infty$. This can be achieved by applying Lemmas \ref{lemma:cont-modulus} and \ref{lemma:halfspace} with $t'=t + s_{N,m}$ for some sequence $s_{N,m}\to\infty$. 

To complete the proof, it remains connect the depth estimate given in display \eqref{eq:depth-final} to the distance between $\wh\mu_{N,m}$ and $\mu$. 
\begin{lemma}
\label{lemma:depth}
    Assume that $z\in \mb R^d$ is such that $D_m(z)\geq \frac12 - \delta$ where $\delta=\delta(\Phi_m)$. Then
    \[
    \|z-\mu\|\leq \frac{\l(\sqrt{2\pi}+o(1)\r)\delta }{\sqrt m}.
    \]
    Here, $o(1)$ is a sequence that converges to $0$ as $m\to\infty$ and $\delta\to 0$.
\end{lemma}
\begin{proof}
Let $u$ be a unit vector and recall that $\phi_m(x)$ is the probability density function of the one-dimensional projection $\frac{1}{\sqrt{m}}\sum_{j=1}^m \dotp{X_j}{u}$ (as we noted before, it does not depend on $u$ for elliptically symmetric distributions). Denote 
\[
\Gamma_m=\sup_{x\in \mb R}\l|\phi_m(x)-\frac{1}{\sqrt{2\pi}}e^{-x^2/2}\r|.
\]
It follows from Lemmas \ref{lemma:uniform-lipschitz} and \ref{lemma:implication-of-uniform-lipschitz} that $\Gamma_m\to 0$ as $m\to\infty$ and that $\phi_m$ is Lipschitz continuous uniformly over all sufficiently large $m$. 
We deduce that for any $u$ such that $\|u\| = 1$ and any $\eps>0$, there exists $\alpha(\eps)>0$ such that 
\ml{
D_m(z) = \inf_{\|u\|=1}\pr{\dotp{\frac{1}{\sqrt{m}}\sum_{j=1}^m X_j}{u} \geq \sqrt{m}\dotp{z}{u}} 
\\
= \frac{1}{2} - \sup_{\|u\|=1}\int_0^{\sqrt{m}\dotp{z}{u}}\phi_m(z) dz
\leq \frac12 - \l(\phi_m(0) - \eps\r)\sqrt{m}\|z\|_2
\\
\leq \frac12 - \l(\frac{1}{\sqrt{2\pi}} - 2\eps\r)\sqrt{m}\|z\|_2
}
whenever $\sqrt{m}\|z\|_2\leq \alpha(\eps)$. Since it also holds that $D_m(z)\geq \frac12 - \delta$, we obtain that
\[
\|z\|\leq \l(\frac{1}{\sqrt{2\pi}} + 2\eps\r)^{-1}\frac{\delta}{\sqrt m}.
\]
The conclusion follows.
\end{proof}
\noindent It remains to note that in our case, Lemma \ref{lemma:depth} applies with 
\[
\delta=\frac{2(1+o_\eps(1))}{\sqrt{2\pi}}\l[(1+\eps)\sqrt{\frac{dm}{N}} +
    \sqrt{\frac{2tm}{N}} \r].
\]
Since $m\to\infty$ and $\max(d,t)\frac{m^2}{N}$ is bounded, $\delta\to 0$ as $m\to\infty$. Claim (a) of the theorem follows, while modifications outlined in Remarks \ref{remark:expectation-simple} and \ref{remark:variance-simple} yield claim (b).


\subsection{Proof of Theorem \ref{thm:normality}}


As we mentioned in section \ref{sec:THL}, due to affine equivariance of Tukey's median, it suffices to assume that $\mu=\mb E X_1 = 0$ and that $\Sigma = I_d$. Observe that the following representation holds:
\begin{align}\label{eq:u-z-process}
    \sqrt{m}\wh \mu_{N,m} = \argmax_{z \in \mb R^d} \inf_{u \in \mathcal{S}^{d-1}} \l[U_{N,m}\l(u,z\r) + \pr{\dotp{\bar X_m  - z}{u}\geq 0}\r]\,,
\end{align}
where
\begin{multline*}
    U_{N,m}(u,z) = \frac{1}{{N\choose m}} \sum_{J\in \m A_{N}^{(m)} }I\left\{ \dotp{\bar X_J - \frac{z}{\sqrt{m}}}{u}\geq 0\right\} \\
    - \pr{\dotp{\bar X_m  - \frac{z}{\sqrt{m}}}{u}\geq 0} \,.
\end{multline*}
Let us also recall that $\bar X_J = \frac{1}{m}\sum_{j \in J}X_j$ and that $\bar X_m = \frac{1}{m}\sum_{j=1}^m X_j$.

Before we present the details of the proof, we briefly map out the strategy. We will show that the process on the right hand side of display \eqref{eq:u-z-process}, after proper scaling, converges weakly to the process
\begin{align*}
    (u,z)\mapsto\dotp{\frac{W}{\sqrt{2\pi}}}{u} - \frac{1}{\sqrt{2\pi}}\dotp{z}{u}, \ \|u\|=1, \ \|z\|\leq R,
\end{align*}
where $W \sim \mathcal{N}(0,I_d)$ and $R>0$. 
Then the ``argmax theorem'' (see Theorem 3.2.2 in \cite{wellner2013weak}) would allow us to conclude that 
\begin{align*}
    \sqrt{N}\l( \wh\mu_{N,m} - \mu \r) \xrightarrow[]{d}  \arg\max_{z\in \mb R^d} \inf_{u\in S^{d-1}} \l[\dotp{\frac{W}{\sqrt{2\pi}}}{u} - \frac{1}{\sqrt{2\pi}}\dotp{u}{z} \r] = W\,.
\end{align*}
The core technicality of this procedure relies on the following lemma.
\begin{lemma}\label{thm:stochastic_convergence}
Let $\mathcal{S}^d_q = \{x \in \mb R^d: \|x\| \leq q\}$ and $\mathcal{S}^{d-1} = \{x \in \mb R^d : \|x\| = 1\}$ be two sets, and denote $\bar{\mathcal{S}} = \mathcal{S}^{d-1} \times \mathcal{S}^d_q$. Under the assumptions of the theorem, the process 
\[
(u,z) \mapsto \sqrt{\frac{N}{m}} U_{N,m}\l(u,z\r)
\]
defined on $\bar{\mathcal{S}}$ converges weakly to $\frac{1}{\sqrt{2\pi}}e^{-\frac{\l\lvert \dotp{z}{u}\r\rvert^2}{2}}\dotp{W}{u}$ where $W$ is a $d$-dimensional standard Gaussian vector.
\end{lemma}
\noindent The proof of this lemma is presented in Section \ref{proof:thm:stochastic_convergence}. Now, let us outline the rest of the proof. The first step is to restrict the random vector $\sqrt{m}\wh \mu_{N,m}$ to the neighborhood of the origin. Since $\limsup_{m,N \rightarrow \infty} \frac{m^2}{N} < c$ for some small $c$ and $d$ is fixed, Lemma \ref{lemma:prelim} implies that $\l\|\sqrt{m}\wh \mu_{N,m}\r\| \leq \frac{1}{r_N}$ with probability approaching $1$ where $r_N = \sqrt{\frac{N}{m \log(m)}}$.
That is, we conclude that
\begin{multline*}
    \sqrt{N}\wh\mu_{N,m} = \argmax_{z : \|z\| \leq \frac{\sqrt{\frac{N}{m}}}{r_N}}\inf_{u \in \mathcal{S}^{d-1}}\l[\sqrt{\frac{N}{m}} U_{N,m}\l(u, \frac{z}{\sqrt{N/m}}\r) \r.\\
    +\l. \sqrt{\frac{N}{m}} \mb P \l(\dotp{\sqrt{m}\bar X_J - \frac{z}{\sqrt{N/m}}}{u}\geq 0\r)\r]
\end{multline*}
with probability approaching $1$. Therefore, it suffices to consider the asymptotic behavior of the above maximizer. Note that
\begin{multline*}
    \sqrt{\frac{N}{m}} U_{N,m}\l(u, \frac{z}{\sqrt{N/m}}\r) + \sqrt{\frac{N}{m}} \mb P \l(\dotp{\sqrt{m}\bar X_J - \frac{z}{\sqrt{N/m}}}{u}\geq 0\r) \\
    = \sqrt{\frac{N}{m}} U_{N,m}\l(u, \frac{z}{\sqrt{N/m}}\r)  +  \sqrt{\frac{N}{m}} \mb P \l(W_{u,m} \leq\dotp{-\frac{z}{\sqrt{N/m}}}{u}\r)\,,
\end{multline*}
where $W_{u,m} = \dotp{\sqrt{m} \bar X_m}{u}$. We remark that for  elliptically symmetric distributions with identity covariance, the law of $W_{u,m}$ is independent of $u$. Denote the density of $W_{u,m}$ by $\phi_m$; its existence is guaranteed by Lemma \ref{lemma:uniform-lipschitz}. In view of Taylor's expansion,
\begin{multline*}
    \sqrt{\frac{N}{m}} \mb P \l(W_{u,m} \leq\dotp{-\frac{z}{\sqrt{N/m}}}{u}\r) \\
    = \sqrt{\frac{N}{m}} \mb P \l(W_{u,m} \leq0\r) - \phi_m\l(-t \dotp{\frac{z}{\sqrt{N/m}}}{u}\r) \dotp{z}{u}\,,
\end{multline*}
where $t \in (0,1)$ is a positive number. Since $X_1 \sim \mathcal{E}\l(0, I_d, F\r)$, $\mb P \l(W_m\geq 0\r) = \frac{1}{2}$, hence suffices to consider
\begin{align*}
    \argmax_{z : \|z\| \leq \frac{\sqrt{N/m}}{r_N}}\inf_{u \in \mathcal{S}^{d-1}} \sqrt{\frac{N}{m}}  U_{N,m}\l(u, \frac{z}{\sqrt{N/m}}\r) - \phi_m\l(-t \dotp{\frac{z}{\sqrt{N/m}}}{u}\r) \dotp{z}{u} \,.
\end{align*}
Note that for $m$ sufficiently large, Lemma \ref{lemma:uniform-lipschitz} implies that there exists a constant $L$ independent of $m$ such that 
\[
\l\lvert \phi_m\l(-t \dotp{\frac{z}{\sqrt{N/m}}}{u}\r) - \phi_m\l(0\r) \r\rvert \leq L \l\lvert \dotp{\frac{z}{\sqrt{N/m}}}{u}\r\rvert. 
\]
Thus, uniformly over all $\|z\| \leq \frac{\sqrt{N/m}}{r_N}$ and $u \in \mathcal{S}^{d-1}$,
\begin{multline*}
    \l\lvert \phi_m\l(-t \dotp{\frac{z}{\sqrt{N/m}}}{u}\r) \dotp{z}{u} - \phi_m\l(0\r)\dotp{z}{u}\r\rvert \\
    \leq \frac{L}{\sqrt{N/m}} \l\lvert \dotp{z}{u}\r\rvert^2 \leq \frac{L\sqrt{N/m}}{r^2_N}\,,
\end{multline*}
and the right-hand side converges to $0$ since $r_N = \sqrt{\frac{N}{m \log(m)}}$. Therefore,
\begin{align*}
    \phi_m\l(-t \dotp{\frac{z}{\sqrt{N/m}}}{u}\r) \dotp{z}{u} = \phi_m(0)\dotp{z}{u} + o(1)
\end{align*}
as $m,N \rightarrow \infty$ uniformly over all $u \in \mathcal{S}^{d-1}$ and $z$ such that $\|z\| \leq \frac{\sqrt{N/m}}{r_N}$. Moreover, in view of Lemma \ref{lemma:implication-of-uniform-lipschitz}, $\phi_m(0) \rightarrow \phi(0)$ where $\phi$ is the density of the standard normal law, implying that $\phi_m(0)$ is bounded away from $0$ for large $m$. Therefore, 
\begin{align*}
    \sqrt{N}\wh\mu_{N,m} = \argmax_{z : \|z\| \leq \frac{\sqrt{\frac{N}{m}}}{r_N}}\inf_{u \in \mathcal{S}^{d-1}} \l[\frac{\sqrt{\frac{N}{m}}  U_{N,m}\l(u, \frac{z}{\sqrt{N/m}}\r)}{\phi_m(0)} -  \dotp{z}{u} + o(1)\r] \,.
\end{align*}
Next, Slutsky's lemma and Lemma \ref{thm:stochastic_convergence} imply that the process 
\begin{align*}
    (u,z) \mapsto \frac{\sqrt{\frac{N}{m}}  U_{N,m}\l(u, \frac{z}{\sqrt{N/m}}\r)}{\phi_m(0)} -  \dotp{z}{u}
\end{align*}
converges to $(u,z) \mapsto \dotp{W - z}{u}$ weakly. Clearly,
\begin{align*}
    W = \argmax_{z : \|z\| \leq \frac{\sqrt{\frac{N}{m}}}{r_N}}\inf_{u \in \mathcal{S}^{d-1}} \dotp{W - z}{u}\,.
\end{align*}
Thus, the argmax theorem \cite{wellner2013weak} applies, yielding that $\sqrt{N}\wh\mu_{N,m} \rightarrow W$ in distribution as $N \rightarrow \infty$.

\subsection{Proof of Lemma \ref{lemma:E_sup_sum_W}}
\label{proof:lemma:E_sup_sum_W}

    Since $\mb E h^{(j)}(X_1,\ldots,X_j) = 0$ for every $h \in \mathcal{H}_m$, it is immediate that $\mb E W_{i}(h) = \var \l(\l(h^{(j)} -h_0^{(j)}\r)(X_1,\ldots,X_j)\r)$. Recall that
    \begin{align*}
        h(X_1, X_2, \ldots, X_m) &= (\delta_{x_1} - P + P) \times (\delta_{x_2} - P + P) \times \ldots \times (\delta_{X_m} - P + P)h \\
        &= \sum_{q=1}^m \sum_{J \in \mathcal{I}^m_q} h^{(q)}(X_i, i \in J) \,,
    \end{align*}
    and all the terms in the sum are mutually orthogonal \citep[][Theorem 3 of Chapter 1.6]{lee2019u}. It implies that $\var (h(X_1,\ldots,X_m)) = \sum_{j=1}^m {m \choose j}\var\l(h^{(j)}\r)$, thus
    \begin{align}\label{eq:var-decomposition}
        \sup_{h \in \mathcal{H}_m} \mb E W_{i}(h)\leq C\sup_{h \in \mathcal{H}_m} \frac{\var(h(X_1,\ldots,X_m))}{{m \choose j}}\,.
    \end{align}
    Next, we are going to estimate $\mb E \sup_{h \in \mathcal{H}_m} \l\lvert \frac{\sum_{i=1}^{\lfloor N/j \rfloor} W_{i}(h) - \mb E W_{i}(h)}{\lfloor N/j \rfloor} \r\rvert$. In view of the the symmetrization inequality \citep{wellner2013weak},
    \begin{align*}
        \mb E \sup_{h \in \mathcal{H}_m} \l\lvert \frac{\sum_{i=1}^{\lfloor N/j \rfloor} W_{i}(h) - \mb E W_{i}(h)}{\lfloor N/j \rfloor} \r\rvert \leq 2 \mb E \sup_{h \in \mathcal{H}_m} \frac{1}{\lfloor N/j \rfloor} \l\lvert \sum_{i = 1}^{\lfloor N/j \rfloor} \varepsilon_i W_{i}(h)\r\rvert \,,
    \end{align*}
    where $\varepsilon_1, \varepsilon_2, \ldots, \varepsilon_{\lfloor N/j \rfloor}$ are i.i.d. Rademacher random variables. Recall that $W_{i}(h) = \l(\l(h^{(j)} -h_0^{(j)}\r)(X_{(i-1)j + 1}, X_{(i-1)j + 2}, \ldots, X_{ij})\r)^2$, in view of the fact that $\|h^{(j)}\|_\infty\leq 2^j$,  $\l\lvert W_{i}(h)\r\rvert \leq 4^{j+1}$. Since the function $x \mapsto x^2$ is Lipschitz continuous on $[-4^j, 4^j]$ with Lipschitz constant $4^{j+2}$, Talagrand's contraction inequality implies that
    \begin{multline*}
        \mb E_\varepsilon \sup_{h \in \mathcal{H}_m} \frac{1}{\lfloor N/j \rfloor} \l\lvert \sum_{i = 1}^{\lfloor N/j \rfloor} \varepsilon_i W_{i}(h)\r\rvert 
        \\
        \leq C^j \mb E_\varepsilon \sup_{h \in \mathcal{H}_m} \frac{1}{\lfloor N/j \rfloor} \l\lvert \sum_{i = 1}^{\lfloor N/j \rfloor} \varepsilon_i h \l(X_{(i-1)j + 1}, \ldots , X_{ij}\r)\r\rvert.
    \end{multline*}
    Next, in view of the definition \eqref{eq:d_Nj} of $d^{(2)}_{N,j}(h_1, h_2)$, generic chaining bound for sub-Gaussian processes \citep{talagrand2014upper} applied conditionally on $X_1,\ldots,X_n$ implies that 
    \[
    \mb E_\varepsilon \sup_{h \in \mathcal{H}_m} \frac{1}{\lfloor N/j \rfloor} \l\lvert \sum_{i = 1}^{\lfloor N/j \rfloor} \varepsilon_i W_{i}(h)\r\rvert
    \leq
    \frac{C^j}{\sqrt{\lfloor N/j \rfloor}} \mb E \gamma_2 \l(\mathcal{H}_m, d_{N,j}^{(2)}\r).
    \]

\subsection{Proof of Lemma \ref{lemma:1}}
\label{proof:lemma:1}

The upper bound for $\mb E\gamma_{2} \l(\mathcal{H}_m, d^{(2)}_{N,j}\r)$ follows along the lines of the proof of Theorem 3.12 in \cite{koltchinskii2011oracle} that applies to empirical processes. In particular, the upper bound is driven by the exponent $V$ and the diameter of $\mathcal{H}_m$ measured with respect to the distance $d^{(2)}_{N,j}$. The latter quantity is at most $2\sup_{h \in \mathcal{H}_m}\sqrt{\var\l(h^{(j)}\r)}$ which is bounded by $\frac{2\Sigma_{\m H_m}}{{m \choose j}^{1/2}}$, as explained in section \ref{proof:lemma:E_sup_sum_W}. 
Estimation of $\l(\mb E \l(\gamma_{2/j, \lfloor \log_2 (pj) \rfloor} (\mathcal{H}_m, d_{N,j}^{(1)}) \r)^p\r)^{1/p}$ is slightly more involved. To get the desired bounds, we will consider the small-$p$ and large-$p$ scenarios separately.

First, we consider the term $\mb E \gamma_2(\mathcal{H}_m, d_{N,j})$. Dudley's entropy integral estimate \citep{talagrand2014upper} implies that
\begin{align*}
    \gamma_{2} \l(\mathcal{H}_m, d^{(2)}_{N,j}\r)\leq C \int_0^{D^{(2)}_{N,j}} \l(\log_+ N(\mathcal{H}_m, d^{(2)}_{N,j}, \varepsilon) \r)^{1/2} d\varepsilon \,,
\end{align*}
where $D^{(2)}_{N,j} = \sup_{h_1, h_2 \in \mathcal{H}_m} d_{N,j}^{(2)}(h_1, h_2)$. Recall that $h^{(j)}(X_1, X_2, \ldots, X_j) = (\delta_{X_1} - P) \times \ldots \times (\delta_{X_j} - P) \times P^{m-j}h$ and note that
\begin{multline}\label{eq:decompose_h^j_distance}
    \l(\l(h_1^{(j)} - h_2^{(j)}\r)(X_1, \ldots, X_j)\r)^2 
    \\=  \l( \sum_{r=0}^j (-1)^{j-r}\sum_{J \in \mathcal{I}^j_r}  P^{m-r}(h_1 - h_2)(X_i, i \in J)\r)^2 \\
    \leq   (j+1) \sum_{r=0}^j  {j \choose r} \sum_{J \in \mathcal{I}^j_r} P^{m-r}(h_1 - h_2)^2(X_i, i \in J)\,,
\end{multline}
where in the case $r=0$, we define $\sum_{J \in \mathcal{I}^j_0} P^m (h_1 - h_2)(X_i, i \in J) \equiv P^m(h_1 - h_2)$ and $\sum_{J \in \mathcal{I}^j_0} P^m (h_1 - h_2)^2(X_i, i \in J) \equiv P^m(h_1 - h_2)^2$. Therefore,
\begin{multline*}
    \frac{1}{\lfloor N/j \rfloor} \sum_{i=1}^{\lfloor N/j \rfloor} \l(\l(h_1^{(j)} - h_2^{(j)}\r)\l(X_{(i-1)j + 1}, \ldots,X_{ij}\r)\r)^2 \\ 
    \leq \sum_{r=0}^j \frac{(j+1){j \choose r}}{\lfloor N/j \rfloor}\sum_{i=1}^{\lfloor N/j \rfloor}\sum_{(k^i_1, \ldots, k^i_r)} P^{m-r}(h_1-h_2)^2(X_{k^i_1}, \ldots, X_{k^i_r}) \,,
\end{multline*}
where the last summation is taken over all 
tuples $(k^i_1, \ldots, k^i_r) \subseteq \{(i-1)j + 1, \ldots, ij\}$. Here, the convention is that for $r = 0$, the term 
\[
\sum_{i=1}^{\lfloor N/j \rfloor}\sum_{(k^i_1, \ldots, k^i_r)} P^{m-r}(h_1-h_2)^2(X_{k^i_1}, \ldots, X_{k^i_r})
\]
equals $P^{m}(h_1-h_2)^2$. Since ${j \choose r} \leq 2^j$ for all $0 \leq r \leq j$,
\begin{multline}\label{eq:decompose_h^j_distance-2}
    \frac{1}{\lfloor N/j \rfloor} \sum_{i=1}^{\lfloor N/j \rfloor} \l(\l(h_1^{(j)} - h_2^{(j)}\r)\l(X_{(i-1)j + 1}, \ldots,X_{ij}\r)\r)^2 \\
    \leq 4^j(j+1)^2 \frac{\sum_{r=0}^j \sum_{i=1}^{\lfloor N/j \rfloor}\sum_{(k^i_1, \ldots, k^i_r)} P^{m-r}(h_1-h_2)^2(X_{k^i_1}, \ldots, X_{k^i_r})}{\lfloor N/j \rfloor (j+1) {j \choose r}}\,.
\end{multline}
The fraction above can be viewed as the $L_2(Q_{N,j})$ distance between $h_1$ and $h_2$ where
the probability measure $Q_{N,j}$ is defined by
\begin{align*}
    Q_{N,j} = \frac{\sum_{r=0}^j \sum_{i=1}^{\lfloor N/j \rfloor} U_{N,i}^{(r)} \times P^{m-r}}{\lfloor N/j \rfloor (j+1) }\,.
\end{align*}
Here, $U_{N, i}^{(r)} \times P^{m-r} = \frac{1}{{j \choose r}} \sum_{(k^i_1, \ldots, k^i_r)} \delta_{X_{k^i_1}, \ldots, X_{k^i_r}} \times P^{m-r}$ for all $i \in \l[\lfloor N/j \rfloor\r]$. Thus
\begin{multline*}
    \int_0^{D^{(2)}_{N,j}} \l(\log_+ N(\mathcal{H}_m, d^{(2)}_{N,j}, \varepsilon) \r)^{1/2} d\varepsilon \\
    \leq C \int_0^{D^{(2)}_{N,j}} \l(\log_+ N \l(\mathcal{H}_m, \|\cdot\|_{L_2(Q_{N,j})}, \frac{\varepsilon}{2^j(j+1)} \r) \r)^{1/2} d\varepsilon  \\
    \leq 2^j(j+1)  \int_0^{D^{(2)}_{N,j}/(2^j(j+1))} \l(\log_+ N \l(\mathcal{H}_m, \|\cdot\|_{L_2(Q_{N,j})}, \varepsilon  \r) \r)^{1/2} d\varepsilon \\
    \leq CV^{1/2}2^j(j+1) \int _0^{D^{(2)}_{N,j}/(2^j(j+1))} \l(\log \l(\frac{C'}{\varepsilon}\r)\r)^{1/2} d\varepsilon
\end{multline*}
where the last inequality is implied by the definition \eqref{eq:euclidean}. 
Therefore,
\begin{align*}
    \mb E \gamma_2(\mathcal{H}_m, d^{(2)}_{N,j}) \leq C^{j}V^{1/2} \mb E \l[\int _0^{ D^{(2)}_{N,j}/(2^j(j+1))} \l(\log \l(\frac{C'}{\varepsilon}\r)\r)^{1/2} d\varepsilon \r]
\end{align*}
Remainder of the proof follows the argument behind Theorem 3.12 in \cite{koltchinskii2011oracle}. We only mention the necessary modifications. 
Note that
\begin{multline*}
    \mb E \l[D_{N,j}^{(2)}\r]^2 = \mb E \sup_{h_1, h_2 \in \mathcal{H}_m}  \frac{1}{\lfloor N/j \rfloor} \sum_{i=1}^{\lfloor N/j \rfloor} \l(\l(h_1^{(j)} - h_2^{(j)}\r)\l(X_{(i-1)j + 1}, \ldots,X_{ij}\r)\r)^2 \\
    \leq 2 \mb E \sup_{h \in \mathcal{H}_m}  \frac{1}{\lfloor N/j \rfloor} \sum_{i=1}^{\lfloor N/j \rfloor} \l(h^{(j)}\l(X_{(i-1)j + 1}, \ldots,X_{ij}\r)\r)^2\,.
\end{multline*}
Lemma \ref{lemma:E_sup_sum_W} implies that
\begin{align*}
    \mb E \l[D_{N,j}^{(2)}\r]^2 \leq  \frac{C_1^{j}}{\lfloor N/j \rfloor^{1/2}} \mb E \gamma_2 (\mathcal{H}_m, d^{(2)}_{N,j}) + \frac{\Sigma_{\m H_m}^2}{{m \choose j}}\,.
\end{align*}
To this end, straightforward application of the aforementioned techniques from \citep{koltchinskii2011oracle} yield that 
\begin{align*}
    \mb E \gamma_{2} \l(\mathcal{H}_m, d^{(2)}_{N,j}\r) = \frac{C^j_1 V^{1/2}\Sigma_{\m H_m}}{{m \choose j}^{1/2}} \log^{1/2} \l(\frac{C_1'}{\Sigma_{\m H_m}^2}{m \choose j}\r) + \frac{C^j_2V}{\lfloor N/j \rfloor^{1/2}} \log \l(\frac{C_1'}{\Sigma_{\m H_m}^2}{m \choose j}\r)\,.
\end{align*}
Finally, the estimates $\l(\frac{m}{j}\r)^j \leq {m \choose j} \leq \l(\frac{em}{j}\r)^j$ yield the bound \eqref{eq:E_gamma_2}. Note that in the display above, instead of $\Sigma_{\m H_m}$ one can use its upper bound that equals $1$ by assumption.

Next, let us prove inequality \eqref{eq:truncated_gamma_bound_1}. Proposition 1.2.1 in \citep{talagrand2005generic} implies that
\begin{align*}
    \gamma_{2/j} (\mathcal{H}_m, d^{(1)}_{N,j})\leq C(\log 2)^{-j/2}(1- 1/2^{j/2})^{-1} \int_0^{D_{N,j}^{(1)}} \l(\log_+ N(\mathcal{H}_m, d^{(1)}_{N,j}, \varepsilon) \r)^{j/2} d\varepsilon,
\end{align*}
where $D_{N,j}^{(1)} = \sup_{h_1, h_2 \in \mathcal{H}_m} d_N^{(1)}(h_1, h_2)$. Repeating the reasoning behind inequalities \eqref{eq:decompose_h^j_distance} to \eqref{eq:decompose_h^j_distance-2}, we deduce that
\begin{multline*}
    e^2\frac{1}{{N \choose j}}\sum_{J \in \mathcal{I}^N_j} \l(h^{(j)}_1 \l(X_i, i \in J\r) - h^{(j)}_2 \l(X_i, i \in J\r)\r)^2 \\
    \leq 4^{j}e^2(j+1)^2 \frac{\sum_{r=0}^j \sum_{J \in \mathcal{I}^N_r} P^{m-r}(h_1 - h_2)^2(X_i, i \in J) }{(j+1){N \choose r}}\,,
\end{multline*}
where for $r=0$,
\(\sum_{J \in \mathcal{I}^j_0} P^m (h_1 - h_2)(X_i, i \in J) \equiv P^m(h_1 - h_2) \) and \(\sum_{J \in \mathcal{I}^j_0} P^m (h_1 - h_2)^2(X_i, i \in J) \equiv P^m(h_1 - h_2)^2 \). Define the probability measure 
\begin{align}\label{eq:u_N}
    Q_N = \frac{1}{j+1} \sum_{r=0} ^ j U^{(r)}_N \times P^{m-r}.
\end{align}
Here, for every $r \in [m]$, $U^{(r)}_N \times P^{m-r}$ denotes a random measure
\begin{align*}
    U^{(r)}_N \times P^{m-r} = \l(\frac{1}{{N \choose r}}\sum_{[i_1, i_2, \ldots, i_r] \in \mathcal{I}_r^N }\delta_{(X_{i_1}, X_{i_2}, \ldots, X_{i_r})} \r) \times P^{m-r}
\end{align*}
defined on $\mb R^m$ and $U^{(0)}_N \times P^m := P^m$. Thus, $\frac{\sum_{r=0}^j \sum_{J \in \mathcal{I}^N_r} P^{m-r}(h_1 - h_2)^2(X_i, i \in J) }{(j+1){N \choose r}}$ equals $d^2_u (h_1, h_2) := \|h_1 - h_2\|^2_{L_2(Q_N)}$. Therefore, whenever $d_u(h_1, h_2) \leq \frac{\varepsilon}{2^{j}e(j+1)}$, we have that $d^{(1)}_{N,j}(h_1, h_2) \leq \varepsilon$. We deduce that
\begin{multline*}
    \int_0^{D_{N,j}^{(1)}} \l(\log_+ N(\mathcal{H}_m, d^{(1)}_{N,j}, \varepsilon) \r)^{j/2} d\varepsilon \\
    \leq \int_0^{D_{N,j}^{(1)}} \l(\log_+ N\l(\mathcal{H}_m, d_u, \frac{\varepsilon}{2^je(j+1)}\r) \r)^{j/2} d\varepsilon \\
    \leq 2^j e(j+1) \int_0^{\frac{D_{N,j}^{(1)}}{2^je(j+1)}} \l(\log_+ N\l(\mathcal{H}_m, d_u, \varepsilon\r) \r)^{j/2} d\varepsilon \,.
\end{multline*}
In view of the definition of Euclidean classes \eqref{eq:euclidean} and the fact that functions in $\mathcal{H}_m$ are uniformly bounded by $1$, 
\begin{equation*}
    \int_0^{\frac{D_{N,j}^{(1)}}{2^je(j+1)}} \l(\log_+ N\l(\mathcal{H}_m, d_u, \varepsilon \r) \r)^{j/2} d\varepsilon 
    \\
    \leq (C_1V)^{j/2} \int_0^{\frac{D_{N,j}^{(1)}}{2^je(j+1)}} \l(\log \l(\frac{C_2}{\varepsilon}\r) \r)^{j/2} d\varepsilon
\end{equation*}
For all $j \geq 2$, we have that 
\begin{multline*}
    \l(\log \l(\frac{C_2 }{\varepsilon}\r) \r)^{j/2} \\
    \leq 2^{j/2-1}\l(\log \l(\frac{C_2 }{\frac{D_{N,j}^{(1)}}{2^je(j+1)}}\r) \r)^{j/2} + 2^{j/2-1}\l(\log \l(\frac{\frac{D_{N,j}^{(1)}}{2^je(j+1)}}{\varepsilon}\r) \r)^{j/2}\,.
\end{multline*}
Since
\begin{multline*}
    \int_0^{\frac{D_{N,j}^{(1)}}{2^je(j+1)}} \l(\log \l(\frac{\frac{D_{N,j}^{(1)}}{2^je(j+1)}}{\varepsilon}\r) \r)^{j/2} d\varepsilon \\
    = \frac{D_{N,j}^{(1)}}{2^je(j+1)} \int_0^1 \l(\log(1/\varepsilon)\r)^{j/2} d\varepsilon \leq C^jj^{j/2} D_{N,j}^{(1)}\,
\end{multline*}
we deduce that 
\begin{multline*}
\mb E \l[ \gamma_{2/j} (\mathcal{H}_m, d_N) \r]^p \leq \l(C_1 Vj\r)^{pj/2} \mb E \l(D_{N,j}^{(1)}\r)^p 
\\
+ (C_2 V)^{pj/2} \mb E \l[D_{N,j}^{(1)} \l(\log\l(\frac{C_3^j}{D_{N,j}^{(1)}}\r)\r)^{j/2}\r]^p\,.
\end{multline*}
Next, let us estimate $\mb E \l(D_{N,j}^{(1)}\r)^p$. Recall that 
\begin{multline*}
    \l(D_N^{(1)}\r)^2 = e^2\sup_{h_1, h_2 \in \mathcal{H}_m}  \frac{1}{{N \choose j}}\sum_{J \in \mathcal{I}^N_j} \l(h^{(j)}_1 \l(X_i, i \in J\r) - h^{(j)}_2 \l(X_i, i \in J\r)\r)^2 \\
    \leq 2e^2\sup_{h \in \mathcal{H}_m}  \frac{1}{{N \choose j}}\sum_{J \in \mathcal{I}^N_j} \l(h^{(j)} \l(X_i, i \in J\r)\r)^2 \,,
\end{multline*}
hence
\begin{align*}
    \mb E \l(D_{N,j}^{(1)}\r)^p \leq C^p \mb E \l[\sup_{h \in \mathcal{H}_m}  \frac{1}{{N \choose j}}\sum_{J \in \mathcal{I}^N_j} \l(h^{(j)} \l(X_i, i \in J\r)\r)^2\r]^{p/2}\,.
\end{align*}
The bound for this term was provided in display \eqref{eq:h_j_square} in the proof of Theorem \ref{thm: thm2}. That is,
\begin{multline*}
\mb E \sup_{h \in \mathcal{H}_m} \l(\frac{1}{{N \choose j}} \sum_{J \in \mathcal{I}^N_j} \l(h^{(j)}\l(X_i, i \in J\r)\r)^2\r)^{p/2} \leq 
        C_1^{p}\l(\frac{j}{m}\r)^{pj/2} \Sigma_{\m H_m}^{p}
        \\
        + C_2^{pj} \l(\frac{j}{N}\r)^{p/4} \l( \l(p\sqrt{\frac{j}{N}}\r)^{p/2} + \l(\mb E \gamma_2 (\mathcal{H}_m, d^{(2)}_{N,j})\r)^{p/2}\r).  
\end{multline*}
The upper bound $\Gamma(j,m,N, V)$ for $\mb E \gamma_2 \l(\mathcal{H}_m, d^{(2)}_{N,j}\r)$ was already given in the first part of this proof and is summarized in display \eqref{eq:E_gamma_2}. It yields an estimate
\begin{multline*}
    \mb E \sup_{h \in \mathcal{H}_m} \l(\frac{1}{{N \choose j}} \sum_{J \in \mathcal{I}^N_j} \l(h^{(j)}\l(X_i, i \in J\r)\r)^2\r)^{p/2} \leq 
        C_1^{p}\l(\frac{j}{m}\r)^{pj/2} \Sigma_{\m H_m}^{p}
        \\
        + C_2^{pj} \l[\sqrt{\frac{j}{N}}\l( p\sqrt{\frac{j}{N}}\bigvee \Gamma(j,m,N, V)\r)\r]^{p/2}. 
\end{multline*}
Observe that in view of the inequality $ab\leq (a^2+b^2)/2$, 
\[
\Gamma(j,m,N,V)\sqrt{\frac{j}{N}}\leq  V\frac{j}{N}\log\l(\frac{C_1'}{\Sigma_{\m H_m}^2}{m \choose j}\r)\bigvee \Sigma^2_{\m H_m}\l(\frac{j}{m}\r)^{j}.
\]
Recalling that
\[
D(p,j):=D(p,j,m,N) = \l[\frac{pj}{N}\bigvee \frac{Vj}{N}\log\l(\frac{C_1'}{\Sigma_{\m H_m}^2}{m \choose j}\r)\bigvee \Sigma^2_{\m H_m}\l(\frac{j}{m}\r)^{j} \r]^{1/2},
\]
we see that $\mb E\l(D_{N,j}^{(1)}\r)^p \leq (D(p,j))^p$.
Now we will find an upper bound for 
\[
\mb E \l[D_{N,j}^{(1)} \l(\log\l(\frac{C_3^j}{D_{N,j}^{(1)}}\r)\r)^{j/2}\r]^p.
\]
Elementary calculus implies that for any $c > 0$, the function $x \mapsto x \l(\log\l(\frac{c}{x}\r)\r)^{j/2}$ is increasing when $0 < x \leq ce^{-j/2}$. Assuming without loss of generality that $C_3>eC$ where the constant $C$ appears in the definition of $D(p,j)$, we deduce that whenever $D_{N,j}^{(1)} \leq D(2)\wedge C^j$,
\begin{equation*}
     D_{N,j}^{(1)} \l(\log\l(\frac{C_3^j }{D_{N,j}^{(1)}}\r)\r)^{j/2}  \leq 
     D(2) \l(\log\l(\frac{C_3^j}{D(2)\wedge C^j}\r)\r)^{j/2},
\end{equation*}
hence
\begin{multline*}
    \mb E \l[D_{N,j}^{(1)} \l(\log\l(\frac{C_3^j }{D_{N,j}^{(1)}}\r)\r)^{j/2}\r]^p \\
    \leq (D(2))^p \l(\log\l(\frac{C_3^j}{D(2)\wedge C^j}\r)\r)^{pj/2}\mb E I\l\{D_{N,j}^{(1)} < D(2)\r\} \\
    + \mb E\l(D_{N,j}^{(1)} \r)^p \l(\log\l(\frac{C_3^j}{D(2)\wedge C^j}\r)\r)^{pj/2} \\
    \leq 2(D(p,j))^p \l(\log\l(\frac{C_3^j}{D(2)\wedge C^j}\r)\r)^{pj/2}.
\end{multline*}
Collecting all the estimates, we get the inequality
\begin{equation*}
    \l(\mb E \l(\gamma_{2/j} (\mathcal{H}_m, d^{(1)}_{N,j}) \r)^p\r)^{1/p} 
    \leq C^{j} V^{j/2} D(p,j) \l(\log\l(\frac{C_3^j}{D(2)\wedge C^j}\r)\bigvee j\r)^{j/2}.
\end{equation*}
Since $D(2)\geq \Sigma_{\m H_m}\l( \frac{Cj}{m}\r)^{j/2}$, we can replace $D(2)$ by its lower bound in the display above to get the bound \eqref{eq:truncated_gamma_bound_1}. 
Finally, we turn our attention to the inequality \eqref{eq:truncated_gamma_bound_2}. For any pseudometric $d$ on $\mathcal{H}_m$, define
\begin{align*}
    D_l(d) = \inf \l\{\varepsilon : N\l(\mathcal{H}_m, d, \varepsilon\r) \leq 2^{2^l}\r\}.
\end{align*}
Let $M > 0$ be arbitrary. If two pseudometrics $d$ and $d'$ satisfy $N\l(\mathcal{H}_m, d, \varepsilon\r) \leq N\l(\mathcal{H}_m, d', M\varepsilon\r)$, then
\begin{multline*}
    \inf \l\{\varepsilon : N\l(\mathcal{H}_m, d, \varepsilon\r) \leq 2^{2^l}\r\} \leq \inf \l\{\varepsilon : N\l(\mathcal{H}_m, d', M\varepsilon\r) \leq 2^{2^l}\r\} \\
    = \frac{1}{M}\inf \l\{M\varepsilon : N\l(\mathcal{H}_m, d', M\varepsilon\r) \leq 2^{2^l}\r\}\,.
\end{multline*}
Thus, $D_l(d) \leq M^{-1}D_l(d')$. For any $l\in N$, truncated chaining complexity $\gamma_{2/j, l} (\mathcal{H}_m, d_N)$ admits a bound
\begin{align*}
    \gamma_{2/j, l} (\mathcal{H}_m, d_N)\leq C(\log 2)^{-j/2}(1- 1/2^{j/2})^{-1} \int_0^{D_{l}(d_N)} \l(\log_+ N(\mathcal{H}_m, d_N, \varepsilon) \r)^{j/2} d\varepsilon\,.
\end{align*}
Let $d_u(h_1,h_2)=\|h_1-h_2\|_{L_2(Q_N)}$ where $Q_N$ is the probability measure defined in \eqref{eq:u_N}. Then
\begin{multline*}
    \int_0^{D_l(d_N)} \l(\log_+ N(\mathcal{H}_m, d_N, \varepsilon) \r)^{j/2} d\varepsilon \\
    \leq \int_0^{2^je(j+1)D_l(d_u)} \l(\log_+ N\l(\mathcal{H}_m, d_u, \frac{\varepsilon}{2^{j}e(j+1)}\r)\r)^{j/2} d\varepsilon \\
    = 2^{j}e(j+1) \int_0^{D_l(d_u)} \l(\log_+ N\l(\mathcal{H}_m, d_u, \varepsilon \r)\r)^{j/2} d\varepsilon\,.
\end{multline*}
Since $\m H_m$ is Euclidean, 
\begin{multline*}
    D_l(d_u) = \inf\l\{\varepsilon : N\l(\mathcal{H}_m, d_u, \varepsilon\r) \leq 2^{2^l}\r\} \\
    \leq \inf\l\{\varepsilon : \l(\frac{C'}{\varepsilon}\r)^{V} \leq 2^{2^l}\r\} = C' 2^{-\frac{2^l}{V}}\,.
\end{multline*}
We deduce that
\begin{multline*}
    \int_0^{D_l(d_u)} \l(\log_+ N\l(\mathcal{H}_m, d_u, \varepsilon \r)\r)^{j/2} d\varepsilon 
    \\
    \leq  C^{j/2}V^{j/2}\int_0^{C' 2^{-\frac{2^l}{V}}} \l(\log_+\l(\frac{C'}{\varepsilon}\r)\r)^{j/2} d\varepsilon \\
    = C^{j/2}V^{j/2} \int_{\frac{2^l \log 2}{V}}^\infty t^{j/2}e^{-t} dt\,.
\end{multline*}
Lemma \ref{lemma:incomplete-gamma-bound} implies that the integral factor is at most $C^j 2^{lj/2} 2^{-2^l/V}$ whenever $2^{l+1} > \frac{Vj}{\log 2}$. Taking $l = \l\lfloor\log_2 p + \log_2 j \r\rfloor$ yields the inequality
\begin{align*}
    \gamma_{2/j, l} (\mathcal{H}_m, d_N)\leq 
    C^j (pj)^{j/2} e^{-pj/V}.
\end{align*}
if $p > \frac{V}{\log 2}$.

\subsection{Proof of Lemma \ref{lemma:prelim}}
\label{proof:lemma:prelim}

We only give a sketch of the proof and skip standard computations. Consider the U-process
\(
F_{N,m}(u)=\frac{1}{{N \choose m}} \sum_{J \in \mathcal{I}^N_m} I\left\{ \dotp{\bar X_J}{u}\geq 0\right\}, \ \|u\|=1,
\)
and recall that it admits the following \emph{Hoeffding's representation} \citep{hoeffding1994probability}: 
    \begin{align*}
        F_{N,m}(u) = \frac{1}{N!} \sum_\pi w_{\pi}(u)\,,
    \end{align*}
    where the sum is taken over all permutations $\pi: [N] \mapsto [N]$, and
    \begin{align*}
        w_{\pi}(u) = \frac{1}{\l\lfloor N/m \r\rfloor}\sum_{i=0}^{\l\lfloor N/m \r\rfloor-1} I\left\{ \dotp{\frac{1}{m}\sum_{j=1}^m X_{\pi(im+j)}}{u}\geq 0\right\}.
     \end{align*}
In other words, it represents $F_{N,m}(u)$ as an average of averages of independent random variables. It is well known that this representation offers a simple reduction from the suprema of U-processes to the suprema of empirical processes. Indeed, it follows from Jensen's inequality that for any convex function $\psi$, 
\[
\mb E\psi\l(\sup_{\|u\|=1}(F_{N,m}(u)-\mb EF_{N,m}(u))\r)\leq \mb E \psi\l(\sup_{\|u\|=1}(w_{\mathrm{id}}(u)-\mb Ew_{\mathrm{id}}(u))\r)
\]
where $\mathrm{id}$ is the identity permutation. In particular, this is true for the moment generating function, hence standard concentration inequalities such as the one given in Lemma \ref{lemma:talagrand-bousquet} can be used to control $\sup_{\|u\|=1 }\l| F_{N,m}(u)-\mb EF_{N,m}(u) \r|$. 
In particular, it can be shown (for instance, see \citep[Theorem 2.1]{chen2018robust}) that 
\[
\sup_{\|u\|=1} \psi\l(w_{\mathrm{id}}(u)-\mb Ew_{\mathrm{id}}(u)\r)\leq C\sqrt{m\frac{\max(d,t)}{N}}
\]
with probability at least $1-e^{-t}$, and the previous discussion implies that the same result is true for $\sup_{\|u\|=1}(F_{N,m}(u)-\mb EF_{N,m}(u))$. From here, it is standard to conclude, following the same argument as Theorem \ref{th:tukey-nonasymp} that $D_{m}(\wh\mu_{N,m})\geq \frac{1}{2} - C\sqrt{\frac{\max(d,t)}{N}}$, whence (see Lemma \ref{lemma:depth}))
\[
\l\|\wh\mu_{N,m}\r\|\leq C'\sqrt{\frac{\max(d,t)}{N}}
\]
with probability at least $1-e^{-t}$.

\subsection{Proof of Lemma \ref{lemma:cont-modulus}}
\label{section:proof:cont-modulus}

Define $U_{N,m}(u,z) =  \frac{1}{{N\choose m}} \sum_{J\in \m I_{m}^{(N)} }I\left\{ \dotp{\bar Z_J-z}{u}\geq 0\right\} - \Phi_m\l( H(z,u)\r)$ and assume that $\m E$ is the event of probability at least $1-e^{-t}$ on which $\l\|\wh\mu_{N,m}\r\|\leq \bar\delta:=C\sqrt{\frac{\max(d,t)}{N}}$. Note that on event $\m E$, 
\[
\l|\tilde R_{N,m}-R_{N,m}\r| \leq \sup_{\|u\|=1,\|z\|\leq \bar \delta}\l| U_{N,m}(u,z) - U_{N,m}(u,0) \r|.
\]
Next, we proceed in a way similar to the proof of Lemma \ref{lemma:prelim}. Namely, Hoeffding's representation implies that
\[
U_{N,m}(u,z) =  \frac{1}{N!} \sum_\pi w_{\pi}(u,z),
\]
where the sum is taken over all permutations $\pi: [N] \mapsto [N]$, and
\begin{align*}
        w_{\pi}(u,z) = \frac{1}{\l\lfloor N/m \r\rfloor}\sum_{i=0}^{\l\lfloor N/m \r\rfloor-1} I\left\{ \dotp{\frac{1}{m}\sum_{j=1}^m X_{\pi(im+j)} - z}{u}\geq 0\right\} - \Phi_m\l( H(z,u)\r).
\end{align*}
Using the same reasoning as Lemma \ref{lemma:prelim}, we claim that it suffices to find an upper bound for 
\[
\sup_{\|u\|=1,\|z\|\leq \bar \delta}\l| w_{\mathrm{id}}(u,z) - w_{\mathrm{id}}(u,0)\r|
\]
where $\mathrm{id}$ stands for the identity permutation $\pi(j)=j, \ j=1,\ldots,N$. 
The latter quantity is the modulus of continuity of the empirical process indexed by half-spaces, and it is standard to estimate it using Lemma \ref{lemma:talagrand-bousquet}. For example, Lemma 3.1 in \citep[][]{minsker2024improved} states that there exist absolute constants $C_1,C_2>0$ such that 
    \begin{multline*}
        \sup\limits_{\|u\|=1,\|z\|\leq \bar \delta} \l| w_{\mathrm{id}}(u,z) - w_{\mathrm{id}}(u,0)\r|
        \\
        \leq C_1\l( \sqrt{\bar\delta}\l(\sqrt{\frac{tm}{N}}+\sqrt{\frac{dm}{N}}\log^{1/2}\l( \frac{C_2}{\bar\delta}\r)\r) + \frac{m}{N}\l( t+d\log\l( \frac{C_2}{\bar\delta}\r)\r)\r)
    \end{multline*}
with probability at least $1-e^{-t}$. Observe that 
\begin{multline*}
C_1\frac{m}{N}\l( t+d\log\l( \frac{C_2}{\bar\delta}\r)\r) \leq 
C_3 \frac{\max(d,t)m}{N} \log\l(\frac{C_2N }{\max(d,t)}\r)
\\
= C_3\sqrt{\frac{\max(d,t)m}{N}} \frac{1}{\sqrt{m}} \sqrt{\frac{\max(d,t)m^2}{N} }\l(\log\l(\frac{C_2N }{\max(d,t)m^2}\r)+ \log(m^2)\r)
\\
\leq C_4 \sqrt{\frac{\max(d,t)m}{N}}\frac{\log(m)}{\sqrt m}
\end{multline*}
since the function $z\mapsto \sqrt{z}\log(1/z)$ is increasing for $0<z\leq e^{-2}$ and $\frac{\max(d,t)m^2}{N}\leq c$ by assumption. Similarly, since $ab\leq (a^2+b^2)/2$,
\begin{multline*}
C_1 \sqrt{\bar\delta}\l(\sqrt{\frac{tm}{N}}+\sqrt{\frac{dm}{N}}\log^{1/2}\l( \frac{C_2}{\bar\delta}\r)\r) 
\\
\leq C_3 \l( \bar\delta + \frac{m}{N}\l( t+d\log\l( \frac{C_2}{\bar\delta}\r)\r)\r)
\\
\leq C_4\sqrt{\frac{\max(d,t)m}{N}}\l( \frac{1}{\sqrt m} + \frac{\log(m)}{\sqrt m} \r).
\end{multline*}
Combination of the displays above implies that
\begin{equation*}
        \sup\limits_{\|u\|=1,\|z\|\leq \bar\delta} \l| w_{\mathrm{id}}(u,z) - w_{\mathrm{id}}(u,0)\r| \\
        \leq C\sqrt{\frac{\max(d,t)m}{N}}\l( \frac{1}{\sqrt m} + \frac{\log(m)}{\sqrt m} \r)
\end{equation*}
with probability at least $1-2e^{-t}$, as claimed.

\subsection{Proof of Lemma \ref{lemma:hat-J}}
\label{proof:hat-J}

Define 
\[
    \zeta(x) = \begin{cases}
        0 & x\leq 1/2,\\
        1 & x\geq 1, \\
        2x - 1 & x\in (1/2, 1).
    \end{cases}
\]
Moreover, set
\[
    Y_k(u) := \frac{|\langle X_k, u\rangle|}{m^{1/4}}\quad\text{and}\quad S=\sup_{\|u\| = 1}\left(\sum_{k=1}^N \zeta(Y_k(u)) - \mb{E} \zeta(Y_k(u))\right). 
\]
By the definition of $J(u)$, for any $u\in \mb{R}^d$,
\begin{multline*}
        |J(u)|  = \sum_{k=1}^N I\left\{\langle X_k, u\rangle^2 \geq m^{1/2}\right\} 
        = \sum_{k=1}^N I\left\{Y_k(u) \geq 1\right\} \leq \sum_{k=1}^N \zeta\left(Y_k(u)\right),
\end{multline*}
Note that $\zeta$ is Lipschitz continuous with Lipschitz constant equal to $2$ and that $I\{x\geq 1/2\}\geq \zeta(x) \geq I\{x\geq 1\}$. In view of Markov's inequality and assumption stating that $M_4^4:=\mb E \dotp{X}{u}^4<\infty$,
\begin{equation}
\label{rhoexpectation}
    \mb{E} \zeta\big(Y_k(u)\big) \leq \mb{P}\left(\langle X_k, u\rangle^2 \geq m^{1/2}\right)\leq \frac{\mb{E}\langle X,u\rangle^4}{m^{1/2}} = \frac{M_4^4}{m^{1/2}}.
\end{equation}
We deduce that 
\begin{multline*}
    \mb{P}\left( \sup_{\|u\| = 1}|J(u)| > c'N\right)
    \leq \mb{P} \left(\sup_{\|u\|=1}\sum_{k=1}^N \zeta(Y_k(u))>c'N\right)
    \\
    \leq \mb{P} \Bigg(S>c'N - M_4^4 \frac{N}{\sqrt m}\Bigg).
\end{multline*}
We will now estimate the right-hand side in the display above using Bousquet's version of Talagrand's inequality (Lemma \ref{lemma:talagrand-bousquet} in the appendix). Denote 
\[
\sigma^2 = \sup_{\|u\| = 1}\var\big(\zeta(Y_1(u))\big)
\]
and observe that $\mb{P} \left(S>c'N - M_4^4 \frac{N}{\sqrt m}\right)\leq e^{-t}$ whenever
\[
     M_4^4 \frac{N}{\sqrt m} + 2\mb{E}S + \sigma\sqrt{2tN} +  4t/3 \leq c'N.
\]
To prove that this relation holds for suitable choices of parameters $c$ and $t$, first note that 
\begin{equation*}
    \var\big(\zeta(Y_1(u))\big) \leq \mb{E}(\zeta(Y_1(u))^2\leq \mb{E}\zeta(Y_1(u)) \leq M_4^4 \frac{1}{m^{1/2}}.
\end{equation*}
Therefore, 
\(
\sigma^2 \leq M_4^4 \frac{1}{m^{1/2}}.
\)
Next, we will estimate $\mb{E}S$. 
Let $\varepsilon_1,\dots \varepsilon_n$ be a sequence of independent random signs. The standard argument based on application of symmetrization and contraction inequalities \citep[Theorem 4.4]{ledoux2013probability}, together with the fact that $\zeta(\cdot)$ is Lipschitz continuous with Lipschitz constant equal to $2$, yields that 
\begin{align*}
    \mb{E}S 
    \leq 2\mb{E}\sup_{\|u\| = 1}\left|\sum_{k=1}^N \varepsilon_k\zeta(Y_k)\right|
    &\leq 8\mb{E}\sup_{\|u\| = 1}\left|\sum_{k=1}^N \varepsilon_k \frac{|\langle X_k, u\rangle|}{m^{1/4}}\right|
    \\
    &\leq \frac{16}{m^{1/4}}\mb{E}\sup_{\|u\| = 1}\left|\sum_{k=1}^N \varepsilon_k \langle X_k, u\rangle\right|.
\end{align*}
Note that 
\begin{multline*}
\mb{E}\sup_{\|u\| = 1}\left|\sum_{k=1}^N \varepsilon_k \langle X_k, v\rangle\right| = \mb E \l\| \sum_{k=1}^N \varepsilon_k  X_k \r\|_2 \leq \mb E^{1/2}\l\| \sum_{k=1}^N  \varepsilon_k  X_k  \r\|_2^2
\\
=\l( \sum_{k=1}^N \mb E \|X_k\|_2^2 \r)^{1/2} = \l( N  d\r)^{1/2},
\end{multline*}
hence we conclude that 
\begin{equation} 
     \mb{E} S \leq \frac{16}{m^{1/4}} \l( N  d\r)^{1/2}.
\end{equation}
As a consequence, we have to choose $c$ and $t$ such that 
\begin{equation}    
\label{eq:t:rank}
M_4^4 \frac{N}{m^{1/2}}+\frac{32}{m^{1/4}} \l( N  d\r)^{1/2}+M_4^2 (2t)^{1/2}\frac{N^{1/2}}{m^{1/4}}+4t/3\leq c'N,
\end{equation}
which is satisfied if both $d$ and $t$ do not exceed a constant times $N$.

\subsection{Proof of Lemma \ref{thm:stochastic_convergence}}
\label{proof:thm:stochastic_convergence}

Note that
\begin{multline*}
    U_{N,m}(u,z) = \frac{1}{{N\choose m}} \sum_{J\in \m A_{N}^{(m)} }I\left\{ \dotp{\Sigma^{-1/2}\l(\bar X_J - \mu\r) - \frac{z}{\sqrt{m}}}{u}\geq 0\right\} \\
    - \pr{\dotp{\Sigma^{-1/2}\l(\bar X_m - \mu\r) - \frac{z}{\sqrt{m}}}{u}\geq 0}
\end{multline*}
So, in the course of this proof, it suffices to assume that $X_1, \ldots, X_N \sim \mathcal{E}(0, I_d, F)$ and define, with abuse of notations, 
\begin{multline*}
    U_{N,m}(u,z) = \frac{1}{{N\choose m}} \sum_{J\in \m A_{N}^{(m)} }I\left\{ \dotp{\bar X_J - \frac{z}{\sqrt{m}}}{u}\geq 0\right\} \\
    - \pr{\dotp{\bar X_m  - \frac{z}{\sqrt{m}}}{u}\geq 0} \,.
\end{multline*}
Let
\begin{multline*}
    h_{u}(X_1, \ldots, X_m;z) = I\left\{ \dotp{\bar X_m - \frac{z}{\sqrt{m}}}{u}\geq 0\right\} \\
    - \pr{\dotp{\bar X_m  - \frac{z}{\sqrt{m}}}{u}\geq 0}\,.
\end{multline*}
Hoeffding's decomposition implies
\begin{multline*}
   \sqrt{\frac{N}{m}} U_{N,m}(u,z) = \sqrt{\frac{N}{m}} \sum_{i=1}^N h^{(1)}_{u}\l(X_i;z\r) \\
   + \sum_{j=2}^m \sqrt{\frac{N}{m}}\frac{{m \choose j}}{{N \choose j}} \sum_{J \in \mathcal{I}^N_j} h^{(j)}_{u}\l(X_i,\  i \in J;z\r)\,.
\end{multline*}
Note that by Jensen's inequality and Corollary \ref{corollary:1},
\begin{align*}
    \mb E \sup_{(u,z) \in \bar{\mathcal{S}}}  \frac{{m \choose j}}{{N \choose j}}\l\lvert  \sum_{J \in \mathcal{I}^N_j} h^{(j)}_{u}(X_i,\  i \in \mathcal{I}^N_j;z)\r\rvert   \leq \frac{{m \choose j}^{1/2}}{{N \choose j}^{1/2}}\max\l\{B_1(p,j), B_2(p,j) \r\} \,,
\end{align*}
for all $p \geq 2$, where $B_1(p,j)$ and $B_2(p,j)$ are defined in Corollary \ref{corollary:1}. Moreover, note that all $\l\{h_{u}(\cdot;z): (u,z) \in \bar{\mathcal{S}}\r\}$ form the half planes in $\mb R^d$. So, the VC dimension of $\l\{h_{u}(\cdot;z): (u,z) \in \bar{\mathcal{S}}\r\}$ satisfies $V = d+1$. Furthermore, since every $\l\lvert h_{u}(\cdot;z) \r\rvert $ is bounded by $1$, Lemma \ref{lemma:higher-order} for $p = 2$ along with Jensen's inequality imply that 
\begin{align}\label{eq:THL-remainder-vanish}
    \mb E\sup_{(u,z) \in \bar{\mathcal{S}}} \sqrt{\frac{N}{m}} \frac{{m \choose j}}{{N \choose j}}\l\lvert  \sum_{J \in \mathcal{I}^n_j} h^{(j)}_{u}\l(X_i,\  i \in J;z\r)\r\rvert \rightarrow 0\,,
\end{align}
as $m,N \rightarrow \infty$ if $\frac{m^2d}{N} < c$ for $c$ small enough. Thus, Chebyshev's inequality implies it suffices to consider asymptotic distribution of the process
\begin{align*}
    (u,z) \mapsto \sqrt{\frac{m}{N}} \sum_{i=1}^N h^{(1)}_{u}\l(X_i;z\r)\,.
\end{align*}
Next, we will rewrite $h^{(1)}_{u}$ in a form that is easier to analyze. Note that $h_{u}(X_1, \ldots ,X_m;z)$ has mean $0$. Thus,
    \begin{multline*}
        h^{(1)}_{u} (X_1;z) = (\delta_{X_1} - P) \times P^{m-1} h = \delta_{X_i} \times  P^{m-1} h_{u,z} \\
        = \mb E_{X_{-1}} I\left\{ \dotp{\frac{\sum_{j \in [m]}X_j}{m}-\frac{z}{\sqrt{m}}}{u}\geq 0\right\} \\
        - \pr{\dotp{\bar X_m -\frac{z}{\sqrt{m}}}{u}\geq 0}\,,
    \end{multline*}
    where $X_{-j}$ denotes the vector $(X_i, i \neq j)$ for any $j$. 
    Let $X_1', \ldots, X_{m}'$ be independent copies of $X_1$ and recall that $W_{u,m} = \sum_{j=1}^{m} \frac{\dotp{X'_j}{u}}{\sqrt{m}}$. Remark that $X_1, \ldots, X_m$ have elliptical symmetric distribution. So, the distributions of $W_{u,m}$ and $-W_{u,m}$ are identical. It follows that 
    \begin{multline*}
        \mb E_{X_{-1}} I\left\{ \dotp{\frac{\sum_{j \in [m]}X_j}{m}  -\frac{z}{\sqrt{m}}}{u}\geq 0\right\} \\
        = \mb E_{W_{u,m-1}} I\left\{-\sqrt{\frac{m-1}{m^2}}W_{u,m-1} + \dotp{\frac{X_1}{m}}{u} \geq \dotp{\frac{z}{\sqrt{m}}}{u}\right\} \\
        = F_{m-1}\left(\sqrt{\frac{1}{m-1}}\dotp{X_1}{u} - \sqrt{\frac{m}{m-1}}\dotp{z}{u} \right)\,,
    \end{multline*}
    where $F_{m}$ is the distribution function of $W_{u,m}$. 
    Here, we use the fact that $X_1', \ldots, X_m'$ have elliptically symmetric distribution with identity covariance, and thus the distribution of $W_{u,m}$ is actually independent of $u$ (see section \ref{sec:elliptical-symmetric-distribution} for details). Hence,
\begin{multline}
\label{eq:Phi_uz}
    h^{(1)}_{u}(X_i;z) = F_{m-1}\left(\sqrt{\frac{1}{m-1}}\dotp{X_i}{u} - \sqrt{\frac{m}{m-1}}\dotp{z}{u}\right) \\
    - \mb E F_{m-1}\left(\sqrt{\frac{1}{m-1}}\dotp{X_i}{u} - \sqrt{\frac{m}{m-1}}\dotp{z}{u}\right)\,,
\end{multline}
The rest of the proof is divided into two parts: (1) asymptotic equicontinuity, and (2) central limit theorem for finite dimensional random vectors. We start with part (1). Recall that 
\[
\bar{\mathcal{S}} = \mathcal{S}^{d-1} \times \mathcal{S}^d_q = \{x \in \mb R^d : \|x\| = 1\} \times \{x \in \mb R^d: \|x\| \leq q\}.
\]
Define
\begin{align}
\label{eq:bar-S-delta}
    \bar{\mathcal{S}}_\delta = \l\{((u_1,z_1), (u_2,z_2)) \in \bar{\mathcal{S}} \times \bar{\mathcal{S}} : \|(u_1, z_1) - (u_2, z_2)\| \leq \delta\r\}\,.
\end{align}
We want to bound
\begin{align*}
    \sqrt{\frac{m-1}{N}} \mb E \sup_{((u_1,z_1), (u_2,z_2)) \in \bar{\mathcal{S}}_\delta} \l\lvert  \sum_{i=1}^N h^{(1)}_{u_1}(X_i; z_1) - h^{(1)}_{u_2}(X_i; z_2)\r\rvert\,.
\end{align*}
To this end, we define $f_{u, z, m}(X_i)$ as
\begin{multline}\label{eq:f-function}
    f_{u,z,m}(X_i) = F_{m-1}\l( \frac{\dotp{X_i}{u}}{\sqrt{m-1}} - \sqrt{\frac{m}{m-1}}\dotp{z}{u}\r) \\
    - F_{m-1}\l(  - \sqrt{\frac{m}{m-1}}\dotp{z}{u}\r)\,.
\end{multline}
Symmetrization inequality implies that 
\begin{multline*}
    \sqrt{\frac{m-1}{N}} \mb E \sup_{((u_1,z_1), (u_2,z_2)) \in \bar{\mathcal{S}}_\delta} \l\lvert  \sum_{i=1}^N h^{(1)}_{u_1}(X_i; z_1) - h^{(1)}_{u_2}(X_i; z_2)\r\rvert \\
    \leq C\sqrt{\frac{m-1}{N}} \mb E \sup_{((u_1,z_1), (u_2,z_2)) \in \bar{\mathcal{S}}_\delta} \l\lvert \sum_{i=1}^N \varepsilon_i \l(f_{u_1,z_1,m}(X_i) - f_{u_2,z_2,m}(X_i)\r)\r\rvert \,.
\end{multline*}
The process $(u,z) \mapsto \sqrt{\frac{m}{N}}\sum_{i=1}^N \varepsilon_i f_{u,z,m}(X_i)$ is sub-Gaussian conditionally on $X_1, \ldots, X_N$ with respect to the distance
\begin{align*}
    g^2_N \l((u_1, z_1), (u_2, z_2)\r) = \frac{m}{N} \sum_{i=1}^N \l(f_{u_1,z_1,m}(X_i)- f_{u_2,z_2,m}(X_i)\r)^2\,.
\end{align*}
To proceed, we need the following technical lemma
\begin{lemma}\label{lemma:Delta2}
    Under the conditions of Theorem \ref{thm:normality}, 
    \begin{multline}\label{eq:f-diff}
     \sqrt{m}\l\lvert f_{u_1,z_1,m}(X_i) - f_{u_2,z_2,m}(X_i) \r\rvert \\
    \leq \l[C(q+1)\|X_i\|\r] \|(u_1, z_1) - (u_2, z_2)\|\,.
\end{multline}
\end{lemma}
This lemma implies that for any $\varepsilon > 0$, 
\[
g_N \l((u_1, z_1), (u_2, z_2)\r) \leq \varepsilon\sqrt{\frac{1}{N}\sum_{i=1}^N C(q+1)^2\|X_i\|^2}
\]
whenever
\(
\|(u_1, z_1) - (u_2, z_2)\| \leq \varepsilon.
\)
Thus
\begin{multline*}
    \mb E_\varepsilon \sup_{((u_1,z_1), (u_2,z_2)) \in \bar{\mathcal{S}}_\delta} \l\lvert \sum_{i=1}^N \varepsilon_i \l(f_{u_1,z_1,m}(X_i) - f_{u_2,z_2,}(X_i)\r)\r\rvert \\ \leq \int_0^{\delta\sqrt{N^{-1}\sum_{i=1}^N C(q+1)^2\|X_i\|^2}} \l(\log_+\l(N(\bar{\mathcal{S}}, g_N, t)\r)\r)^{1/2} dt \,.
\end{multline*}
It is easy to check that $N(\bar{\mathcal{S}}, g_N, t) \leq N\l(\bar{\mathcal{S}}, \|\cdot\|, \frac{t}{\sqrt{N^{-1}\sum_{i=1}^N C(q+1)^2\|X_i\|^2}}\r)$, and thus the above integral can be bounded by
\begin{align*}
    \sqrt{N^{-1}\sum_{i=1}^N C(q+1)^2\|X_i\|^2} \int^\delta_0 \l(\log_+\l(N(\bar{\mathcal{S}}, \|\cdot\|, t)\r)\r)^{1/2} dt\,.
\end{align*}
Moreover, it is easy to see that 
\[
N(\bar{\mathcal{S}}, \|\cdot\|, t) \leq N(\mathcal{S}^{d-1}, \|\cdot\|, t/2) \times N(\mathcal{S}^{d}_q, \|\cdot\|, t/2).
\]
The first covering number is bounded by $\l(\frac{C}{t}\r)^{d}$ since $\mathcal{S}^{d-1}$ is the unit sphere in $\mb R^d$, whereas $\mathcal{S}^d_{q}$ is a ball in $\mb R^d$ with radius $q$, which implies the second covering number is also at most $\l(\frac{Cq}{t}\r)^{d}$. Hence, 
\begin{align*}
\int^\delta_0 \l(\log_+\l(N(\bar{\mathcal{S}}, \|\cdot\|, t)\r)\r)^{1/2} dt \leq d^{1/2}\int^\delta_0 \l(\log_+\l(\frac{C_1}{t} \r) + \log_+\l(\frac{C_2q}{t}\r)\r)^{1/2} dt\,. 
\end{align*}
Thus, it suffices to show $\mb E \sqrt{N^{-1}\sum_{i=1}^N C(q+1)^2\|X_i\|^2}$ is bounded. However, by Jensen's inequality, this is bounded by $C(1+q)\sqrt{d}$. To conclude, we have obtained
\begin{multline*}
    \sqrt{\frac{m-1}{N}} \mb E \sup_{((u_1,z_1), (u_2,z_2)) \in \bar{\mathcal{S}}_\delta} \l\lvert  \sum_{i=1}^N h^{(1)}_{u_1}(X_i; z_1) - h^{(1)}_{u_2}(X_i; z_2)\r\rvert \\
    \leq C(1+q)d \int^\delta_0 \l(\log_+\l(\frac{C_1}{t} \r) + \log_+\l(\frac{C_2q}{t}\r)\r)^{1/2} dt
\end{multline*}
Clearly, as $\delta \rightarrow 0$, the term on the right hand side converges to $0$, which validates the asymptotic equicontinuity.

Now, let us focus on part (2), i.e., finite dimensional convergence to Gaussian vectors. We will invoke Cram\'{e}r-Wold theorem. Let $M$ be a natural number and $t = (t_1, t_2, \ldots, t_M)^\top \in \mb R^M$ be an arbitrary non-zero real vector. Moreover, let $\{(u_1, z_1), \ldots, (u_M, z_M)\}$ be a set of parameters. We will show that $\sqrt{\frac{m}{N}}\sum_{j=1}^N H (X_j; (u_i, z_i), i \in [M])$ is asymptotically normal as $m,N$ grow where
\begin{align}\label{eq:H}
    H (X_j; (u_i, z_i), i \in [M]) = \sum_{i = 1}^M t_i h^{(1)}_{u_i} (X_j;z_i)\,.
\end{align}
Given that $\mb E H (X_1; (u_i, z_i), i \in [M]) = 0$, by Lindeberg's condition, the goal is to find a vanishing upper bound for 
\begin{align}
\label{eq:Lindeberg}
    \frac{1}{\mb E \l(\sqrt{m}H(X_1; (u_i, z_i), i \in [M])\r)^2} \mb E \l[\l\lvert\sqrt{m} H(X_1; (u_i, z_i), i \in [M]) \r\rvert^2 I \l\{A_\eps \r\}\r]\,,
\end{align}
where
\begin{multline}\label{eq:A_epsilon}
    A_\eps = \Biggl\{ \l\lvert \sqrt{m} H(X_1; (u_i, z_i), i \in [M])\r\rvert \\
    > \varepsilon \sqrt{N \mb E \l(\sqrt{m}H(X_1; (u_i, z_i), i \in [M])\r)^2 }\Biggr\}
\end{multline}
and $\varepsilon > 0$ is an arbitrary fixed constant. To this end, define
\begin{multline*}
    K(X_1; (u_i, z_i), i \in [M]) \\
    = \sum_{i, j \in [M]} t_i t_{j} \l(e^{-\frac{\l\lvert \dotp{z_i}{u_i}\r\rvert^2}{2}}\frac{\dotp{X_1}{u_i}}{\sqrt{2\pi}}\r)\l(e^{-\frac{\l\lvert \dotp{z_{j}}{u_{j}}\r\rvert^2}{2}}\frac{\dotp{X_1}{u_j}}{\sqrt{2\pi}}\r)\,.
\end{multline*}
Since $X_1$ has covariance matrix $I_d$, we have that
\begin{align*}
    \mb E \l(\dotp{X_1}{u_i}\dotp{X_1}{u_{j}}\r)  = \dotp{u_i}{u_{j}}\,.
\end{align*}
Hence, we conclude that $\mb E K(X_1; (u_i, z_i), i \in [M]) = t^\top \Sigma_M t$ where $\Sigma_M$ is a matrix with entries
\begin{align*}
    \l(\Sigma_M\r)_{i,j} = \exp{-\frac{\l\lvert \dotp{z_{i}}{u_{i}}\r\rvert^2 + \l\lvert \dotp{z_j}{u_j}\r\rvert^2}{2}}\frac{\dotp{u_i}{u_{j}}}{2\pi}
\end{align*}
We will replace $\l\lvert\sqrt{m} H(X_1; (u_i, z_i), i \in [M]) \r\rvert^2$ with $K(X_1; (u_i, z_i), i \in [M])$ that is easier to analyze. This replacement relies on the following lemma. Recall that $\bar{\mathcal{S}} = \mathcal{S}^{d-1} \times \mathcal{S}^d_q = \{x \in \mb R^d : \|x\| = 1\} \times \{x \in \mb R^d: \|x\| \leq q\}$.
\begin{lemma}
\label{lemma: sqrt_l_h_convergence}
    Let $C_F$ be a constant that depends only on the distribution $F$. Recall that $K_{q,m} = \sup_{z \in [-2q,2q]}\l\lvert \phi_{m-1}(z) - \frac{1}{\sqrt{2\pi}}\exp{-\frac{z^2}{2}}\r\rvert$. Under the assumptions of Theorem \ref{thm:normality}, 
    \begin{align}
    \label{eq:variance-hajek-THL}
        \l\lvert \var\l(\sqrt{m} h^{(1)}_{u}\l(X_1;z\r)\r) - \frac{1}{2\pi}\exp{-\l\lvert \dotp{z}{u} \r\rvert^2} \r\rvert \leq C_F\l(K^2_{q,N} +\frac{1}{m} + \frac{q^2}{m}\r)\,.
    \end{align}
    Moreover, for any vector $t \in \mb R^M$ and $(u_1, z_1), \ldots, (u_M, z_M) \in \bar{\mathcal{S}}$,
    \begin{multline}\label{eq:key-L1-convergence}
    \mb E \l\lvert  \l(\sqrt{m} H(X_1; (u_i, z_i), i \in [M])\r)^2 - K(X_1; (u_i, z_i), i \in [M]) \r\rvert \\
    \leq C_F\sum_{i, j \in [M]} t_i t_j \l(K_{q,m}^2 + \frac{1}{m} + \frac{q^2}{m^2}\r)\,.
\end{multline}

\end{lemma}
\noindent In view of Lemmas \ref{lemma: sqrt_l_h_convergence} and \ref{lemma:implication-of-uniform-lipschitz}, it is immediate that
\begin{align*}
    \mb E \l( \sqrt{m} H(X_1; (u_i, z_i), i \in [M]) \r)^2- t^T \Sigma_M t  \rightarrow 0
\end{align*}
for fixed $q$. Hence, convergence in \eqref{eq:Lindeberg} can be reduced to convergence of

\begin{align*}
    \mb E \l[\l\lvert\sqrt{m} H(X_1; (u_i, z_i), i \in [M]) \r\rvert^2 I \l\{A_\varepsilon \r\}\r]\,.
\end{align*}
In turn, this quantity can be estimated from above by
\begin{multline*}
    \mb E \l[ \lvert K(X_1; (u_i, z_i), i \in [M]) \rvert I \l\{A_\varepsilon \r\}\r]\\
    + \mb E \l\lvert \l(\sqrt{m} H(X_1; (u_i, z_i), i \in [M]) \r)^2- K(X_1; (u_i, z_i), i \in [M]) \r\rvert\,.
\end{multline*}
The second term vanishes as $m \rightarrow \infty$ by \eqref{eq:key-L1-convergence}. It remains to show the convergence of the first term. Note that for any $R > 0$,
\begin{multline}\label{eq:upper-bound-for-L1}
    \mb E \l[\lvert K(X_1; (u_i, z_i), i \in [M]) \rvert I \l\{A_\varepsilon \r\}\r] \leq R \pr{A_\varepsilon} \\
    + \mb E \l(\lvert K(X_1; (u_i, z_i), i \in [M]) \rvert I\l\{\lvert K(X_1; (u_i, z_i), i \in [M]) \rvert > R\r\}\r)\,.
\end{multline}
Recall the definition of $A_\varepsilon$ in \eqref{eq:A_epsilon}, Chebyshev's inequality implies $M \pr{A_\varepsilon} \leq \frac{R}{\varepsilon^2 N}$. Moreover, the fact that $\mb E \l\lvert \dotp{X_1}{u}\r\rvert^2  = 1$ implies
\begin{multline*}
    \mb E \l\lvert K(X_1; (u_i, z_i), i \in [M]) \r\rvert \leq C\sum_{i, j \in [M]} t_i t_{j} \mb E\l\lvert \dotp{X_1}{u_i}\dotp{X_1}{u_j}\r\rvert \\
    \leq 
    C\sum_{i, j \in [M]} t_i t_{j} \mb E\l\lvert \dotp{X_1}{u_i}\r\rvert^2\mb E \l\lvert \dotp{X_1}{u_j}\r\rvert^2 < \infty
\end{multline*}
Thus, dominated convergence theorem implies
\begin{align*}
    \mb E \l(\lvert K(X_1; (u_i, z_i), i \in [M]) \rvert I\l\{\lvert K(X_1; (u_i, z_i), i \in [M]) \rvert > R\r\}\r) \rightarrow 0
\end{align*}
as $R \rightarrow \infty$. Selecting $R = N^{1/2}$ implies \eqref{eq:upper-bound-for-L1} converges to $0$ as $m,N$ increase to infinity.

\subsection{Proof of Lemma \ref{lemma:Delta2}}
\label{proof:Delta2}

Note that
\begin{align*}
    f_{u,z,m}(X_i) = \int_{0}^{\frac{\dotp{X_i}{u}}{\sqrt{m-1}}} \phi_{m-1}\l(t - \sqrt{\frac{m}{m-1}}\dotp{z}{u}\r)  dt \,.
\end{align*}
Since Lemma \ref{lemma:uniform-lipschitz} implies $\phi_{m-1}$ is Lipschitz with a constant uniformly for all $m$ large enough, by adding and subtracting $\phi_{m-1}\l(t - \sqrt{\frac{m}{m-1}}\dotp{z_1}{u_2}\r)$, we deduce that
\begin{multline*}
    \l\lvert \phi_{m-1}\l(t - \sqrt{\frac{m}{m-1}}\dotp{z_1}{u_1}\r) - \phi_{m-1}\l(t - \sqrt{\frac{m}{m-1}}\dotp{z_2}{u_2}\r) \r\rvert \\
    \leq C \l\lvert \dotp{z_1 - z_2}{u_1} \r\rvert + C \l\lvert \dotp{z_2}{u_1 - u_2} \r\rvert \,.
\end{multline*}
Since $z_2 \in \mathcal{S}^d_q$, we conclude that the last term can be bounded by $C_1 \|z_1 - z_2\| + C_2 q \|u_1 - u_2\|$. Therefore,
\begin{multline*}
    \l\lvert f_{u_1,z_1,m}(X_i) - f_{u_2,z_2,m}(X_i) \r\rvert \\
    \leq \l\lvert \int^{\frac{\dotp{X_i}{u_1}}{\sqrt{m-1}}}_{\frac{\dotp{X_i}{u_2}}{\sqrt{m-1}}} \phi_{m-1}\l(t - \sqrt{\frac{m}{m-1}}\dotp{z_1}{u_1}\r) dt \r\rvert \\
    + \l\lvert \int_0^{\frac{\dotp{X_i}{u_2}}{\sqrt{m-1}}} C_1 \|z_1 - z_2\| + C_2 q \|u_1 - u_2\| dt \r\rvert \,.
\end{multline*}
Since $\phi_{m-1}$ is bounded uniformly for all large $m$ by Lemma \ref{lemma:implication-of-uniform-lipschitz}, the first integral on the right is bounded by $\frac{C}{\sqrt{m-1}}\l\lvert \dotp{X_i}{u_1-u_2} \r\rvert \leq \frac{C}{\sqrt{m-1}}\l\| X_i\r\| \|u_1 - u_2\|$. The second integral on the right is bounded by $\frac{\l\| X_i\r\|}{\sqrt{m-1}}\l(C_1 \|z_1 - z_2\| + C_2 q \|u_1 - u_2\|\r)$. Hence
\begin{equation*}
     \sqrt{m}\l\lvert f_{u_1,z_1,m}(X_i) - f_{u_2,z_2,m}(X_i) \r\rvert 
     \leq C_1(1+q)\|X_i\| \|u_1 - u_2\| + C_2\|X_1\| \|z_1 - z_2\| \,,
\end{equation*}
which implies \eqref{eq:f-diff}.

\subsection{Proof of Lemma \ref{lemma: sqrt_l_h_convergence}}

In order to establish \eqref{eq:variance-hajek-THL} and \eqref{eq:key-L1-convergence}, a key intermediate step consists in controlling the magnitude of
    \[
    \sqrt{m}h^{(1)}_{u}(X_1;z) - \frac{1}{\sqrt{2\pi}}\exp{-\frac{\l\lvert \dotp{z}{u}\r\rvert^2}{2}}\dotp{X_i}{u}
    \]
    as $m \rightarrow \infty$ for any fixed $(u,z) \in \bar{\mathcal{S}}$. Recall the definition of $h^{(1)}_{u}$ in \eqref{eq:Phi_uz}. We will focus on the quantity
    \begin{multline*}
        \sqrt{m}\Biggl[F_{m-1} \left(\sqrt{\frac{1}{m-1}}\dotp{X_1}{u} - \sqrt{\frac{m}{m-1}}\dotp{z}{u}\right) \\
        - \mb E F_{m-1} \left(\sqrt{\frac{1}{m-1}}\dotp{X_1}{u} - \sqrt{\frac{m}{m-1}}\dotp{z}{u}\right)\Biggr] \\
        - \frac{1}{\sqrt{2\pi}}\exp{-\frac{\l\lvert \dotp{z}{u}\r\rvert^2}{2}}\dotp{X_i}{u} \,.
    \end{multline*}
    Because $\|u\| = 1$, it is immediate that
    \begin{align}\label{eq:2-moment-V1}
        \mb E \l\lvert \dotp{X_1}{u}\r\rvert^2  = 1\,,
    \end{align}
    which further implies
    \begin{multline*}
        \mb E\l[\phi_{m-1}\l(-\dotp{z}{u}\r)\dotp{X_1}{u}-\frac{1}{\sqrt{2\pi}}\exp{-\frac{\l\lvert \dotp{z}{u}\r\rvert^2}{2}}\dotp{X_1}{u}\r]^2 \\
        \leq K^2_{q,m} \mb E\l(\dotp{X_1}{u}\r)^2 = K^2_{q,m}\,,
    \end{multline*}
    where $K_{q,m} = \sup_{z \in [-2q,2q]}\l\lvert \phi_{m-1}(z) - \frac{1}{\sqrt{2\pi}}\exp{-\frac{z^2}{2}}\r\rvert$. Next, since $\phi_{m-1}$ is Lipschitz continuous with a uniformly bounded Lipschitz constant for large $m$ (this follows from Lemma \ref{lemma:uniform-lipschitz}), it is easy to see that
    \begin{multline*}
        \l\lvert \phi_{m-1}\l(-\dotp{z}{u}\r) - \phi_{m-1}\l(-\sqrt{\frac{m}{m-1}}\dotp{z}{u}\r) \r\rvert  \\
        \leq C_F \lvert \dotp{z}{u}\rvert  \l( \sqrt{\frac{m}{m-1}}  - 1\r)\leq \frac{qC_F}{m}\,,
    \end{multline*}
    where we use the fact that $z \in \bar{\mathcal{S}}_q := \{x \in \mb R^d: \|x\| \leq q\}$ and $u \in \mathcal{S}^{d-1} = \l\{x \in \mb R^d : \|x\| = 1\r\}$ in the last inequality. Hence,
    \begin{align*}
        \mb E \l[\phi_{m-1}\l(-\dotp{z}{u}\r)\dotp{X_1}{u} - \phi_{m-1}\l(-\sqrt{\frac{m}{m-1}}\dotp{z}{u}\r)\dotp{X_1}{u}\r]^2 \leq \frac{q^2C_F}{m^2}
    \end{align*}
    Moreover, uniform boundedness of $\phi_{m-1}$ by Lemma \ref{lemma:implication-of-uniform-lipschitz} implies that 
    \begin{multline*}
        \mb E \Biggl[\phi_{m-1}\l(-\sqrt{\frac{m}{m-1}}\dotp{z}{u}\r)\dotp{X_1}{u} \\
        - \sqrt{\frac{m}{m-1}}\phi_{m-1}\l(-\sqrt{\frac{m}{m-1}}\dotp{z}{u}\r)\dotp{X_1}{u}\Biggr]^2 \leq \frac{C_F}{m^2}\,.
    \end{multline*}
    Therefore,
    \begin{multline*}
        \mb E\l[\sqrt{m}h^{(1)}_{u}(X_1;z) - \frac{1}{\sqrt{2\pi}}\exp{-\frac{\l\lvert \dotp{z}{u}\r\rvert^2}{2}}\dotp{X_1}{u}\r]^2 \\
        \leq K_{q,m}^2 + \frac{C_F(q^2+1)}{m^2} + \mb E\l(\sqrt{m} f_{u,z,m}(X_1)\r)^2 \,,
    \end{multline*}
    where
    \begin{multline*}
        f_{u,z,m}(X_1) = \Biggl[F_{m-1}\left(\frac{\dotp{X_1}{u}}{\sqrt{m-1}} - \sqrt{\frac{m}{m-1}}\dotp{z}{u}\right) \\
        - \mb E F_{m-1}\left(\frac{\dotp{X_1}{u}}{\sqrt{m-1}} - \sqrt{\frac{m}{m-1}}\dotp{z}{u}\right)\Biggr] \\
        -  \phi_{m-1}\l(- \sqrt{\frac{m}{m-1}}\dotp{z}{u}\r)\frac{\dotp{X_1}{u}}{\sqrt{m-1}}\,.
    \end{multline*}
    Given that $\phi_{m-1}\l(- \sqrt{\frac{m}{m-1}}\dotp{z}{u}\r)\frac{\dotp{X_1}{u}}{\sqrt{m-1}}$ has mean $0$, adding and subtracting the deterministic $F_{m-1}\l(- \sqrt{\frac{m}{m-1}}\dotp{z}{u}\r)$, it suffices to focus on
    \begin{multline}\label{eq:L2-intermediate}
        \sqrt{m}\l[F_{m-1}\left(\frac{\dotp{X_1}{u}}{\sqrt{m-1}} - \sqrt{\frac{m}{m-1}}\dotp{z}{u}\right) - F_{m-1}\left(- \sqrt{\frac{m}{m-1}}\dotp{z}{u}\right)\r]\\
        - \sqrt{\frac{m}{m-1}}\phi_{m-1}\l(- \sqrt{\frac{m}{m-1}}\dotp{z}{u}\r)\dotp{X_1}{u}  \,.
    \end{multline}
    Note that
    \begin{multline*}
    \sqrt{m}\l[F_{m-1}\left(\frac{\dotp{X_1}{u}}{\sqrt{m-1}} - \sqrt{\frac{m}{m-1}}\dotp{z}{u}\right) - F_{m-1}\left(- \sqrt{\frac{m}{m-1}}\dotp{z}{u}\right)\r] \\
    = \sqrt{\frac{m}{m-1}} \phi_{m-1}\l(\frac{t\dotp{X_1}{u}}{\sqrt{m-1}} - \sqrt{\frac{m}{m-1}}\dotp{z}{u}\r) \dotp{X_1}{u}  \,,
    \end{multline*}
    where the equation is implied by Taylor expansion for some $t \in (0,1)$. Thus, defining $g(x,y) = \phi_{m-1}(x-y) - \phi_{m-1}(-y)$, it follows that \eqref{eq:L2-intermediate} can be bounded by $Cg\l(\frac{t\dotp{X_1}{u}}{\sqrt{m-1}}, - \sqrt{\frac{m}{m-1}}\dotp{z}{u}\r) \dotp{X_1}{u}$. Next, Cauchy-Schwarz inequality along with \eqref{eq:2-moment-V1} implies that
    \begin{multline*}
        \mb E \l(g\l(\frac{t\dotp{X_1}{u}}{\sqrt{m-1}}, - \sqrt{\frac{m}{m-1}}\dotp{z}{u}\r) \dotp{X_1}{u}\r)^2 \\
        \leq \mb E \l(g\l(\frac{t\dotp{X_1}{u}}{\sqrt{m-1}}, - \sqrt{\frac{m}{m-1}}\dotp{z}{u}\r)\r)^2 \\
        \leq \frac{C_F t^2 \mb E \l(\dotp{X_1}{u}\r)^2}{m-1} \leq \frac{C_F}{m}\,.
    \end{multline*}
    Collecting all estimates yields
    \begin{multline}\label{eq:convergence-L2}
        \mb E\l[\sqrt{m}h^{(1)}_{u}(X_1;z) - \frac{1}{\sqrt{2\pi}}\exp{-\frac{\l\lvert \dotp{z}{u}\r\rvert^2}{2}}\dotp{X_i}{u}\r]^2 \\
        \leq K_{q,m}^2 + \frac{C_F}{m} + \frac{q^2C_F'}{m^2}\,.
    \end{multline}
    The display above immediately yields inequality \eqref{eq:variance-hajek-THL}. Next, let us establish relation \eqref{eq:key-L1-convergence}. Recall that 
\begin{align*}
    \l(\sqrt{m} H^{(1)}(X_1)\r)^2 =  \sum_{i, j \in [M]} t_i t_{j} \l(\sqrt{m} h^{(1)}_{u_i} (X_1; z_i)\r)\l(\sqrt{m} h^{(1)}_{u_{j}} (X_1;z_j)\r)\,,
\end{align*}
by the definition of $H(X_1; (u_i, z_i), i \in [M])$ in \eqref{eq:H} and that
\begin{multline*}
    K(X_1; (u_i, z_i), i \in [M]) \\
    = \sum_{i, j \in [M]} t_i t_{j} \l(e^{-\frac{\l\lvert \dotp{z_i}{u_i}\r\rvert^2}{2}}\frac{\dotp{X_1}{u_i}}{\sqrt{2\pi}}\r)\l(e^{-\frac{\l\lvert \dotp{z_{j}}{u_{j}}\r\rvert^2}{2}}\frac{\dotp{X_1}{u_j}}{\sqrt{2\pi}}\r)\,.
\end{multline*}
It implies that
\begin{multline}\label{eq:quadratic-form-with-G}
    \mb E \l\lvert  \l(\sqrt{m} H(X_1; (u_i, z_i), i \in [M])\r)^2 - K(X_1; (u_i, z_i), i \in [M]) \r\rvert \\
    \leq \sum_{i, j \in [M]} t_i t_{j} \mb E \l\lvert G_{(u,z), (u',z')} \r\rvert\,,
\end{multline}
where
    \begin{multline*}
    G_{(u,z), (u',z')} = \l(\sqrt{m} h^{(1)}_{u} (X_1;z)\r)\l(\sqrt{m} h^{(1)}_{u'} (X_1;z')\r) \\
    - \l(e^{-\frac{\l\lvert \dotp{z}{u}\r\rvert^2}{2}}\frac{\dotp{X_1}{u}}{\sqrt{2\pi}}\r)\l(e^{-\frac{\l\lvert \dotp{z'}{u'}\r\rvert^2}{2}}\frac{\dotp{X_1}{u'}}{\sqrt{2\pi}}\r)\,.
    \end{multline*} 
Thus, it suffices to find an upper bound for $\mb E \l\lvert G_{(u,z), (u',z')} \r\rvert$ that holds uniformly over all $(u,z), (u',z') \in \bar{\mathcal{S}}$. Note that
    \begin{multline}\label{eq:G_term}
        G_{(u,z), (u',z')} = \sqrt{m}h^{(1)}_{u}(X_1;z) \left( \sqrt{m} h^{(1)}_{u'}(X_1;z') - \frac{1}{\sqrt{2\pi}}e^{-\frac{\l\lvert \dotp{z'}{u'}\r\rvert^2}{2}}\dotp{X_1}{u'} \right) \\
        + \left( \sqrt{m} h^{(1)}_{u}(X_1;z) - \frac{1}{\sqrt{2\pi}}e^{-\frac{\l\lvert \dotp{z}{u}\r\rvert^2}{2}}\dotp{X_1}{u} \right) \frac{1}{\sqrt{2\pi}}e^{-\frac{\l\lvert \dotp{z'}{u'}\r\rvert^2}{2}}\dotp{X_1}{u'}\,.
    \end{multline}
    We will estimate both terms on the right separately. For the first term, Cauchy-Schwarz inequality and inequality \eqref{eq:convergence-L2} imply that
    \begin{multline*}
        \mb E\l\lvert \sqrt{m}h^{(1)}_{u}(X_1;z) \left( \sqrt{m} h^{(1)}_{u'}(X_1;z') - \frac{1}{\sqrt{2\pi}}e^{-\frac{\l\lvert \dotp{z'}{u'}\r\rvert^2}{2}}\dotp{X_1}{u'} \right)\r\rvert \\
        \leq  C_F\l(K_{q,m}^2 + \frac{1}{m} + \frac{q^2}{m^2}\r)\,,
    \end{multline*}
    provided that $\mb E\l(\sqrt{m}h^{(1)}_{u}(X_1;z)\r)^2$ is bounded uniformly over all $m,N$ and $(u,z)$. To validate this, recall the representation of $h^{(1)}_{u}$ in \eqref{eq:Phi_uz}. It is easy to see that
\begin{multline*}
    \mb E\l[\sqrt{m}h^{(1)}_{u}\l(X_1;z\r)\r]^2 = \var\l(\sqrt{m} F_{m-1}\left(\frac{\dotp{X_1}{u}}{\sqrt{m-1}} - \sqrt{\frac{m}{m-1}}\dotp{z}{u}\right)\r) \\
    = \frac{m}{2} \mb E\Biggl[F_{m-1}\left(\frac{\dotp{X_1}{u}}{\sqrt{m-1}} - \sqrt{\frac{m}{m-1}}\dotp{z}{u}\right) \\
    - F_{m-1}\left(\frac{\dotp{X_2}{u}}{\sqrt{m-1}} - \sqrt{\frac{m}{m-1}}\dotp{z}{u}\right)\Biggr]^2.
\end{multline*}
Next, Lemma \ref{lemma:implication-of-uniform-lipschitz} implies that $F_{m-1}$ is Lipschitz continuous with constant uniformly bounded over all sufficiently large $N$, hence
\begin{multline*}
    \frac{m}{2} \mb E\Biggl[F_{m-1}\left(\frac{\dotp{X_1}{u}}{\sqrt{m-1}} - \sqrt{\frac{m}{m-1}}\dotp{z}{u}\right) \\
    - F_{m-1}\left(\frac{\dotp{X_1}{u}}{\sqrt{m-1}} - \sqrt{\frac{m}{m-1}}\dotp{z}{u}\right)\Biggr]^2 \\
    \leq \frac{C_{F}m}{2} \mb E\l(\frac{\dotp{X_1 - X_2}{u}}{\sqrt{m-1}}\r)^2 \leq C'_F\var\l(\dotp{X_1}{u}\r)\,,
\end{multline*}
where, recalling \eqref{eq:2-moment-V1}, $\var\l(\dotp{X_1}{u}\r)=1$. This concludes the estimation of the first term on the right of \eqref{eq:G_term}. The bound for second term follows in a similar fashion. Therefore,
    \begin{align*}
        \mb E \lvert G_{(u,z), (u',z')}\rvert \leq C_F\l(K_{q,m}^2 + \frac{1}{m} + \frac{q^2}{m^2}\r)\,,
    \end{align*}
    for all $(u,z), (u',z') \in \bar{\mathcal{S}}$. Combining this inequality with display  \eqref{eq:quadratic-form-with-G} yields the result.

\subsection{Technical background and auxiliary results}
\label{section:tech-results}

\begin{lemma}
\label{lemma:tail_prob}
Let $\{X_t, \ t\in T\}$ be the homogeneous Rademacher chaos of order $j$. Then for any $u > 0$,
\begin{align*}
    \mb P\l(|X_t| \geq ue\|X_t\|_2\r) \leq e^{-u^{2/j}}\,.
\end{align*}
\end{lemma}

\begin{proof}
Let $q\geq 1$ and $c > 0$ be an arbitrary number. Then, by Markov's inequality,
\begin{align*}
    \mb P\l(|X_t| \geq c\|X_t\|_2\r) = \mb P\l(|X_t|^q \geq c^q\|X_t\|_2^q\r) \leq c^{-q}\|X_t\|^q_q \|X_t\|^{-q}_2\,.
\end{align*}
Bonami's inequality \citep[][Theorem 3.2.2]{de2012decoupling} implies that $\|X_t\|_q^q \leq (q-1)^{qj/2} \|X_t\|_2^q$. Thus, 
\begin{align*}
    \mb P\l(|X_t| \geq c\|X_t\|_2\r) \leq c^{-q}(q-1)^{qj/2}\,.
\end{align*}
Setting $q = 1 + (c/e)^{2/j}$ and $c=ue$ in this inequality yields the result.
\end{proof}

\begin{lemma}\label{lemma:compute_expectation}
    Fix $1 \leq p < \infty$, $0 < \alpha, l < \infty$. Let $\gamma > 0$ and suppose $\xi$ is a positive random variable such that for some $c, u_* > 0$,
    \begin{align*}
        \mb P(\xi > \gamma u) \leq c e^{-2^{l-1}u^\alpha} \,,
    \end{align*}
    for every $u \geq u_*$. Then,
    \begin{align*}
        \l(\mb E  \xi^p \r)^{1/p} \leq \gamma\l(C^{1/p}\l(\frac{p}{\alpha}\r)^{1/\alpha + 1/(2p)}2^{-(l-1)/\alpha - 1/(2p)} + u_*\r) \,,
    \end{align*}
    where $C$ is some absolute constant.
\end{lemma}
\begin{proof}
See the proof of Lemma A.5 in \cite{dirksen2015tail}. 
\end{proof}

\begin{lemma}
\label{lemma:incomplete-gamma-bound}
    For all $a \geq 1$ and $x \geq 2a$,
    \begin{align*}
        \int_x^\infty t^ae^{-t} dt \leq 2x^ae^{-x}\,.
    \end{align*}
\end{lemma}

\begin{proof}
    Integration by parts yields that
    \begin{align*}
        \int_x^\infty t^ae^{-t} dt = x^ae^{-x} + a \int_x^\infty t^{a-1} e^{-t} dt\,.
    \end{align*}
    Moreover, $\int_x^\infty t^{a-1} e^{-t} dt \leq \frac{1}{x} \int_x^\infty t^a e^{-t} dt$. Whenever $a/x \leq 1/2$, the claim follows.
\end{proof}

\begin{lemma}\label{lemma:uniform-lipschitz}
    Let $X_1, X_2 \ldots$ be i.i.d. copies of a random variable $X \in \mb R^d$ with $\mb E X = 0$ and $\cov (X) = I_d$. Moreover, assume that $X$ has elliptically symmetric distribution with characteristic generator $\psi : \mb R_+ \mapsto \mb R$. Given $u \in \mathcal{S}^{d-1}$, let $Z_n(u) = \dotp{\frac{\sum_{j=1}^n X_j}{\sqrt{n}}}{u}$. Assume that there exists $\alpha > 0$ and $M > 0$ such that $\psi(t^2) = o\l(\frac{1}{t^\alpha}\r)$ whenever $t > M$. Then for $n_0 \in \mb N$ large enough and all $n \geq n_0$, the density of $Z_n(u)$ exists and is Lipschitz continuous with a uniform Lipschitz constant for all $n$ sufficiently large. 
\end{lemma}

\begin{proof}
    Define $Y_j(u) = \dotp{X_j}{u}$. As stated in Section \ref{sec:elliptical-symmetric-distribution}, the distribution of $Y_1(u)$, hence its characteristic function $\phi(t)$, does not depend on $u$. Moreover, $\phi(t) = \psi(t^2)$ for all $t \in \mb R$. Next, the characteristic function of $\sum_{j=1}^n Y_j$ equals $\l(\phi(t)\r)^n$. Thus, the density of $\sum_{j=1}^n Y_j$ at $x \in \mb R$ exists if $\int_{\mb R} \l| e^{-itx} \l(\phi(t)\r)^n \r| dt < \infty$. However, this is true for all $x$ since
    \begin{align*}
        \int_{\mb R} \l\lvert e^{-itx} \l(\phi(t)\r)^n \r\rvert dt \leq \int_{\mb R} \lvert \phi(t) \rvert^n \leq 2M + \int_{\lvert t \rvert > M} \frac{C}{\lvert t\rvert^{n\alpha}} dt\,,
    \end{align*}
    which is finite for $n > \lfloor\frac{1}{\alpha}\rfloor+1$.
    
    Next, let $f_n$ denote the density of $Z_n(u)$. Using the Fourier inversion formula and the dominated convergence theorem, we deduce that
    \begin{align*}
        \l\lvert \frac{d f_n(x)}{dx}\r\rvert \leq \frac{1}{2\pi}\int_0^\infty \lvert t\rvert \l\lvert \phi\l(\frac{t}{\sqrt{n}}\r) \r\rvert^n dt = \frac{n}{2\pi} \int_0^\infty |t| \l\lvert \phi\l(t\r) \r\rvert^n dt
    \end{align*}
    for all $x \in \mb R$. Thus, establishing a finite upper bound on $\sup_{n \geq n_0}n \int_0^\infty |t| \l\lvert \phi\l(t\r) \r\rvert^n dt$ suffices to complete the proof. To this end, we divide the integral into three parts: the integrals over $A_1 = \{t: \lvert t\rvert < \delta\}$, $A_2 = \{t: \delta \leq \lvert t\rvert \leq K\}$, and $A_3 = \{t : \lvert t\rvert > K\}$ where $\delta > 0$ is a sufficiently small constant and $K > M$ is a sufficiently large constant. Since $Y_1$ has symmetric distribution, $\phi(t) = \mb E \cos(tY_1)$. Using the inequality 
    \[
    0 < \cos(x) \leq 1 - \frac{x^2}{4} 
    \]
    that holds for $x \in (-1, 1)$, we deduce that for $t \in (-1, 1)$,
    \begin{multline*}
        \lvert\phi(t) \rvert \leq \mb E\l[\l(1 - \frac{(tY_1)^2}{4} \r)I\l\{\lvert tY_1\rvert \leq 1\r\} \r] + \mb E\l[I\l\{\lvert tY_1\rvert > 1\r\}\r] \\
        \leq 1 - \frac{t^2 \mb E\l(Y_1^2 I\{\lvert tY_1\rvert \leq 1\}\r)}{4}.
    \end{multline*}
    Since $\mb E Y_1^2 = \mb E \l\lvert \dotp{X_1}{u} \r\rvert^2 = 1$, the dominated convergence theorem implies that there exists a $\delta > 0$ such that $\mb E\l(Y_1^2 I\{\lvert tY_1\rvert \leq 1\}\r) > \frac{1}{2}$ for all $t \in (-\delta, \delta)$, which in turn implies that $\lvert\phi(t) \rvert \leq 1 - \frac{t^2}{8}$. Using the elementary inequality $1 - x \leq e^{-x}$, we conclude that$\lvert\phi(t) \rvert \leq e^{-t^2/8}$. Therefore,
    \begin{align*}
        n \int_{A_1} |t| \l\lvert \phi\l(t\r) \r\rvert^n dt \leq 2n \int_0^\delta  t e^{-\frac{nt^2}{8}} dt = C \l(1 - e^{-cn\delta^2}\r),
    \end{align*}
    Next, since $Y_1$ has absolutely continuous distribution, $\lvert \phi(t) \rvert < 1$ for all $t \in A_2$. Moreover, because $A_2$ is compact and $\phi$ is continuous, $\sup_{t \in A_2} \lvert \phi(t) \rvert = A$ for some $A < 1$, hence
    \begin{align*}
        n \int_{A_2} |t| \l\lvert \phi\l(t\r) \r\rvert^n dt \leq 2n \int_{\delta}^K A^n t dt = C_K n A^n \to 0 \text{ as } n\to\infty.
    \end{align*}
    Therefore, $n \int_{A_2} |t| \l\lvert \phi\l(t\r) \r\rvert^n dt$ is uniformly bounded for large $n$. Finally, Riemann-Lesbegue lemma implies $\lvert\phi(t) \rvert \rightarrow0$ as $\lvert t\rvert \rightarrow \infty$ and we can choose $K>0$ such that $\lvert \phi(t) \rvert \leq \frac{1}{2}$ whenever $\lvert t\rvert > K$, implying that
    \begin{align*}
        n \int_{A_3} |t| \l\lvert \phi\l(t\r) \r\rvert^n dt \leq Cn \l(\frac{1}{2}\r)^{n - k} \int_K^\infty \frac{1}{t^{k\alpha - 1}} dt\,.
    \end{align*}
    Here, we used the fact that $\psi(t) = o\l(\frac{1}{t^\alpha}\r)$ for $t > K$. Selecting $k$ to be an integer such that $k\alpha > 2$ yields $n \l(\frac{1}{2}\r)^{n - k} \int_K^\infty \frac{1}{t^{k\alpha - 1}} dt \rightarrow 0$ as $n \rightarrow \infty$. Thus, the integral is uniformly bounded for $n$ large enough.
\end{proof}

\begin{lemma}
\label{lemma:implication-of-uniform-lipschitz}
Let $X_1, X_2,\ldots$ be i.i.d. copies of a random variable $X \in \mb R$ with absolutely continuous distribution. Assume that $\mb E X= 0$ and $\var(X) = 1$. Define $Y_n = \frac{1}{\sqrt{n}}\sum_{j=1}^n X_j$, and let the density of $Y_n$ be denoted $f_n$. If $f_n$ is Lipschitz continuous with Lipschitz constant $L$ uniformly for all $n \geq n_0$, then $f_n(x) \rightarrow \phi(x)$ uniformly over all $x \in \mb R$ as $n \rightarrow \infty$, where $\phi$ is the density of the standard normal law. Consequently, $f_n$ is also uniformly bounded.
\end{lemma}

\begin{proof}
    For any $x \in \mb R$ and $h > 0$,
    \begin{multline*}
        \l\lvert \frac{1}{h}\mb P\l(x \leq Y_n \leq x + h\r) - f_n(x) \r\rvert \leq \frac{1}{h}\int_0^{h} \lvert f_n(x + t) - f_n(x) \rvert dt \\
        \leq \frac{1}{h} \int_0^h L t dt = \frac{Lh}{2}\,.
    \end{multline*}
    Let $W$ be a standard normal random variable. A similar inequality can be deduced. That is,
    \begin{align*}
         \l\lvert \frac{1}{h}\mb P\l(x \leq W \leq x + h\r) - \phi(x) \r\rvert \leq  \frac{L_\phi h}{2}\,,
    \end{align*}
    where $L_\phi$ is the Lipschitz constant for standard normal density $\phi$. Thus,
    \begin{multline*}
        \l\lvert f_n(x) - \phi(x)\r\rvert \\
        \leq \l\lvert \frac{1}{h}\l(\mb P\l(x \leq Y_n \leq x + h\r) - \mb P(x \leq W \leq x+h)\r)\r\rvert + \frac{(L + L_\phi)h}{2}  \\
        \leq \frac{(L + L_\phi)h}{2} + \frac{2\Delta_n}{h}\,,
    \end{multline*}
    where $\Delta_n = \sup_z \l\lvert \mb P(Y_n \leq z) - \mb P(W \leq z)\r\rvert$. Polya's theorem (see exercise 3.2.9 in \cite{durrett2019probability}) implies that $\Delta_n \rightarrow 0$. Thus, taking $h = h(n) := \sqrt{\Delta_n}$ yields the result.
\end{proof}

\bibliography{u-stats}
\bibliographystyle{apalike}

\end{document}